\documentclass[11pt]{amsart}
\usepackage{amsfonts,amssymb,verbatim,amsmath,amsthm,latexsym,textcomp,amscd,eucal}
\usepackage{latexsym,amsfonts,amssymb,epsfig,verbatim}
\usepackage{amsmath,amsthm,amssymb,latexsym,graphics,textcomp, hyperref, enumerate}
\usepackage[capitalize]{cleveref}
\usepackage{comment}
\usepackage[utf8]{inputenc}
\usepackage[dvipsnames]{xcolor}
\usepackage[export]{adjustbox}
\usepackage{csquotes}
\title[Thurston geometry of Teichmüller space]{The infinitesimal and global Thurston geometry of Teichm\"uller space}
\author{Yi Huang, Ken'ichi Ohshika and Athanase Papadopoulos}
\address{Yi Huang, Yau Mathematical Sciences Center,
Tsinghua University, Haidian District
Beijing 100084, China.
email: {\rm yihuangmath@tsinghua.edu.cn}}
\address{Ken'ichi Ohshika,    
Department of Mathematics,
Gakushuin University,
Mejiro, Toshima-ku, Tokyo, Japan.
email: {\rm ohshika@math.gakushuin.ac.jp}}
\address{Athanase Papadopoulos, Institut de Recherche Mathématique Avancée
(Université de Strasbourg et CNRS),
7 rue René Descartes,
67084 Strasbourg Cedex France.
email: {\rm papadop@math.unistra.fr}}

\newcommand{\reals}{\mathbb{R}}

\newcommand{\naturals}{\mathbb{N}}
\newcommand{\codim}{\operatorname{codim}}

\newcommand{\hyperbolic}{\mathbb{H}}

\newcommand{\ie}{i.e.\ }

\newcommand{\len}{\ell}
\newcommand{\ml}{\mathcal{ML}}
\newcommand{\pml}{\mathcal{PML}}

\newcommand{\fdim}{\mathrm{Fdim}}
\newcommand{\adim}{\mathrm{Adim}}

\newcommand{\teich}{\mathcal{T}}

\definecolor{NoteColor}{rgb}{1,0,0}

\newtheorem{theorem}{Theorem}[section]
\newtheorem{proposition}[theorem]{Proposition}
\newtheorem{lemma}[theorem]{Lemma}
\newtheorem{corollary}[theorem]{Corollary}

\theoremstyle{definition}
\newtheorem{definition}[theorem]{Definition}
\newtheorem{remark}[theorem]{Remark}
\newtheorem{question}[theorem]{Question}
\newtheorem{conjecture}[theorem]{Conjecture}

\includecomment{comments}
\specialcomment{comments}{\color{red}}{\color{black}}
 \begin{document}
\maketitle

\begin{abstract}
We undertake a systematic study of the infinitesimal geometry of the Thurston metric on Teichm\"uller space, showing that the topology, convex geometry and metric geometry of the tangent and cotangent spheres based at any marked hyperbolic surface representing a point in Teichm\"uller space can recover the marking and geometry of this marked surface. We then translate the results concerning the  infinitesimal structures to global geometric statements for the Thurston metric, most notably deriving rigidity statements for the Thurston metric analogous to the celebrated Royden theorem.

\medskip

\end{abstract}

\sloppy

\section{Introduction}

An important and beautiful legacy of S. S. Chern's {\oe}uvre in differential geometry is the idea that many \emph{global} geometric and topological properties of a manifold may be inferred from its \emph{infinitesimal} structure. Our goal is to investigate Thurston's Finsler metric geometry on Teichm\"uller space along this theme. Throughout this paper, 
\begin{itemize}
\item
$S=S_{g,n}$ denotes an orientable hyperbolic surface of genus $g\geq 0$, and with $n\geq0$ cusps (but without boundary);
\item
  $\teich(S)$ denotes the Teichm\"uller space of $S$, \ie  the space of isotopy classes of complete finite-area hyperbolic metrics on $S$. 
\end{itemize}
Here are some highlights of our work.\medskip

\noindent\textsc{Structured stratification}: For any $x$ in $\teich(S)$, we establish and study stratifications of the Thurston Finsler unit norm ball in $T_x\teich(S)$ and of the unit co-norm ball in $T^*_x\teich(S)$ induced by their convex geometry. These stratifications are highly heterogeneous, with polytope-like strata, many of which are naturally labelled by geodesic laminations in a functorially structured manner. They are used in our proof of the infinitesimal rigidity results for the Thurston metric. 
\medskip

\noindent\textsc{Length extraction}: We obtain a limiting procedure for recovering any hyperbolic metric $x\in\teich(S)$ from the Thurston Finsler norm on the tangent space $T_x\teich(S)$ at $x$.\medskip

\noindent\textsc{Infinitesimal rigidity}: We prove an infinitesimal rigidity theorem for the Thurston metric which is a stronger form of an analogue of the celebrated Royden infinitesimal rigidity theorem that holds for the Teichm\"uller metric. We show that isometries of the Thurston Finsler norm $\|\cdot\|_{\mathrm{Th}}$ on $T_x\teich(S)$ are precisely given by extended mapping classes on $S$. Stated precisely,
\begin{theorem}[Infinitesimal rigidity]
\label{main}
Consider two arbitrary points $x,y\in\teich(S)$. The normed vector spaces 
\[
(T_x \teich(S),\|\cdot\|_{\mathrm{Th}})\text{ and }(T_y \teich(S),\|\cdot\|_{\mathrm{Th}})\text{ are isometric }
\] 
if and only if the hyperbolic surfaces $(S,x)\text{ and }(S,y)$ are isometric. 
\end{theorem}

This project grew out of our attempt to prove \cref{main}, and the second author presented this earlier work in a series of lectures given at a CIMPA school in Varanasi in December 2019 \cite{O}. In the course of preparing the present paper, we learned of Pan's independent work \cite{P} establishing \cref{main}. Our respective proof strategies and the flavour of our work differ substantially: Pan's approach makes beautiful use of \cite[Theorem~3.1 and Theorem~10.1]{ThS} to efficiently establish infinitesimal rigidity. In contrast, in part informed by conversations with David Dumas and Kasra Rafi, we opt to follow Thurston's philosophy \cite{ThS}:
\begin{displayquote}
\textit{In the course of this paper, we will develop some other ideas which are of interest in their own right. The intent is not to give the slickest proof of the main theorem, but to develop a good picture.}
\medskip
\end{displayquote}
We now present an overview of the picture we have learned.

\subsection{The Thurston metric}

Let  $\mathcal{S}= \mathcal{S}(S)$ denote the set of free homotopy classes of essential simple closed curves on $S$, that is, simple closed curves which are neither homotopic to a point nor to a cusp. In \cite{ThS}, Thurston introduces the following quantity, defined for any \emph{ordered} pair $x,y\in\teich(S)$ 
\begin{align}
d_{\mathrm{Th}}(x,y)
:=\sup_{\gamma  \in \mathcal{S}}
\log\frac{\len_y(\gamma)}{\len_x(\gamma)},\label{eq:metric-curve}
\end{align}
 where $\len_z(\gamma)$ denotes the hyperbolic length of the unique geodesic in $(S,z)$ freely homotopic to $\gamma$. He shows that $e^{d_{\mathrm{Th}}(x,y)}$ is the optimal Lipschitz constant for self-homeomorphisms of $S$ homotopic to $\mathrm{id}_S: (S,x)\to(S,y)$ (see \cref{eq:metric-Lip} in \S\ref{s:optimal}), and that $d_{\mathrm{Th}}$ defines an asymmetric Finsler metric on $\teich(S)$. In particular, Thurston's theory furnishes the following description of the asymmetric Finsler norm of an arbitrary tangent vector $v\in T_x\teich(S)$:
\begin{align}
\|v\|_{\mathrm{Th}}
=\sup_{\gamma  \in \mathcal{S}}
\left(
d\log\len_{(\cdot)}(\gamma)
\right)_x(v)
=\sup_{\gamma  \in \mathcal{S}}
\left(
\frac{v(\len_{(\cdot)}(\gamma))}{\len_{x}(\gamma)}
\right),\label{eq:finsler}
\end{align}
where $\len_{(\cdot)}(\gamma):\teich(S)\to\mathbb{R}$ measures the length of the geodesic representative of $\gamma$ on $x$ for each $\gamma\in\mathcal{S}$.

\subsection{Infinitesimal structures on tangent and cotangent space} \label{ss:Infinitesimal}

The differential $d \log \len(\lambda)$ of $\log\len_{(\cdot)}(\lambda):\teich(S)\to\mathbb{R}$ at $x\in\teich(S)$ defines a linear function from the tangent space $T_x\mathcal{T}(S)$ to $\mathbb{R}$, and hence defines an element of the cotangent space $T^*_x\teich(S)$. Moreover, $d\log \len(\lambda)$ depends only on the projective class of $\lambda$ and defines a map
\begin{align}
\iota_x : \pml(S) \to T^*_x\teich(S),\quad
[\lambda]\mapsto d_x \log\len_{(\cdot)}(\lambda).\label{eq:iota}
\end{align}
\begin{theorem}[Thurston, \cite{ThS}] 
\label{Thurston} For every $x\in \mathcal{T}(S)$, the map $\iota_x$ is an embedding of $\pml(S)$ into $T^*_x\teich(S)$. The image set
\[
\mathbf{S}_x^*:=\iota_x(\pml(S))\subsetneq T_x^* \teich(S)
\]
bounds a convex body  in $T^*_x\teich(S)$ containing the origin.
\end{theorem}
Setting $\mathbf{S}_x^*\subsetneq T_x^*\teich(S)$ to be a unit co-norm sphere and extending by positive homothety define an asymmetric norm $\|\cdot\|_{\mathrm{Th}}^*$ on each vector space $T_x^*\teich(S)$. Its dual norm $\|v\|_{\mathrm{Th}}^{**}$ on the tangent space $T_x\teich(S)$ coincides with the original norm $\|\cdot\|_{\mathrm{Th}}$. To see this, observe that 
\begin{align}
\|v\|_{\mathrm{Th}}^{**}(v)
:=&
\sup
\left\{
w^*(v) \mid
w^*\in\mathbf{S}_x^*
\right\}\notag
\\
=&\sup 
\left\{
\iota_x([\lambda])(v)\in\mathbb{R} \mid\ [\lambda]\in\pml(S)
\right\}
=
\|v\|_{\mathrm{Th}}(v)
\label{eq:norm}.
\end{align}
The rightmost equality in \cref{eq:norm} follows from \cref{eq:finsler} and the density of the set of projectivised simple closed geodesics in $\pml(S)$. We denote the unit sphere for the norm $\|v\|_{\mathrm{Th}}$ by $\mathbf{S}_x\subsetneq T_x\teich(S)$.

\subsection{The unit norm sphere $\mathbf{S}_x$}
In \cref{sec:convexbodies}, we develop some elements of the general theory of convex bodies, and show that the boundary of any convex body has a \emph{convex stratification} as in \cref{defn:convexstrat} which decomposes it into (generically uncountably many) convex strata which are the interiors of its \emph{faces}.
In such a context, a face, introduced in \cref{defn:classicalface}, is defined as a convex subset of $\mathbf{S}_x$ containing the endpoints of every segment that has an interior point in that set.
Thurston shows, in a deep result \cite[Theorem~10.1]{ThS}, that the top-dimensional faces of the stratification of the unit Thurston norm sphere $\mathbf{S}_x\subsetneq T_x\teich(S)$, seen as the boundary of a convex body,  are naturally indexed by the set of homotopy classes of simple closed curves and that these faces cover almost all this Thurston norm unit  sphere.\smallskip

The aforementioned correspondence between the top-dimensional faces and simple closed curves is explicit: for any $\gamma\in\mathcal{S}$, its associated top-dimensional face is given by
\begin{align*}
\left\{
v\in \mathbf{S}_x
\mid 
\iota_x([\gamma])(v):=(\mathrm{d}\log\len_{(\cdot)}(\gamma))_x(v)
=\left\|v\right\|_{\mathrm{Th}}
\right\}.
\end{align*}
The above expression depends only on the projective measure class of $\gamma$, and we generalise it to general projective measured laminations $[\lambda]\in\pml(S)$ by
\begin{align}
N_x([\lambda])
:=\left\{ 
v\in \mathbf{S}_x
\mid 
\iota_x([\lambda])(v)
=\left\|v\right\|_{\mathrm{Th}}
 \right\}.
\end{align}
The set $N_x([\lambda])$ is necessarily an \emph{exposed} face of $\mathbf{S}_x$, that is, the intersection of $\mathbf{S}_x$ with one of the support hyperplanes of the convex body bounded by $\mathbf{S}_x$, where a \emph{support hyperplane} of $\mathbf{S}_x$ is a hyperplane in $T_x \teich(S)$ with non-empty intersection with $\mathbf{S}_x$, such that one of the two half-spaces it bounds contains the entire set $\mathbf{S}_x$.
\begin{corollary}[{contravariant labelling, \cref{contravariant}}]
For two arbitrary projective multicurves $[\Gamma_1],[\Gamma_2]\in\pml(S)$, denote their respective supports by $|\Gamma_1|$ and $|\Gamma_2|$. Then,
\begin{align}
|\Gamma_1|\subseteq |\Gamma_2|
\textit{ if and only if }
N_x([\Gamma_2])\subseteq N_x([\Gamma_1]).
\end{align}
\end{corollary}
The above corollary induces the embedding of a na\"ive ``dual'' to the curve complex as a subcomplex of $\mathbf{S}_x$.

\begin{corollary}[{embedded dual curve complex, \cref{prop:embeddedccx}}]
For every $x\in\teich(S)$, the convex stratification of $\mathbf{S}_x$ contains a set $\mathbf{C}_x$ consisting of a subcollection of strata which is dual to the curve complex in the sense that
\begin{itemize}
\item
the support $|\Gamma|$ of any projective multicurve $[\Gamma]\in\pml(S)$ is assigned to a face $N_x([\Gamma])\subsetneq \mathbf{S}_x$;
\item
the subset-relation for simplices in the curve complex corresponds, by the above assignment, to the superset-relation for faces in $\mathbf{C}_x$;
\item
the dimension of each simplex in the curve complex is equal to the codimension of its corresponding face in $\mathbf{C}_x$ as a subset of $\mathbf{S}_x$.
\end{itemize}
\end{corollary}

This polytope-like convex geometry of $\mathbf{S}_x^*$ and $\mathbf{S}_x$ tells us that the Finsler structure of Thurston's metric is highly heterogeneous, lending further support for the infinitesimal rigidity of the Thurston metric.

\subsection{The unit co-norm sphere $\mathbf{S}_x^*$}
The convex stratification of the dual unit co-norm sphere $\mathbf{S}_x^*$ enjoys a cleaner and more natural ``dual'' statement:
\begin{corollary}[{embedded curve complex, \cref{cor:embeddedcx}}]
The convex stratification of $\mathbf{S}_x^*$ contains the curve complex for $S$ as a subcomplex $\mathbf{C}_x^*$.
\end{corollary}
The convex stratification of $\mathbf{S}_x^*$ consists of uncountably many faces of dimensions varying from $0$ to $3g-4+n$ (\cref{thm:bound}), with many of its faces ``covariantly" labelled by geodesic laminations which are the support of measured laminations. Let $\left|\mathcal{ML}\right|$ denote the set of supporting laminations of (projective) measured laminations on $S$, and let $\mathcal{F}$ denote the set of faces of $\mathbf{S}_x^*$. We define an injective map  (see \cref{lem:supportset}) 
\[
F_{(\cdot)}\colon \left|\mathcal{ML}\right| \to \mathcal{F},\quad |\lambda|\mapsto F_{|\lambda|}
\] 
which assigns to the support $|\lambda|\in\left|\mathcal{ML}\right|$ of an arbitrary measured lamination $\lambda\in\mathcal{ML}(S)$ the minimal face $F_{|\lambda|}$ containing the $\iota_x$-image of the set of  (projective) measured laminations supported on $|\lambda|$. We show in \cref{thm:equisupport} that $F_{|\lambda|}$ consists precisely of the $\iota_x$-image of projective measured laminations supported on $|\lambda|$, thereby deriving the following
\begin{corollary}[covariant labelling, \cref{inclusion}]
For two arbitrary projective laminations $[\lambda],[\mu]\in\pml(S)$, 
\[
|\mu|\subsetneq|\lambda|\text{ if and only if }
F_{|\mu|}
\subsetneq
F_{|\lambda|}.
\]
\end{corollary}

\subsection{Closure correspondence}

There is deeper structure attached to the face-assigning map $F_{(\cdot)}$. We introduce two notions of closure,
\begin{itemize}
\item
\emph{support closure} (\cref{defn:supportclosure}) for laminations: given an arbitrary measured lamination $\lambda$, we add to its support lamination $|\lambda|$ the boundary of the smallest (up to isotopy) incompressible subsurface of $S$ containing $\lambda$. We denote the new geodesic lamination by $\widehat{|\lambda|}$ and we call it the \emph{support closure} of $\lambda$.

\item
\emph{adherence closure} (\cref{defn:adherence}) for faces of convex bodies: We assign to each face $F$ of $\mathbf{S}_x^*$ a specific superface $\widehat{F}$ in a manner purely determined by the convex geometry of $\mathbf{S}_x^*$. Precisely, the definition is as follows:
First we say that
a face $F$ of $\mathbf{S}_x^*$ is \emph{adherent} to a face $F'\supseteq F$ if for any face $F''$ containing $F$ there is a face which contains both $F'$ and $F''$ as subfaces.
Each face $F$ is adherent to a unique face which is maximal with respect to inclusion and which we call the \emph{adherence closure $\widehat{F}$ of $F$}. 

\end{itemize}

We show, as a corollary to \cref{measure} and \cref{set of measures}, that

\begin{corollary}[{\cref{naturaltransformation}}]
For any projective measured lamination $[\lambda]$,
\[
\widehat{F_{|\lambda|}}
=
F_{\widehat{|\lambda|}}.
\]
\end{corollary}

The above result asserts that the support closure operation $\widehat{\cdot}$ defines a natural transformation of the face assigning map $F_{(\cdot)}$ --- regarded as a functor between two subcategories of the category of sets (see \cref{q:covariant}).

\subsection{Topological rigidity}

Any homeomorphism $f^*:\mathbf{S}_y^*\to\mathbf{S}_x^*$ between two unit co-norm spheres induces a homeomorphism 
\[
f_{y,x}:=\iota_x^{-1} \circ f^* \circ \iota_y:\pml(S)\to\pml(S).
\] 
\cref{prop:embeddedccx}, combined with the celebrated rigidity of the curve complex \cite{Iv,K,L} and the density of projectivised multicurves in $\pml(S)$, tells us that if $f^*$
\begin{itemize}
\item
preserves the convex stratification, and 
\item
takes the embedded curve complex $\mathbf{C}_y^*\subsetneq\mathbf{S}_y^*$ to the curve complex in $\mathbf{C}_x^*\subsetneq\mathbf{S}_x^*$, 
\end{itemize}
then $\iota_x\circ f^*\circ \iota_y^{-1}$ is induced by an extended mapping class.\medskip

In particular, linear isometries from $(T_y^*\teich(S),\|\cdot\|_\mathrm{Th})$ to $(T_x^*\teich(S),\|\cdot\|_\mathrm{Th})$ restrict to homeomorphisms from $\mathbf{S}_y^*$ to $\mathbf{S}_x^*$ which preserve the convex stratification (\cref{thm:linvariance}) and take $\mathbf{C}_y^*$ to $\mathbf{C}_x^*$ --- we prove the latter claim using various adherence-related properties to distinguish the faces in $\mathbf{C}_x^*$ from the other faces in $\mathbf{S}_x^*$. This yields:

\begin{theorem}[Topological rigidity]
\label{first part}
Let $f^*\colon T_y^* \teich(S) \to T_x^* \teich(S)$ denote the linear isometry  dual to a linear Thurston-norm isometry $f: T_x \teich(S) \to T_y \teich(S)$.
Then the map $f_{y,x}=\iota_x^{-1} \circ f^* \circ \iota_y:\pml(S)\to\pml(S)$ is necessarily induced by some diffeomorphism $h: (S,x) \to (S,y)$.
\end{theorem}

\begin{remark}
\cref{first part} is key to our strategy for proving infinitesimal rigidity of the Thurston metric, and relies upon Ivanov's theorem, completed by Korkmaz and Luo, \cite{Iv,K,L}. This argument fails when $(g,n)=(1,1)\text{ or }(0,4)$, which are precisely the two cases (effectively) covered by \cite{DLRT}.
\end{remark}

\subsection{Thurston's stretch maps}
We have hitherto mainly considered the infinitesimal theory of the Thurston metric. Thurston bridges this to the global perspective via objects called \emph{stretch paths} (\cref{sec:stretch}). He (implicitly) proves that the Finsler metric defined by \cref{eq:finsler} is equal to the asymmetric metric defined by \cref{eq:metric-curve}.
\medskip

A geodesic lamination is said to be \emph{complete} if all its complementary components are ideal triangles. 
Stretch paths are parametrised by three inputs: an initial point $x\in\teich(S)$, a complete geodesic lamination $\Lambda$ on $(S,x)$, and a time period over which to stretch. A time-$t$ stretch uniformly increases the length of every geodesic segment in $\Lambda$ by a factor of $e^t$ whilst non-uniformly shrinking the ``dual'' perpendicular measured foliation $\mu$, which is determined by the initial point $x\in\teich(S)$. In each complementary component of $\Lambda$, the leaves of the perpendicular measured foliation are pieces of horocycles that intersect the leaves of the lamination perpendicularly,  see \cref{fig:stretch-triangle}.
The third author shows in \cite{Pap2}  that stretch paths always converge in forward time to the projective measured lamination representing $\mu$ in the Thurston compactification. The point of the forward-time flow induced by a stretch path associated with a fixed geodesic lamination $\Lambda$  is highly dependent on the starting point, and this provides a codimension-$1$ projection map from $\teich(S)$ to $\pml(S)$ when $S$ has no cusps \cite{A}.\medskip

A geodesic lamination is said to be \emph{finite-leaf} if it consists of finitely many geodesics.
In \cref{s:crown}, we explicitly express stretch paths for all complete finite-leaf laminations as real analytic functions of certain Fenchel--Nielsen length and twist parameters and Thurston's shearing parameters on $\teich(S)$ (see \cite{Huang-Thesis} and \cite{CF}), and give a new proof to the following result (\cref{thm:backtime}) due to Guillaume Th\'{e}ret \cite[Theorem~4.1]{The}.
In this statement, given the point $x$ in the Teichm\"uller space $\mathcal{T}(S)$ and  given the complete geodesic lamination $\Lambda$ on $x$, then for $t\in \mathbb{R}$, the $e^{-t}$-stretch path with respect to $\Lambda$ is a path in $\mathcal{T}(S)$ which is a time-$t$ stretch path but traversed in the opposite direction. In general, it is not a geodesic for the Thurston metric.

\begin{theorem}[back-time convergence {\cref{thm:backtime}}]
Consider a finite-leaf complete geodesic lamination $\Lambda$, and denote its closed leaves by $\gamma_1,\ldots,\gamma_k$. As $t\to\infty$, for any $x\in\teich(S)$, the $e^{-t}$-stretch path with respect to $\Lambda$, $x_{-t}:=\mathrm{stretch}(x,\Lambda,-t)$ converges to the uniformly weighted projectivised multicurve
\[
[\gamma_1+\gamma_2+\cdots+\gamma_k]\in\pml(S)=\partial_\infty\teich(S),
\] 
where $\pml(S)$ is seen as the boundary $\partial_\infty \teich(S)$ of a Thurston compactification of $\teich(S)$. Note that this limit is independent of the starting point $x$.
\end{theorem}

The use of a combination of  Fenchel--Nielsen and shear coordinates will show clearly that this is an effective combination of coordinates.

\begin{remark}
There is a proto-dictionary between the structure of various $\pml$-related flows on $\teich(S)$ and the Anosov nature of the geodesic flow on the unit tangent bundle $T^1\mathbb{H}$ of the hyperbolic plane $\mathbb{H}$:

\begin{itemize}
\item
The projectivisation $\pml(S)$ of the measured lamination space $\ml(S)$ defines a boundary at infinity of $\teich(S)$, just as the projectivisation $\mathbb{P}(L^+)=\mathbb{S}^1$ of the positive light-cone $L^+$ in Minkowski space $\mathbb{R}^{2,1}$ (see \cite[p.~302]{Pen}) defines the ideal boundary of (the hyperboloid model of) $\mathbb{H}\subsetneq\mathbb{R}^{2,1}$.

\item
The Thurston Finsler cometric on the cotangent bundle of $\teich(S)$ has conorm given by taking $\iota_x(\pml(S))\subsetneq T_x^*\teich(S)$ as the unit sphere (\cite[Theorem~5.1]{ThS}, see also \cref{ss:Infinitesimal} of the present paper). Analogously, the standard hyperbolic cometric on the cotangent bundle of $\mathbb{H}$ has conorm given by embedding $\mathbb{P}L^+$ as the unit circle in $T^*\mathbb{H}$, where the embedding is defined (in analogy with \cref{eq:iota}) as follows: for $u\in\mathbb{H}\subsetneq\mathbb{R}^{2,1}$ and $[v]\in\mathbb{P}L^+$, define
\begin{align}
\iota_u\colon
\mathbb{P}L^+\to T_u^*\mathbb{H},\quad [v]\mapsto d_u \log \langle \cdot, v\rangle, 
\end{align}
where $\langle\cdot,\cdot\rangle$ here denotes the Minkowski metric on $\mathbb{R}^{2,1}$.

\item
Consequentially, the Thurston Finsler metric on $\teich(S)$ parallels the hyperbolic metric on $\mathbb{H}$, and hence the unit tangent bundle $T^1\teich(S)$ is akin to the unit tangent bundle $T^1\mathbb{H}$.

\item
The flow induced by Thurston's stretch maps is akin to the geodesic flow --- it is expanding in positive time and contracting in negative time.

\item
The earthquake flow with respect to uniformly weighted projectivised multicurves (which constitute the closed leaves of a finite-leaf lamination) is akin to the unstable horocyclic flow.

\item
The earthquake flow with respect to the dual perpendicular measured foliation is analogous to the stable horocyclic flow.

\end{itemize}
\end{remark}

\subsection{Stretch vectors.} Thurston shows that his stretch paths are analytic with respect to their time parameter \cite[Corollary~4.2]{ThS}, and we refer to their tangent vectors as \emph{stretch vectors} (\cref{def:infinitesimal}). 
He also shows  in \cite[Theorem~8.5]{ThS} that any two points in $\teich(S)$ can be joined by a geodesic with respect to the Thurston metric which is a concatenation of stretch paths (\ie  subsegments of stretch rays).
For $x, y \in \teich(S)$, following Dumas, Lenzhen, Rafi and Tao in \cite{DLRT} and Guéritaud in  \cite[Problem 3.2]{Su-Problems}, we define the \emph{envelope} $\mathrm{Env}(x,y)$ to be the set of points through which geodesics from $x$ to $y$ pass.

\begin{remark}
Thurston's construction of concatenated stretch paths as Thurston-metric geodesics is analogous to Dantzig's simplex method in linear programming. In this analogy:

\begin{itemize}
\item
The role of the finitely-constrained feasible region
   is served by the envelope $\mathrm{Env}(x,y)\subsetneq\teich(S)$.
\item
The feasible region in the simplex method is specified by a finite system of inequalities. One naïvely expects $\mathrm{Env}(x,y)$ to be defined by a countable collection of inequalities labelled by simple closed curves $\gamma$ in the complement $S\setminus \mu(x,y)$, where $\mu(x,y)$ denotes the maximal ratio-maximising chain recurrent lamination, whose form is specified by how much one may Fenchel--Nielsen twist with respect to $\gamma$ whilst increasing its length from $x$ and to $y$;

\item
The goal of the simplex method is to start at a basic feasible solution (\ie extreme point) and to algorithmically traverse the feasible region and reach the maximum of some function over the feasible region. Analogously, Thurston takes the point $x$ and maximises the Thurston distance from $x$ within $\mathrm{Env}(x,y)$. In particular, this algorithmically produces a path in $\mathrm{Env}(x,y)$ going from $x$ to $y$, as the maximum distance from $x$ is uniquely attained by $y$.
\item
The algorithm for the simplex method is given by traversing (only) along edges of the feasible region. One expects the edges of $\mathrm{Env}(x,y)$ to be given by stretch paths with respect to maximal chain-recurrent laminations, and Thurston's algorithm similarly starts at $x$ and travels along stretch paths, pivoting to new stretch paths whenever we reach the end of a stretch path on $\mathrm{Env}(x,y)$, and repeating this until we reach $y$.
\end{itemize} 

\end{remark}

Consider the cross-sections of $\mathrm{Env}(x,y)$ at given distance from $x$. As the distance tends to $0$, we expect to get the Thurston-norm unit sphere $\mathbf{S}_x$, and so the above conjectural description suggests a conjecture which we shall state now (\cref{conj:extremepoint}), but before stating it we need to recall some terminology: 

A geodesic lamination $\lambda$ on a hyperbolic surface $S$ is said to be \emph{chain-recurrent} if for any point $x$ on $\lambda$ and for any $\epsilon >0$ there exists a parametrised $C^1$-closed curve $\delta$ on $S$ containing $x$ such that for any unit length path $l_1$ contained in $\delta$, there exists a unit length path $l_2$ contained in $\lambda$ such that $l_1$ and $l_2$ are $\epsilon$-close in the $C^1$-topology. 
A chain-recurrent lamination $\lambda$ on $S$ is said to be \emph{maximal} if there is no chain-recurrent lamination on $S$ properly containing it. 
We also recall that a point $p$ on the boundary $S$ of a convex body is called an {\em extreme point} when it does not lie on the interior of a segment lying on $S$.

\begin{conjecture}\label{conj:extremepoint}
The set of stretch vectors in $\mathbf{S}_x$ that are tangent to stretch paths with respect to maximal chain-recurrent laminations characterise the set of extreme points of $\mathbf{S}_x$. 
\end{conjecture}
\cref{conj:extremepoint} is true when $S$ is either the once-punctured torus or the $4$-punctured sphere \cite{DLRT}, but is unknown for surfaces of general type. In fact, neither direction of the characterisation is known. Nevertheless, the conjecture being true would imply the set of stretch vectors (with respect to maximal chain-recurrent laminations) would be preserved under linear Thurston-norm isometries. We provide support for \cref{conj:extremepoint} by showing something stronger:

\begin{corollary}[equivariant stretch vectors, \cref{equivariant stretch}]
\label{equivariant stretch 1}
Given an arbitrary linear Thurston-norm  isometry $f\colon T_x\teich(S)\to T_y\teich(S)$, let $[h]\in\mathrm{Mod}^*(S)$ denote its inducing extended mapping class (topological rigidity, \cref{first part}). For an arbitrary maximal chain-recurrent lamination $\Lambda$ of $S$, the stretch vector $v_{\Lambda}\in T_x\teich(S)$ satisfies the following equivariance property:
\begin{align*}
f(v_\Lambda)=v_{h(\Lambda)}\in T_y\teich(S).
\end{align*}
\end{corollary}

\begin{remark}
It is tempting to wonder if \cref{equivariant stretch 1} (and even \cref{conj:extremepoint}) might be derived from Thurston's \cite[Theorem~8.4]{ThS}. The difficulty lies in showing that the faces $N_x([\lambda])$ vary  continuously in the Hausdorff topology as $x\in\teich(S)$ varies. We prove \cref{equivariant stretch 1} by showing that any stretch vector with respect to maximal chain-recurrent laminations necessarily arises as the Hausdorff limit of a sequence of shrinking faces of $\mathbf{S}_x$ (\cref{uniquelimit}) and then invoking topological rigidity.
\end{remark}

\subsection{Geometric rigidity}

We show in \cref{lem:length} that the simple length spectrum for any hyperbolic metric $x\in\teich(S)$ may be extracted from the asymptotic behaviour of Thurston norms of differences between certain stretch vector pairs.\medskip

For any simple closed geodesic $\gamma\in\mathcal{S}$, there are sequences $\{v_{+,m} , v_{-,m}\}$ of stretch vector pairs such that
\begin{align*}
\ell_x(\gamma)
=
\lim_{m\to\pm\infty}
\frac{-\log\|v_{+,m}-v_{-,m}\|_{\mathrm{Th}}}{i(\gamma,\alpha_0)\cdot |m|}
,
\end{align*}
where  $\alpha_0$ is a simple closed geodesic which has nonempty transverse intersection with $\gamma$, whose Dehn twists are used to generate $\{v_{+,i}\}$ and $\{v_{-,i}\}$. This observation, coupled with the equivariance of stretch vectors under Thurston-norm isometries (\cref{equivariant stretch 1}), promotes topological rigidity to \emph{geometric rigidity}:

\begin{theorem}[Geometric rigidity]
\label{second part}
Let $f \colon T_x \teich(S) \to T_y \teich(S)$ be a linear Thurston-norm isometry and let $h$ be a diffeomorphism representing the mapping class inducing $f$ as given by topological rigidity (\cref{first part}). Then, for every simple closed curve $\gamma\in\mathcal{S}$ on $S$, we have 
\[
\ell_x(\gamma)=\ell_y(h(\gamma)),
\]
and hence $(S,y)=(S,h(x))$.
\end{theorem}

\subsection{Infinitesimal rigidity}
\label{sec:strategy}

We begin with a remark:

\begin{remark}
\label{mazurulam}
The statement of \cref{main} \emph{appears to be} slightly stronger than the main result in Pan's paper \cite{P}, as the latter assumes linearity of the isometry. However, this linearity condition is unnecessary. Indeed, isometries of the Thurston norm are also isometries of its (additive) symmetrisation. Therefore we can use the Mazur--Ulam theorem \cite{MU} (which applies to symmetric normed spaces), which implies that any isometry of the Thurston norm is linear.
\end{remark}

When combined, geometric rigidity (\cref{second part}) and the Mazur--Ulam theorem (\cref{mazurulam}) suffice to prove infinitesimal rigidity (\cref{main}). Further combining the infinitesimal rigidity with \cite[Theorem~6.1]{DLRT} then yields:

\begin{corollary}[Local rigidity]\label{thm:local}
Consider a connected open set $U\subset\teich(S)$. Then any isometric embedding $(U,d_{\mathrm{Th}})\to (\teich(S),d_{\mathrm{Th}})$ is the restriction to $U$ of an element of the extended mapping class group of $S$.
\end{corollary}

And this in turn implies the following result:

\begin{corollary}[Global rigidity] \label{thm:walsh}
Isometries of $(\teich(S),d_{\mathrm{Th}})$ are precisely those induced by extended mapping classes of $S$.
\end{corollary}

As previously noted, \cref{thm:walsh}, which was first proven by Walsh \cite{W}, is an analogue of the celebrated  (global form of the) Royden Theorem \cite{Royden1971}, with the role of the Teichm\"uller metric supplanted by the Thurston metric. Walsh's proof is based on the action of the mapping class group on the horocyclic boundary of the Teichm\"uller space of the surface equipped with the Thurston metric. Dumas, Lenzhen, Rafi and Tao \cite{DLRT} utilise a different strategy: they first derive infinitesimal rigidity in the two special cases when $S$ is either a one-cusped torus or a four-cusped sphere, thereby completing the Thurston metric analogue of Royden's theorem. Our method of proof shares many similar ingredients to theirs. In particular, our main results in the special cases where $S$ is either a one-cusped torus or a four-cusped sphere are due to \cite{DLRT}.

\subsection{Paper outline}

This paper is structured as follows:
\begin{itemize}
\item
In \cref{s:stretch} we provide a sketch of the preliminaries underpinning the theory of the Thurston metric on hyperbolic surfaces. The basic construction is that of a $K$-Lipschitz homeomorphism of an ideal triangle. Such maps are glued to constitute stretch maps between hyperbolic surfaces, which define stretch lines in Teichm\"uller space. We highlight the main results from Thurston's paper \cite{ThS} needed in the constructions for the present paper.
\item
In \cref{s:crown} we describe all $K$-stretch maps on crowned annuli, \ie annuli with boundary cusps, in terms of Fenchel--Nielsen coordinates and shearing parameters. This also paves the way for describing stretch maps with respect to finite-leaf laminations on (finite type) hyperbolic surfaces.

\item
In \cref{s:infinitesimal} we take the stretch maps described in \cref{s:crown} and we determine the behaviour of stretch vectors, \ie tangent vectors to stretch maps. These estimates are used to isolate the lengths of simple closed geodesics in the limiting behaviour of the difference between certain sequences of pairs of stretch vectors, and hence play a major role in establishing geometric rigidity. Moreover, they help us to determine the codimension of $\iota_x$-images of multicurves in $\mathbf{S}_x^*$
  
\item
In \cref{sec:convexbodies} we develop the general theory of convex bodies $D$ in finite-dimensional vector spaces $\mathbb{V}$, using faces on the boundary $\partial D$ to stratify this boundary. We introduce various invariants associated with points and faces on $D$, including dimension, face-dimension, adherence-dimension and codimension.

\item
In \cref{s:unitball} we explore the convex geometry of the unit tangent and contangent spheres in $T_x\teich(S)$ and $T^*_x\teich(S)$, showing that much of the inclusion structure for many faces of $\mathbf{S}_x^*$ is encoded in terms of  geodesic laminations associated with faces. We also produce bounds and derivations for various notions of dimension associated with points in $\mathbf{S}^*_x$.

\item
In \cref{s:rigidity} we give two distinct proofs of topological rigidity before promoting topological to geometric rigidity, thereby establishing the infinitesimal rigidity of the Thurston metric. We also put forward consequences of the infinitesimal rigidity result and raise further questions and lines of investigation. 

\end{itemize}

 Let us say now a word on references. The bases of Thurston's theory of surfaces including hyperbolic structures, geodesic lamination, Teichm\"uller spaces and mapping class groups are contained in Chapter 8 of his lecture notes \cite{Th4} and in his article \cite{Th2}, with more details and developments in \cite{FLP}, \cite{HP} and \cite{CB}. In the next section, we shall give references for Thurston's theory of best Lipschitz maps and Thurston's metric on Teichm\"uller space.  At the same time, we shall provide details for  some constructions of Thurston that were only sketched in his paper \cite{ThS}, at a few times slightly correcting some of the statements.

In October 2012, a workshop with the title ``Thurston's metric on Teichm\"uller space" was organized at the IAM in Palo Alto. A list of open problems was compiled by Weixu Su during that workshop, and an updated list,  compiled by the same author, was published in Volume V of the Handbook of Teichm\"uller theory \cite{Su-Problems}. The work in the present paper answers positively two questions of the updated list, viz.\medskip

\noindent\textbf{Problem 2.3.} (K. Rafi) \textit{What properties of the hyperbolic surface are determined by the Thurston infinitesimal norm?}\medskip

\noindent\textbf{Problem 2.6.} (Collective) \textit{Is each local isometry of  Teichm\"uller space with the Thurston metric induced by an element of the extended mapping class group?}\medskip

  Several questions on that list, concerning the  properties of the Thurston metric, remain open.

\subsection{Acknowledgements}

The first named author wishes to acknowledge support by Beijing Natural Science Foundation Grant no. 04150100122.
The second named author is supported by JSPS Grant-in-Aid for Fund for the Promotion of Joint International Research (Fostering Joint International Research(B)) 18KK0071.
The three authors express their gratitude to the Erwin Schrödinger International Institute for Mathematics and Physics for its  hospitality during the writing of this paper.

\section{Preliminaries on Thurston's metric}\label{s:stretch}
 
 We shall use the fact that the length function defined on the set $\mathcal{S}$ of homotopy classes of simple closed curves has a continuous and positively homogeneous extension to the space $\ml=\ml(S)$ of compactly supported measured laminations on $S$ (see Thurston \cite{Th3} and Kerckhoff \cite{KeC}) and that the supremum in  \cref{eq:metric-curve} is realised over the projective space $\pml(S)$, as it is compact. Thus, the distance function $d_{\mathrm{Th}}$ defined by \cref{eq:metric-curve} is also
equal to the function
\begin{equation}\label{eq:K}
R(x,y)
:=\max_{[\lambda]  \in \pml(S)}
\log\frac{\len_y(\lambda)}{\len_x(\lambda)}
\end{equation}
(the letter $R$ here stands for ``Ratio"), where $\lambda\in\ml(S)$ is an arbitrary measured geodesic lamination representative of $[\lambda]\in\pml(S)$.
\subsection{Lipschitz constants formulation of the Thurston metric}\label{s:optimal}

In \cite{ThS}, Thurston gives another formula for the metric defined in \cref{eq:metric-curve}. This alternative formulation is based on the comparison of the shapes 
 of two marked hyperbolic surfaces by examining the smallest Lipschitz constant of a homeomorphism between them in the right homotopy class. \medskip

Given any two hyperbolic structures $x,y\in\teich(S)$ on $S$ and any homeomorphism $\varphi: (S,x)\to (S,y)$, the Lipschitz constant $\mathrm{Lip}(\varphi)$ of $\varphi$ is defined as the quantity
\[
\mathrm{Lip}(\varphi)
:=\sup_{p\not=q\in S}\frac{d_y(\varphi(p),\varphi(q))}{d_x(p,q)}\in\mathbb{R}\cup\{\infty\},
\]
where $d_x$ denotes the distance function on the hyperbolic surface $(S,x)$. Consider the expression
\begin{align}
\label{eq:metric-Lip}
L(x,y)
:=\inf_{\varphi\sim\mathrm{Id}_S}\log \mathrm{Lip}(\varphi) \in\mathbb{R}.
\end{align}

The quantity $L(x,y)$ is invariant under isotopic deformations of the hyperblic metrics $x$ and $y$ on $S$, and hence yields a well-defined function $L\colon\teich(S)\times \teich(S)\to\mathbb{R}$. Thurston shows that this function satisfies all the properties of an asymmetric metric \cite[\S 2]{ThS} and he proves the following: 
\begin{theorem}[{Thurston \cite[Theorem~8.5]{ThS}}]\label{th:R=L}
The asymmetric metric defined in \cref{eq:metric-Lip} coincides with the asymmetric metric $d_{\mathrm{Th}}$ defined in \cref{eq:metric-curve} or \cref{eq:K}, \ie
\begin{align}\label{eq:R=L}
L\equiv R
\end{align}
\end{theorem}

\subsection{Stretch maps on ideal triangles}

The key advantage of \cref{eq:K} over \cref{eq:metric-curve} is that, unlike taking the supremum over $\mathcal{S}(S)$ for $L$, the projective measured lamination $[\lambda]\in\pml(S)$ maximising $R$ is realised, since $\pml(S)$ is compact.
 The proof that Thurston gives of Theorem \ref{th:R=L} encodes the fact that
 the best Lipschitz constant of a homeomorphism between the two hyperbolic surfaces $(S,x)$ and $(S,y)$ is attained on some compactly supported measured lamination $\lambda$, and that any optimal Lipschitz constant-realising such homeomorphism must be stretching ``maximally'' at a constant rate on (the unit tangent bundle of) the support of $\lambda$. Suppose, in particular, that some projective measured lamination $[\lambda]$ which realises the maximum of $R$ between $(S,x)$ and $(S,y)$ is supported on a complete geodesic lamination $|\lambda|$, \ie a geodesic lamination  $|\lambda|$ such that the complementary regions on $S\setminus|\lambda|$ are composed of ideal triangles. In this case, the fact that a uniquely determined Lipschitz map on $|\lambda|$ homeomorphically extends to all of $S$ means that the theory of Lipschitz maps between ideal triangles is both natural and unavoidable in Thurston's theory of Lipschitz homeomorphisms between hyperbolic surfaces.\smallskip

In \cite[\S4]{ThS}, Thurston concretely describes the construction of a canonical $K$-Lipschitz self-map of an ideal triangle, for every $K\geq 1$. We recall now this construction. It is illustrated in  \cref{fig:stretch-triangle}. 

\begin{figure}[h!]
\begin{center}
\includegraphics[scale=0.5]{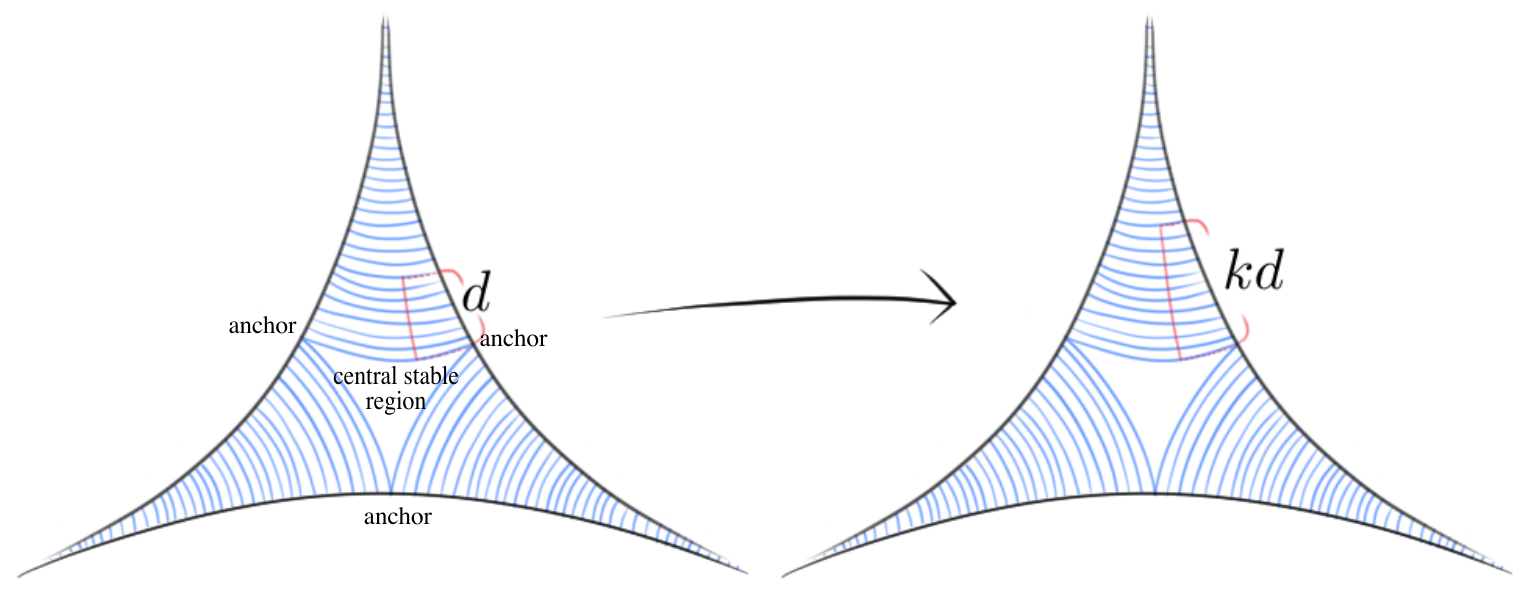}
\caption{A stretch map between two ideal triangles.}
\label{fig:stretch-triangle}
\end{center}
\end{figure}

The ideal triangle is foliated by pieces of horocycles which are perpendicular to the boundary except for a central non-foliated region bordered by three pieces of such horocycles. 

\begin{definition}[central stable triangle and anchors]
We refer to the central nonfoliated region of an ideal triangle as the \emph{central stable triangle}, and refer to the three vertices of the stable triangle, \ie the three points where two distinct horocycle leaves meet, as the \emph{anchors}. 
\end{definition}

The ideal triangle equipped with its horocyclic foliation is unique up to isometry. The $K$-Lipschitz self-map of the triangle is taken to be the identity on the central stable triangle and to send every horocyclic segment at distance $d$ from this region to a horocyclic segment at distance $Kd$. Furthermore, on each horocyclic segment at distance $d$ from the central stable triangle, the map uniformly contracts horocylic arc-length by the factor  $e^{-dK}$.  These properties uniquely determine the map. Restricted to the boundary geodesics, this map expands distance by the factor $K$.  One can show that this map is $K$-Lipschitz. It is called the \emph{$K$-stretch map} of the ideal triangle.

\begin{definition}[stretch-invariant foliations]
There are two foliations that are invariant under stretching and supported on the complement of the central stable triangle: the horocyclic foliation illustrated in \cref{fig:stretch-triangle} and the geodesic foliation orthogonal to it. We refer to these respectively as the \emph{stretch-invariant horocyclic foliation} and the \emph{stretch-invariant geodesic foliation}.
\end{definition}

\subsection{Stretch maps on surfaces}\label{sec:stretch}

Thurston uses $K$-stretch maps between ideal triangles to generate $K$-stretch maps between hyperbolic surfaces \cite{ThS}. Consider a complete geodesic lamination $\nu$ on the surface $S$ equipped with a hyperbolic structure $x\in\teich(S)$. We assume that not all leaves of $\nu$ converge to cusps at both ends, \ie that $\nu$ is not an ideal geodesic triangulation, but impose no other condition; $\nu$ need not carry a transverse measure of full support. Thurston constructs a family $\{x_t\}_{t\in[0,\infty)}$ of hyperbolic structures on $S$ such that for every $t\geq 0$, 

\begin{itemize}
\item $L(x,x_t)=R(x,x_t)=t$;
\item there exists a homeomorphism $\phi_t\colon (S,x)\to (S,x_t)$ with $\mathrm{Lip}(\phi)= e^t$;
\item $\phi_t$ sends the geodesic lamination $\nu$ on $(S,x)$ to the (unique geodesic) lamination (isotopic to) $\nu$ on $(S,x_t)$;
\item the restriction of $\phi_t$ to the lamination $\nu$ on $(S,x)$ is an $e^t$-Lipschitz map which uniformly expands arc-length on $\nu$ by $e^t$.
\end{itemize}

Thurston's construction is fairly straightforward when the lamination $\nu$ consists of finitely many geodesic leaves: the $e^t$-stretch maps on each of the $4g-4+2n$ ideal triangles complementary to $\nu$ stretches its boundary geodesics uniformly by a factor of $e^t$. Hence these maps are glued together continuously. In particular, all of the leaves of $\nu$ are boundary geodesics and hence this defines an $e^t$-stretch map from $x$ on $S$ to a different hyperbolic structure $x_t$ on $S$ satisfying the properties listed above.

However, when the lamination $\nu$ has (uncountably) infinitely many leaves, the situation is much subtler as only finitely many leaves in $\nu$ bound the complementary ideal triangles, and it is far from obvious as to whether Thurston's $e^t$-stretch maps on these complementary ideal triangles continuously extend to all the leaves in $\nu$. Instead, Thurston studies the space of all possible hyperbolic metrics on small $\epsilon$-neighbourhoods of $\nu$, and parametrises this space in terms of functions describing transversely measured foliations (with mainly horocyclic leaves) transverse to $\nu$ on the aforementioned $\epsilon$-neighbourhood. 

To get a sense of how this works, observe that the stretch-invariant horocyclic foliation on an ideal triangle (see \cref{fig:stretch-triangle}) is endowed with a natural transverse measure which measures the arc-length of subsegments of the leaves of $\nu$ (this determines the transverse measure). 
Thurston shows \cite[Proposition~4.1]{ThS} that any such transversely measured foliation (with closed leaves around cusps and with infinite transverse measure in the neighbourhoods of cusps) is realised by a hyperbolic metric  on the $\epsilon$-neighbourhood of $\nu$ in such a way that the transverse measure of the foliation measures arc-length on subsegments of the leaves of $\nu$. In so doing, he shows that the initial transversely measured foliation describing the neighbourhood of $\nu$ in $x\in\teich(S)$, when multiplied by $e^t$, $t\in[0,\infty)$, corresponds to a ray of hyperbolic metrics in the $\epsilon$-neighbourhood of $\nu$. The requisite  ray $\{x_t\}_{t\in[0,\infty)}$ of metrics and the associated Lipschitz maps $\{\phi_t\}$ are obtained by filling in the rest of the surface using $e^t$-stretch maps on ideal triangles. The family of hyperbolic metrics defined by this ray satisfies the previously listed properties.

\begin{definition}[stretch lines and stretch rays]
\label{defn:stretchpath}
The ray $\{x_t\}_{t\geq 0}$ in Teichm\"uller space is called a \emph{stretch ray} starting at $x_0$,  and a line 
 $\{x_t\}_{t\in \mathbb{R}}$ obtained in this way is a called \emph{stretch line}. We use the term \emph{stretch path} to refer to a subarc of either a stretch line or a stretch ray.
 The complete lamination $\nu$ used for this construction is called the {\em stretching locus} of the path.
 \end{definition}

\begin{remark}
\label{maxstretch}
We see by construction that a $K=e^t$-stretch map supported by a geodesic lamination $\nu$ on a hyperbolic surface $(S,x)$ stretches the stretch-invariant geodesic foliation on each complementary ideal triangle by a factor of $K>1$ and contracts in the direction of the (orthogonal) stretch-invariant horocyclic foliation. Therefore, the only geodesic segments stretched by $K$ are those which lie on $\nu$ and the stretch-invariant geodesic foliation. In particular, since the leaves of the stretch-invariant geodesic foliation which do not lie on the boundary of the complementary ideal triangles (and hence $\nu$) necessarily meet one central stable triangle, such leaves can never be complete geodesics. Hence, any complete geodesic which is stretched by $K$ is necessarily a leaf of $\nu$. In particular, this means that (compactly supported) measured laminations which transversely intersect $\nu$ cannot have their length increase by the factor $K$. This fact is used by Thurston, for example, in the proof of \cite[Theorem~5.1]{ThS}, and we highlight it for future reference. We shall give a precise infinitesimal version of this fact in \cref{thm:infcomparison}.
\end{remark}

\subsection{Laminations}

We specify several special classes of geodesic laminations which play a crucial role in the Lipschitz theory of hyperbolic surfaces.

\begin{definition}[chain-recurrent laminations]\label{defn:chainrec}
A geodesic lamination $\lambda$ on a hyperbolic surface is said to be \emph{chain-recurrent} if for any point $x$ on  $\lambda$ and  for any $\epsilon >0$, there exists a  parametrised $C^1$ closed path $\delta$ on $S$ containing $x$ such that for any unit length path $I_1$ contained in $\delta$ there exists a unit length  path $I_2$ contained in $\lambda$ such that the two paths $I_1$ and $I_2$ are $\epsilon$-close in the $C^1$ topology. 
\end{definition}

We advise the reader to see \cite[p. 24-25]{ThS} and elsewhere in Thurston's paper for some basic properties of chain-recurrent laminations. Crucially, the support of any measured lamination is chain-recurrent. We note, however, that there is a small error in the statement of \cite[Lemma~8.3]{ThS}: it is \emph{not} true that any chain-recurrent lamination is approximated arbitrarily closely by simple closed curves. The correct statement is for chain-recurrent laminations that are \emph{connected}. This correction is supported by the second sentence on \cite[p.~38]{ThS}, where Thurston invokes the ``\emph{hypothesis that $\lambda$ is connected}''. The following is an immediate corollary to Thurston's Lemma~8.3 \cite{ThS}:

\begin{lemma}\label{lem:curveapprox}
For any chain-recurrent lamination $\lambda$ there exists a (simple closed) multicurve which approximates it arbitrarily closely in the Hausdorff topology.
\end{lemma}

\begin{definition}[maximal chain-recurrent laminations]
A chain-recurrent lamination $\lambda$ on $S$ is \emph{maximal} if there are no chain-recurrent laminations on $S$ properly containing  $\lambda$.
\end{definition}

\begin{remark}\label{rmk:uniqueextension}
Thurston states in the paragraph before Theorem~10.2 of \cite{ThS} that the complement of a maximal chain-recurrent lamination $\lambda$ consists of
\begin{itemize}
\item
ideal triangles and/or 
\item
once-punctured monogons (the latter arises only in the setting when $S$ has cusps), as well as 
\item
a single once-punctured bigon in the special case when $S$ is a one-cusped torus.\footnote{Thurston does not specify this last case, and it is a consequence of the hyperelliptic involution on the once-punctured tori. Interested readers may consult \cite{DLRT} or \cite[Remark~3.4]{HPap}.}
\end{itemize}
A careful reading of the proof of \cite[Theorem~8.5]{ThS} tells us that  complete geodesic laminations extending a specified maximal chain-recurrent lamination induce the same stretch path. This observation is highlighted in \cite[Corollary~2.3]{DLRT} and utilised in \cite{HPap}, and allows us to adopt the notation
\[
\mathrm{stretch}(x,\lambda,t)
\]
even when a maximal chain-recurrent lamination $\lambda$ is not complete (this arises precisely when $S$ has cusps).
\end{remark}

The set of compactly supported geodesic laminations on $S$ is a subset of the set of compact subsets of $S$, and hence naturally inherits the Hausdorff metric; indeed, it forms a compact metric space (see Chap.\ 4 of \cite{CEG}). Thurston shows that the set of chain-recurrent laminations forms a closed (and hence also compact) subset, and we use this fact to establish the following lemma:

\begin{lemma}\label{laminationlimit}
Given a sequence of chain-recurrent laminations $\{\lambda_i\}$ whose Hausdorff limit $\Lambda$ is a maximal chain-recurrent lamination, let $\{\Lambda_i\}$ be a sequence of chain-recurrent laminations such that $\lambda_i\subseteq\Lambda_i$. Then, the sequence $\{\Lambda_i\}$ also converges to $\Lambda$ in the Hausdorff topology.
\end{lemma}

\begin{proof}
\cite[Proposition~6.2]{ThS} tells us that the sequence $\{\Lambda_i\}$ has limit points, and our goal is show that every convergent sequence tends to $\Lambda$. 
Passing to a subsequence, we may assume without loss of generality that $\{\Lambda_i\}$ converges to some geodesic lamination $\Lambda_{\infty}$. By the definition of the Hausdorff topology, the limit has the form \[
\Lambda_{\infty}
=
\left\{\;
p\in (S,x)\mid 
\text{ there is a sequence }\{p_m\in\Lambda_i\}\text{ converging to }p
\;\right\},
\]
and by the closedness of the set of chain-recurrent laminations, $\Lambda_\infty$ is chain-recurrent.
This necessarily contains 
\[
\Lambda
=
\left\{\;
p\in (S,x)\mid 
\text{ there is a sequence }\{p_m\in\lambda_i\}\text{ converging to }p
\;\right\}.
\]
The maximality of $\Lambda$ amongst chain-recurrent laminations implies that $\Lambda_\infty=\Lambda$ and hence the Hausdorff limit of $\{\Lambda_i\}$ is indeed $\Lambda$.
\end{proof}

\begin{definition}[ratio-maximising lamination]
Given an ordered pair of distinct points $x,y\in\teich(S)$, we say that a geodesic lamination $\lambda$ is \emph{ratio-maximising} from $x$ to $y$ if 
\begin{itemize}
\item
there exists a Lipschitz homeomorphism, in the correct homotopy class as determined by the markings of $x,y\in \mathcal{T}(S)$, with optimal (\ie minimal) Lipschitz constant $e^{\mathrm{d}_{\mathrm{Th}}(x,y)}$ mapping from a neighbourhood of $\lambda$ in $(S,x)$ to a neighbourhood of $\lambda$ in $(S,y)$, and
\item
every such optimal Lipschitz constant homeomorphism takes the leaves of $\lambda$ in $(S,x)$ to the leaves of $\lambda$ in $(S,y)$ by multiplying arclength on $\lambda$ by $e^{\mathrm{d}_{\mathrm{Th}}(x,y)}$.
\end{itemize}
\end{definition}

Thurston unfortunately gives the incorrect definition for ratio-maximising laminations in the paragraph prior to \cite[Theorem~8.2]{ThS}. This is corrected in both \cite[Section~2.6]{DLRT} and \cite[Definition~4.2]{HPap}. In any case, the main motivation for defining ratio-maximising laminations is the central role they play in Thurston's concatenated stretch path construction for Thurston metric geodesics \cite[Theorem~8.5]{ThS}.

\begin{definition}[{maximal ratio-maximising chain-recurrent lamination, \cite[Theorem~8.2]{ThS}}]
For any ordered pair of distinct points $x,y\in\teich(S)$, there is a unique ratio-maximising chain-recurrent lamination which contains all other ratio-maximising chain-recurrent laminations from $x$ to $y$. This lamination  is called the \emph{maximal ratio-maximising chain-recurrent lamination}, and we denote it by $\mu(x,y)$.
\label{defn:maxratmax}
\end{definition}

The existence of such a lamination is a non-trivial result of Thurston. 

The maximal ratio-maximising chain-recurrent lamination is central to Thurston's construction of concatenated stretch path geodesics \cite[Theorem~8.5]{ThS}.
He there showed that the stretching locus of every stretch path constituting a geodesic from $x$ to $y$ always contains $\mu(x,y)$.

\begin{remark}[maximal ratio-maximising laminations]
Note that Thurston also refers to maximal ratio-maximising laminations as \emph{maximal maximally-stretched laminations} on two occasions in \cite{ThS}:
\begin{itemize}
\item
in the last paragraph of the proof of Theorem~8.5, and
\item
in the statement of Theorem~10.7.
\end{itemize}
Similar nomenclature is adopted in \cite[Section~2.6]{DLRT}.
\end{remark}

\section{Stretch maps and stretch vectors} \label{s:crown}

The goal of this section is to better understand the trajectory and behaviour of Thurston's stretch paths. We focus on the specific instance where $S$ is a complete hyperbolic surface which is topologically an annulus and has boundary cusps on both of its boundary components.
 We refer to such surfaces as \emph{$(n_L,n_R)$-crowned annuli} \cite{CB, Huang-Thesis}, where $n_L,n_R(\geq1)$ denote the number of boundary cusps on the ``left'', respectively ``right'' boundary component. Up to the action of the mapping class group relative to the boundary, any crowned annulus  admits only finitely many complete geodesic laminations, and each of the laminations contains finitely many leaves. 
From now on, we always assume that the leaves of a lamination do not all go from cusp to cusp. 

Let $S$ be an $(n_L, n_R)$-crowned annulus.
Let $\gamma_S$ be its core curve, which we take to be a closed geodesic when $x \in \teich(S)$ is given.
We fix orientations on $S$ and $\gamma_S$.
We number boundary cusps on the left of $\gamma_S$ by $p_1, \dots , p_{n_L}$ and those on the right by $p'_1, \dots, p'_{n_R}$.

For each point $x \in \teich(S)$, we define its {\em twist parameter} , denoted by $\tau_{\gamma_S}(x)$ as follows.
We fix an open simple arc $\sigma$ connecting $p_1$ with $p_1'$.
For $x \in \teich(S)$, we homotope $\sigma$ to a broken geodesic arc consisting of three geodesics:  the first one is a geodesic perpendicular from $p_1$ to $\gamma_S$, the second one is a geodesic immersed in $\gamma_S$, and the third one is a geodesic perpendicular from $p_1'$ to $\gamma_S$ (traversed in the opposite direction).
The twist parameter of $x \in \teich(S)$ is defined to be the signed length of the second immersed geodesic where the sign is positive when the second geodesic lies on the left of the first one.

Now we consider a lift of $(S,x)$ with its twist parameter $\tau_0$ to its universal cover embedded in $\hyperbolic$ regarded as the upper half-plane.
We lift $\gamma_S$ to a geodesic $\tilde\gamma_S$, which we assume to connect $0$ and $\infty$ so that the induced orientation is upward, and $\sigma$ to a geodesic $\tilde \sigma$ crossing $\tilde{\gamma}_S$.
Then we have the following immediately from the definition of the cross ratio.

\begin{lemma}
\label{cross ratio}
Let $a, b$ be the left and right endpoints of $\tilde \sigma$.
Then we have $\displaystyle-\frac{b}{a}=e^{\tau_0}.$
\end{lemma}

In the following argument, we consider a point $x \in \teich(S)$ with its twist parameter $\tau_0$ and look at how it is moved by the $K$-stretch map by computing the twist parameter after the $K$-stretch.
For that purpose, we choose a lift $\tilde\sigma$ so that the right endpoint at infinity is situated at $\displaystyle -e^{\frac{\tau_0}{2}}$, and hence the other endpoint at infinity is at $\displaystyle e^{\frac{\tau_0}{2}}$.

\subsection{Stretch maps for $(1,1)$-crowned annuli}\label{ss:stretch}

We begin with the case where $S$ is a \emph{$(1,1)$-crowned annulus}.
%
The Teichm\"uller space for $(1,1)$-crowned annuli is $\teich(S)\cong\mathbb{R}_{>0}\times\mathbb{R}$ (see \cite{Huang-Thesis}), where, for a given marked hyperbolic surface $x\in\teich(S)$, the first coordinate parametrises the length of $\gamma_S$ and the second is the twist parameter.
In this subsection, we explicitly write down the effects of all $K$-stretch maps for the $(1,1)$-annulus $S$. 
To do so, we first determine the collection of all complete geodesic laminations on $S$ containing $\gamma_S$ as a leaf.
Any such lamination consists of five geodesics: a geodesic starting from $p_1$ and spiralling around $\gamma_S$, which we denote by  $\alpha_1$, the closed geodesic $\gamma_S$,  a geodesic starting from $p_1'$ and spiralling around $\gamma_S$, which we denote by $\alpha_2$, and the two boundary components.
There are four possible laminations depending on the directions of spiralling of $\alpha_1$ and $\alpha_2$, as depicted in \cref{fig:2}.
We name them $\lambda_+, \lambda_-$ and $\lambda_0$, where 
we intentionally conflate the two laminations where $\alpha_1$ and $\alpha_2$ spiral toward $\gamma_S$ in the same direction and denote them by $\lambda_0$. 
%

\begin{figure}[h!]
\begin{center}
\includegraphics[scale=1.7]{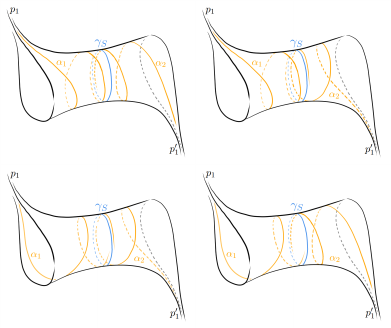}
\caption{The four possible configurations of a complete lamination composed of $\gamma_S$, $\alpha_1$ and $\alpha_2$. Case 1 is top left, case 2 is top right, case 3 is bottom left, case 4 is bottom right. We respectively denote the maximal laminations for cases~1 and 3 as $\lambda_-$ and $\lambda_+$. Cases 2 and 4 are those where the two geodesics $\alpha_1$ and $\alpha_2$ spiral around $\gamma_S$  in the same direction, and where we denote the resulting maximal lamination  by $\lambda_0$.}

\label{fig:2}
\end{center}
\end{figure} 


\begin{theorem}[stretching $(1,1)$-crowned annuli]
\label{thm:key}
Consider a hyperbolic metric $x\in\teich(S)$ on $S$ given by the Fenchel--Nielsen coordinates $(\ell_0,\tau_0)$ and let $K\geq 1$ be an arbitrary real number $\geq 1$. The $K$-stretch map with respect to $\lambda_0$ deforms the Fenchel--Nielsen coordinates $(\ell_0,\tau_0)$ to $(K\ell_0,K\tau_0)$, whereas the $K$-stretch map with respect to $\lambda_\pm$ takes $(\ell_0,\tau_0)$ to
\[
\left(K\ell_0,K\tau_0\pm2\left(\log(1-e^{-K\ell_0})-K\log(1-e^{-\ell_0})\right)\right).
\]
\end{theorem}

\begin{proof}
We first note that since $\gamma_S$ is in a leaf of the maximally stretched lamination, it  is stretched by a factor of $K$. 
Thus the length coordinate $\ell_0$ increases to $K\ell_0$ in all three cases.
Since the cases of $\lambda_-$ and $\lambda_+$ can be dealt with in the same way by just changing the figure by symmetry, we shall only consider the cases for $\lambda_0$ and $\lambda_+$.
In both cases, we lift the lamination to $\hyperbolic$ in such a way that $\gamma_S$ is lifted to the geodesic connecting $0$ and $\infty$, which we denote by $\tilde \gamma_S$.
By \cref{cross ratio}, the twist parameter is the logarithm of a cross-ratio between $0,\infty$ and the two endpoints of $\tilde{\sigma}$.
We shall compute the positions of the endpoints after applying the $K$-stretch map.

\begin{figure}[h!]
\begin{center}
\includegraphics[scale=0.55]{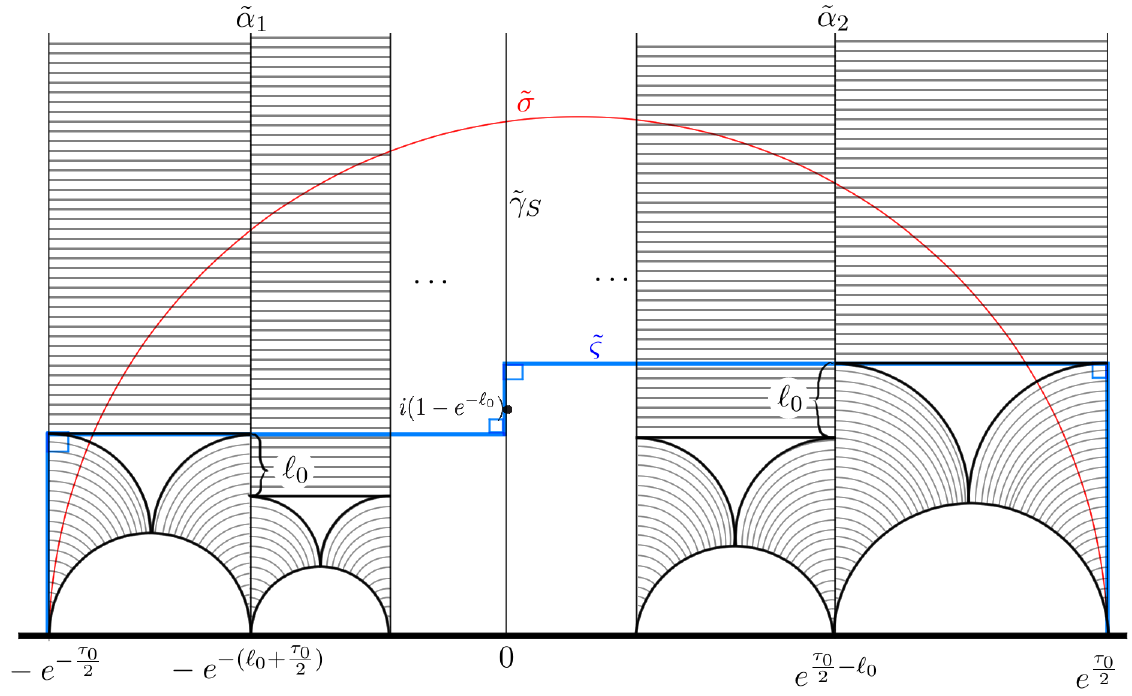}
\caption{A lift of $\lambda_0=\gamma_S \cup \alpha_1 \cup \alpha_2$ to the universal cover with specifications of $\tilde{\gamma}_S$, $\tilde{\sigma}$ and $\tilde{\varsigma}$. 
}

\label{fig:3}
\end{center}
\end{figure}

\noindent
\textbf{The $\lambda_0$ case:} We first calculate the twist parameter for the stretch map along $\lambda_0$. 
%
The universal cover of the hyperbolic surface $(S,x)$ is a contractible bordered domain embedded in the hyperbolic plane $\mathbb{H}$, which we identify with the upper half-plane as in \cref{fig:3}.
On each side of $\tilde{\gamma}_S$, we see a $\mathbb{Z}$-family of lifts of $\alpha_1$ and $\alpha_2$ which tend to $\infty$.
We mark each ideal triangle in the complement of the lifts of the $\alpha_i$ with its central stable triangle and note that the distance, along a lift $\tilde{\alpha}_i$, of the anchors for any two adjacent ideal triangles is $\ell_0$.
Fix an arbitrary lift $\tilde{\sigma}$ of the infinite arc $\sigma$ connecting $p_1$ and $p_1'$ we fixed before. 
Since the initial twist parameter is $\tau_0$, the ideal points for $\tilde{\sigma}$ are $-e^{-\frac{\tau_0}{2}}$ and $e^{\frac{\tau_0}{2}}$ for some convenient parametrisation of the boundary at infinity by $\mathbb{R}\cup\{\infty\}$. 
The infinite arc $\sigma$ can be homotoped to a piecewise smooth arc as follows.
Consider a half-infinite arc $\varsigma_1$ starting from $p_1$ that lies on $\alpha_1$ until it reaches an anchor for the first time, and then travels on  a leaf of the horocycle foliation until it reaches $\gamma_S$. 
Likewise we construct an infinite arc $\varsigma_2$ starting from $p_1'$.
We join the endpoints of $\varsigma_1$ and $\varsigma_2$ by an arc immersed in $\gamma_S$ so that we get an infinite arc properly homotopic to $\sigma$.
We let $\tilde \varsigma$ be a lift of $\varsigma$ which has the same endpoints at infinity as $\tilde \sigma$.
We use this $\tilde \varsigma$ to compute how the endpoints of $\tilde \sigma$ are moved after $K$-stretching.

%
%

The $K$-stretch map acts on $\varsigma$ by uniformly stretching the geodesic segments shown in the figure by the factor $K$ and non-uniformly shrinking the two horocylic segments, each joining $\alpha_i$ and $\gamma_S$. 
Each of these horocyclic segments consists of countably infinite many horocyclic segments joining two lifts of $\alpha_1$ or $\alpha_2$. An easy calculation shows that both segments have length equal to
\begin{align}
1+e^{-\ell_0}+e^{-2\ell_0}+e^{-3\ell_0}+\ldots=(1-e^{-\ell_0})^{-1}.\label{eq:horo}
\end{align}
We note that this horocyclic length only dependens on $\ell_0$, the length of $\gamma_S$ (see the statement of \cref{thm:key}).

Since the leftmost ideal triangle in \cref{fig:3} has vertices at $\displaystyle -e^{-\frac{\tau_0}{2}}, -e^{-(\ell_0+\frac{\tau_0}{2})}$ and $\infty$,  the first left horocyclic segment is positioned at height $e^{\frac{-\tau_0}{2}}(1-e^{-\ell_0})$, and in the same way  the first right horocyclic segment is positioned at height $e^{\frac{\tau_0}{2}}(1-e^{-\ell_0})$. 
This implies that the midpoint of the vertical geodesic segment in the middle of $\tilde{\varsigma}$ corresponds to $i(1-e^{-\ell_0})$.

After applying the $K$-stretch map, the new metric deforms the surface in such a way that the geodesic paths $\alpha_1$ and $\alpha_2$ remain  geodesics which spiral to $\gamma_S$, and the transverse horocyclic foliation is preserved. 
We lift the $K$-stretch map to the universal cover so that $0$, $\infty$, and  $i(1-e^{-\ell_0})$ are fixed.
Then, the configuration depicted in \cref{fig:3} still holds, but the endpoints of $\tilde \alpha_1$ and $\tilde\alpha_2$, which are $e^{\frac{\tau_0}{2}}$ and $e^{-\frac{\tau_0}{2}}$ in the original figure, will be moved to other points, which we denote by $u_1$ and $u_2$. 

We now compute $u_1$ and $u_2$.
First observe that since the metric along $\tilde{\gamma}_S$ is stretched by a factor of $K$,   the two ends of the vertical segment in $\tilde \varsigma$ are 
$\displaystyle ie^{\frac{-K\tau_0}{2}}(1-e^{-\ell_0})$
and $\displaystyle ie^{\frac{K\tau_0}{2}}(1-e^{-\ell_0})$.
Recall from above that each composite horocyclic segment is made up of infinitely many horocyclic segments, each contained in a distinct lift of one of the two ideal triangles in $S-\lambda_0$. The total length of each of the new composite horocyclic segments is given by
\begin{align}
1+e^{-K\ell_0}+e^{-2K\ell_0}+e^{-3K\ell_0}+\ldots=(1-e^{-K\ell_0})^{-1}.\label{eq:transformedhoro}
\end{align}
For the left one, since the segment is at height $e^{\frac{-K\tau_0}{2}}(1-e^{-\ell_0})$, its vertical projection onto the real axis has length equal to $\displaystyle\frac{e^{\frac{-K\tau_0}{2}}(1-e^{-\ell_0})}{(1-e^{-K\ell_0})}$, which implies that $\displaystyle u_1=-\frac{e^{\frac{-K\tau_0}{2}}(1-e^{-\ell_0})}{(1-e^{-K\ell_0})}$.
By the same argument, we have $\displaystyle u_2=\frac{e^{\frac{K\tau_0}{2}}(1-e^{-\ell_0})}{(1-e^{-K\ell_0})}$.

%
%
%
%
Having obtained $u_1$ and $u_2$, which are the endpoints of  the $K$-stretched image of $\tilde\sigma$, we see by \cref{cross ratio}  that the twist parameter after the $K$-stretch is equal to $\displaystyle\log|\frac{u_2}{u_1}|=K\tau_0$, as desired.

%
\medskip
\noindent
\textbf{The $\lambda_+$ case}: we repeat the same strategy as for $\lambda_0$.
In particular, we again set $\tilde{\gamma}_S$ to be the geodesic joining $0$ and $\infty$, set $\tilde{\sigma}$ as the geodesic joining $-e^{-\frac{\tau_0}{2}}$ and $e^{\frac{\tau_0}{2}}$, and we lift $\varsigma$ to $\tilde \varsigma$ sharing the two endpoints at infinity with $\tilde \sigma$. 
We know immediately from the computation in the $\lambda_0$ case  that the horocyclic subsegment of $\tilde{\varsigma}$ to the left of $\tilde{\gamma}_S$ is placed at height $e^{-\frac{\tau_0}{2}}(1-e^{-\ell_0})$. 
Observe that the hyperbolic involution $z\mapsto -z^{-1}$ preserves both $\tilde{\gamma}_S$ and $\tilde{\varsigma}$, and permutes the two horocyclic subsegments of $\tilde{\varsigma}$. This immediately tells us that the horocyclic segment to the right of $\tilde{\gamma}_S$ lies on the Euclidean circle passing through $0$ and $ie^{\frac{\tau_0}{2}}(1-e^{-\ell_0})^{-1}$ and tangent to $0$. Furthermore, observe that the hyperbolic length of the left composite horocyclic segment is still given by $(1-e^{-\ell_0})^{-1}$, and hence, by the $z\mapsto -z^{-1}$ involution. This is also the length of the right composite horocyclic segment. Furthermore, these two composite horocyclic segments are both $K$-stretched to horocyclic segments of hyperbolic length $(1-e^{-K\ell_0})^{-1}$.\medskip

\begin{figure}[h!]
\begin{center}
\includegraphics[scale=0.55]{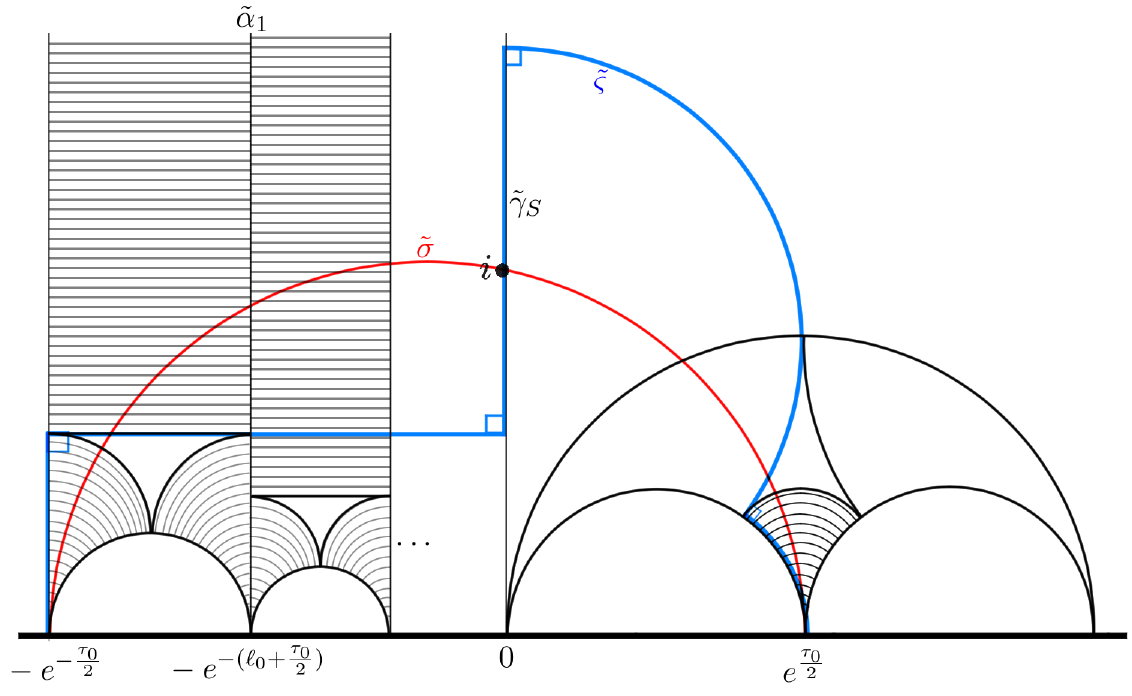}
\caption{A lift of $\lambda_+$ to the universal cover with specifications of $\tilde{\gamma}_S$, $\tilde{\sigma}$ and $\tilde{\varsigma}$.}
\label{fig:4}
\end{center}
\end{figure}

We now compute the twist parameter for $\lambda_+$ after applying the $K$-stretch map  (see \cref{fig:4}). 
We lift the $K$-stretch map to $\hyperbolic$ so that it fixes $0, \infty$ and $i \in \hyperbolic$ and so that the image is invariant under $z\mapsto -z^{-1}$.
Our computations for $\lambda_0$ show that the height of the lower endpoint is $e^{-\frac{\tau_0}{2}}(1-e^{-\ell_0})$, hence the distance between this end and the mid-point $i$ is $\log(e^{\frac{\tau_0}{2}}(1-e^{-\ell_0})^{-1})=\tfrac{\tau_0}{2}-\log(1-e^{-\ell_0})$. 
By symmetry, the same holds for the distance between the upper endpoint and $i$.
The $K$-stretch distorts distances along $\tilde{\gamma}_S$ by a factor of $K$, hence the distance between $i$ and the new endpoints of the central geodesic subsegment of $\tilde{\varsigma}$ is $\tfrac{K\tau_0}{2}-K\log(1-e^{-\ell_0})$. Therefore, the new endpoints of the $K$-stretched central (vertical) geodesic segment of the $K$-stretched $\tilde{\varsigma}$ are at 
$\displaystyle
ie^{\frac{-K\tau_0}{2}}(1-e^{-\ell_0})^K\text{ and }ie^{\frac{K\tau_0}{2}}(1-e^{-\ell_0})^{-K}\in\mathbb{H}.$

The same computation as in the preceding case shows that 
$\displaystyle
u_1=-\frac{e^{\frac{-K\tau_0}{2}}(1-e^{-\ell_0})^{K}}{(1-e^{-K\ell_0})}.$
To determine the position of $u_2$, we again apply the $z\mapsto -z^{-1}$ involution and see that
$\displaystyle
u_2=-(u_1)^{-1}=\frac{e^{\frac{K\tau_0}{2}}(1-e^{-K\ell_0})}{(1-e^{-\ell_0})^{K}}.$
Thus we now know the two endpoints at infinity of the $K$-stretched image of $\tilde{\sigma}$, and hence by \cref{cross ratio}, we see that the new twist coordinate is 
\begin{align*}
\log\left|\tfrac{u_2}{u_1}\right|
&=2\log\left(
\frac{e^{\frac{K\tau_0}{2}}(1-e^{-K\ell_0})}{(1-e^{-\ell_0})^{K}}
\right)\\
&=K\tau_0+2\left(\log(1-e^{-K\ell_0})-K\log(1-e^{-\ell_0})\right),
\end{align*}
as claimed.
%
\end{proof}

\subsection{Stretch maps for general crowned annuli}
We now seek to understand stretch maps on an $(n_L,n_R)$-crowned annulus $S$, \ie an annulus with $n_L,n_R>0$ boundary cusps to the left and the right of the core curve $\gamma_S$,  endowed with a complete finite-area geodesic-bordered hyperbolic metric \cite{CB, Huang-Thesis}. 
The geodesic $\gamma_S$ cuts $S$ into two annuli, and we refer to the annulus containing the $n_L$ boundary cusps as the left component, and the and the one with the $n_R$ boundary cusps as the right component.

Let $\Lambda$ be a complete lamination containing $\gamma_S$ as a leaf.
Then no leaves of $\Lambda$ pass from the left component of $S\setminus \gamma_S$ to the right. 
Since both the left and right components are topologically annuli, the left component of $S\setminus \gamma_S$ contains precisely $n_L$ leaves (not including $\gamma_S$), and the right contains $n_R$ leaves. 
Since each component of $S\setminus \Lambda$ is an ideal triangle, and since an ideal triangle is unique up to isometry, the isometry type of $S$ is determined by how they are pasted together.

For two ideal triangles $\Delta_1$ and $\Delta_2$ on $S$ sharing a geodesic $\ell$, the {\em shearing parameter} between $\Delta_1$ and $\Delta_2$ is defined to be the signed distance between an anchor of $\Delta_1$ and an anchor of $\Delta_2$ on $\ell$, where the sign is positive if the anchor of $\Delta_2$ is to the left of that of $\Delta_1$.

Now we can parametrise the Teichm\"uller space $\teich(S)$ of $S$ with
\begin{itemize}
\item
$n_L$ shearing parameters, one for each of the $n_L$ leaves on the left component of $S\setminus\gamma_S$ with the linear constraint that the shearing parameters which spiral around $\gamma_S$ sum to $\ell_{\gamma_S}$;

\item
$n_R$ shearing parameters, one for each of the $n_R$ leaves on the right component of $S\setminus\gamma_S$ with the linear constraint that the shearing parameters which spiral around $\gamma_S$ sum to $\ell_{\gamma_S}$; and

\item
one Fenchel--Nielsen length parameter $\ell_{\gamma_S}$ and one twist parameter $\tau_{\gamma_S}$ for $\gamma_{S}$.
\end{itemize}

The $K$-stretch map along $\Lambda$ multiplies all the shearing parameters and the Fenchel--Nielsen length parameter $\ell_{\gamma_S}$ by the factor $K$ (the two linear constraints are still satisfied). 
Thus, we need only determine what happens to the twist parameter to determine the stretch path given by the $K$-stretch map along $\Lambda$.
Since only leaves spiralling around $\gamma_S$ can actually affect $\tau_{\gamma_S}$,  we can determine $\tau_{\gamma_S}$ completely from the subsurface of $S$ which consists of the convex hull of the leaves of $\Lambda$ spiralling around $\gamma_S$. 
Therefore, by possibly reducing to a smaller crowned annulus, we need only determine $\tau_{\gamma_S}$ for stretch maps with respect to $\Lambda$ which consist of $\gamma_S$ and $n_L+n_R$ bi-infinite leaves spiralling around $\gamma_S$. 
We need only consider four possible cases for $\Lambda$: there are two cases where the left and right leaves spiral in the same direction toward $\gamma_S$ (like for $\lambda_0$ in \cref{fig:2}), and there are two cases where the left and right leaves spiral in opposite directions toward $\gamma_S$ (like for $\lambda_\pm$ in \cref{fig:2}). 

\subsection{$(n,0)$-crowned annulus}
Before starting to compute the changing of twisting parameter by $K$-stretch maps, we establish a useful lemma for a single component of $S\setminus \gamma_S$, \ie an $(n,0)$-crowned annulus.
%
%
%

Consider an $(n,0)$-crowned annulus $A$ for $n \in \naturals$, \ie a complete finite-area geodesic bordered hyperbolic surface which is topologically an annulus with $n$ boundary cusps on one of its boundaries. Denote the closed geodesic boundary of $A$ by $\alpha$.
We fix an orientation on $A$, and give a direction of $\alpha$ compatible with the given orientation.
 We consider a complete geodesic lamination $\mu$ on $A$ consisting of $\alpha$ and $n$ bi-infinite geodesic leaves which spiral towards $\alpha$ in the the direction of the orientation of $\alpha$. 
 See \cref{fig:19_2} for a depiction of the universal cover of $A$, which lies in the upper half plane $\hyperbolic$.
 We lift $\alpha$ to the geodesic connecting $0$ and $\infty$.

\begin{figure}[h!]
\begin{center}
\includegraphics[scale=0.55]{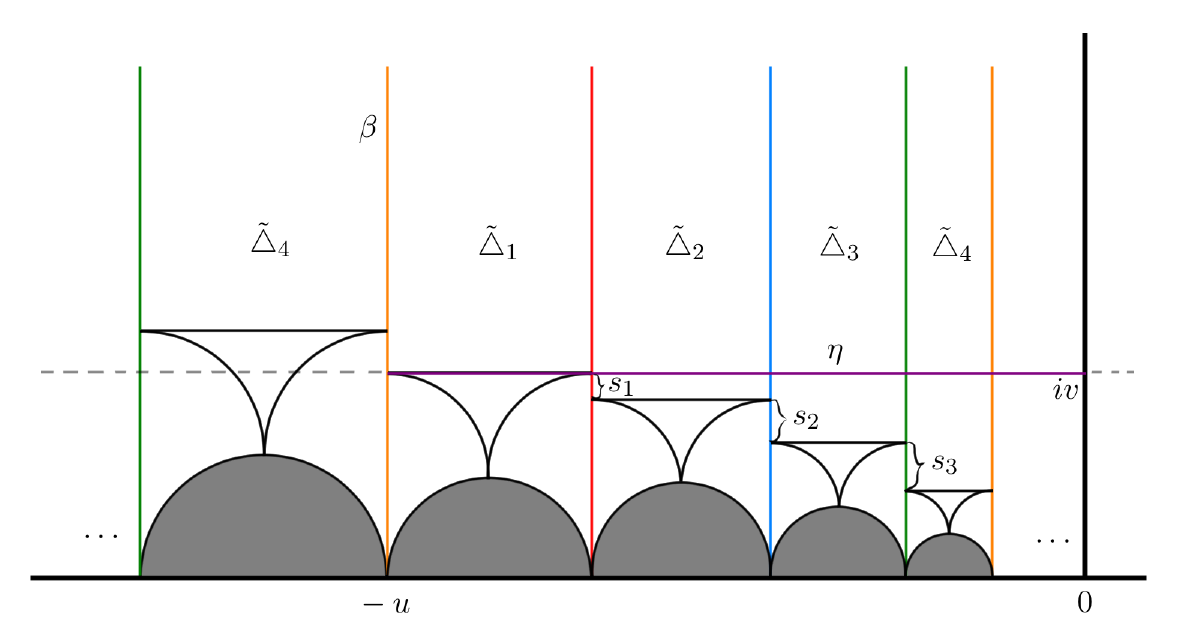}
\caption{The universal cover for an $(4,0)$-crowned annulus.
This is a picture when all the $s_i$ are negative.}
\label{fig:19_2}
\end{center}
\end{figure}

We sequentially label the ideal triangles in $A \setminus \mu$ by $\triangle_1,\ldots,\triangle_{n}$ so that $\triangle_{j-1}$ and $\triangle_j$ have sides with endpoint at $p_j$.  \cref{fig:19_2} represents the situation in the universal cover. Let $s_i\in\mathbb{R}$ denote the shearing parameter for the edge shared by $\triangle_i$ and $\triangle_{i+1}$ (cyclically indexed so that $\triangle_1=\triangle_{n+1}$). 
Take an arbitrary lift $\beta$ of the ideal edge shared by $\triangle_n$ and $\triangle_{1}$ and let $-u$ denote the endpoint at infinity of $\beta$ other than $\infty$. 
We consider juxtaposed lifts of the triangle $\triangle_1, \dots, \triangle_n$ so that the lift of $\triangle_1$ contains $\beta$, and denote them by $\tilde \triangle_1, \dots , \tilde \triangle_n$.
Let $\tilde \triangle_n'$ be the lift of $\triangle_n$ lying on the other side of $\beta$ from $\tilde \triangle_1$.
We consider a horocyclic arc on $\tilde{\triangle}_1$ centred at $\infty$  connecting two anchors, and extend it to the entire horocycle  tangent to $\infty$, which we denote by $\eta$. 
The intersection of $\beta$ and $\eta$ is expressed as $-u+iv$ with some $v >0$.
Then, the (hyperbolic) length of the horocyclic subsegment of $\eta$ which traverses from $-u+iv$ to $iv$ is $\tfrac{u}{v}$. We have the following:

\begin{lemma}
\label{lem:component}
With the above notation, we have the following relation:
\begin{align}
\frac{u}{v}
=
\frac{1+e^{-s_1}+e^{-(s_1+s_2)}+\cdots+e^{-(s_1+s_2+\cdots+s_{n-1})}}
{1-e^{-\ell_{\alpha}}}.
\end{align}
\end{lemma}

\begin{proof}
The horocyclic subsegment of $\eta$ which traverses from $-u+iv$ to $iv$ and whose length is $\tfrac{u}{v}$ is composed of countably many horocyclic segments each of which is contained in a lift of one of the $\triangle_1, \dots , \triangle_n$. The first of these segments is a boundary of the central stable triangle of $\tilde \triangle_1$, and hence has length $1$. 
The second horocyclic segment, which is on $\tilde \triangle_2$, is at distance $s_1$ away from the horocyclic segment which constitutes the boundary of the central stable triangle of $\triangle_2$ and hence has length $e^{-s_1}$. 
We continue in this manner until we get through all $n$ ideal triangles and return to $\triangle_1$, at which point we now know that the $(n+1)$-th horocyclic segment is at distance $\ell_\alpha$ away from the original length $1$ horocyclic segment, and hence is of length $e^{-\ell_\alpha}$. Similarly, the $(n+2)$-th horocyclic segment, which lies on $\triangle_2$, is of length $e^{-(\ell_\alpha+s_1)}$, and so forth. Therefore we have 
\begin{align*}
\frac{u}{v}
=&
1+e^{-s_1}+e^{-(s_1+s_2)}+\cdots+e^{-(s_1+s_2+\cdots+s_{n-1})}\\
&+e^{-\ell_\alpha}+e^{-(\ell_\alpha+s_1)}+e^{-(\ell_\alpha+s_1+2)}+e^{-(\ell_\alpha+s_1+s_2+\cdots+s_{n-1})}\\
&+e^{-2\ell_\alpha}+e^{-(2\ell_\alpha+s_1)}+e^{-(2\ell_\alpha+s_1+2)}+e^{-(2\ell_\alpha+s_1+s_2+\cdots+s_{n-1})}+\ldots\\
=&
\left(1+e^{-s_1}+\cdots+e^{-(s_1+s_2+\cdots+s_{n-1})}\right)
\left(1+e^{-\ell_\alpha}+e^{-2\ell_\alpha}+\cdots\right)\\
=&
\frac{1+e^{-s_1}+e^{-(s_1+s_2)}+\cdots+e^{-(s_1+s_2+\cdots+s_{n-1})}}{1-e^{-\ell_{\alpha}}}
,
\end{align*}
as desired.
\end{proof}

\subsubsection{Parallel spiralling case}
We now move on to working with stretch maps of the $(n_L, n_R)$-crowned annulus $S$.
We first consider the case where all bi-infinite leaves of the maximally stretched lamination spiral around $\gamma_S$ in the same direction. Let $\mu_+$ denote the complete lamination on $S$ containing, as its leaves (see \cref{fig:20_1}),
the closed geodesic $\gamma_S$,
$n_L+n_R$ geodesics going from the boundary cusps of $S$ and spiralling towards $\gamma_S$ in the direction of $\gamma_S$,
and
the boundary geodesics of $S$.

\begin{figure}[h!]
\begin{center}
\includegraphics[scale=0.35]{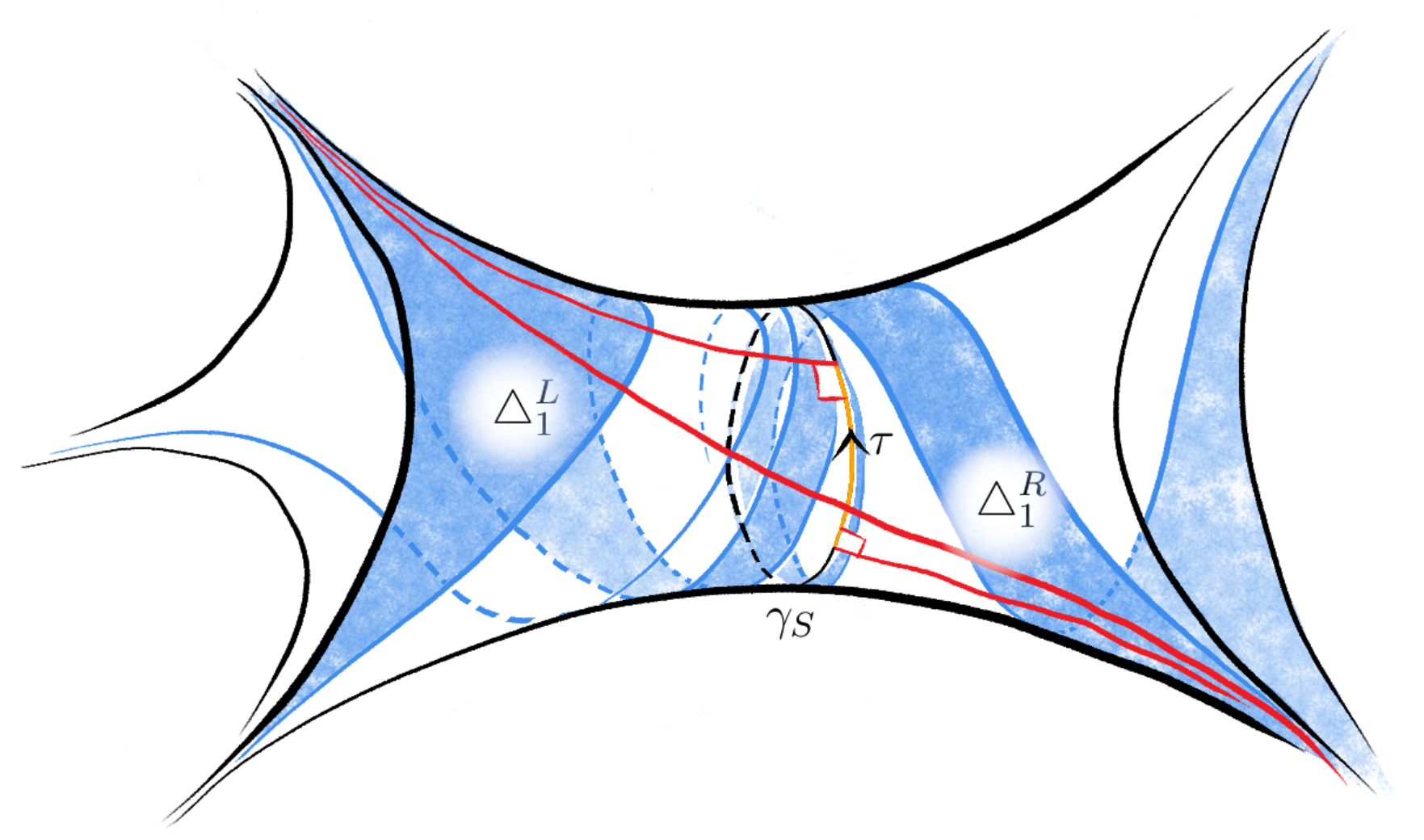}
\caption{A $(3,2)$-crowned annulus with $\tau$ coordinate given by the signed length of the orange geodesic arc.}
\label{fig:20_1}
\end{center}
\end{figure}

We label the ideal triangles on the left component by $\triangle^L_1,\ldots,\triangle^L_{n_L}$ and those on the right by $\triangle^R_1,\ldots,\triangle^R_{n_R}$ ordered as in \cref{fig:20_2}. Let $s^L_i\in\mathbb{R}$ denote the shearing parameter for the edge shared by $\triangle^L_i$ and $\triangle^L_{i+1}$ (cyclically indexed), and likewise let $s^R_i\in\mathbb{R}$ denote the shearing parameter for the edge shared by $\triangle^R_i$ and $\triangle^R_{i+1}$. 
Then we have the following.

\begin{figure}[h!]
\begin{center}
\includegraphics[scale=0.55]{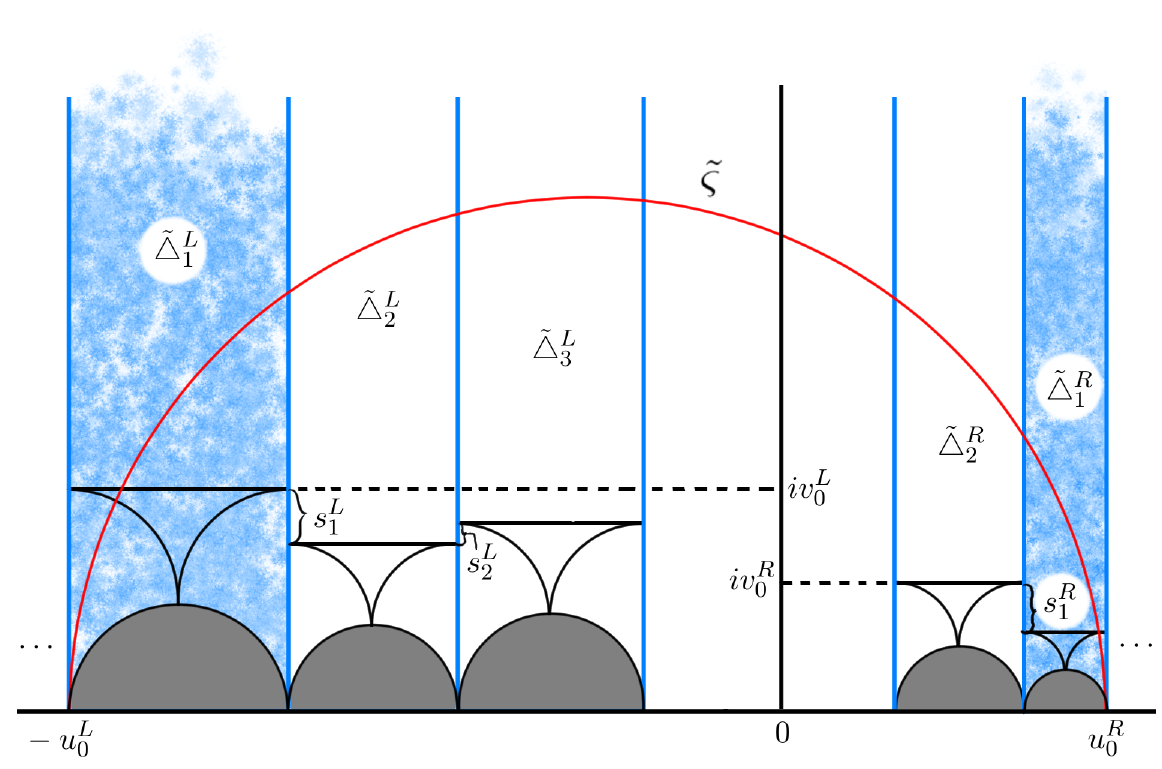}
\caption{The universal cover of the $(3,2)$-annulus in \cref{fig:20_1}.}
\label{fig:20_2}
\end{center}
\end{figure}


\begin{lemma}\label{lem:pregencrown}
Let $(S,x)$ be an $(n_L, n_R)$-crowned annulus.
Then,
\begin{align*}
&\tau_{\gamma_S}(\mathrm{stretch}(x,\mu_+,t))\\
&=e^t\tau_{\gamma_S}(x)\\
&+e^t\log\left(1+e^{-s^L_1(x)}+
\cdots+e^{-(s^L_1(x)+\ldots+s^L_{n_L-1}(x))}\right)\\
&-\log\left(1+e^{-e^ts^L_1(x)}+
\cdots+e^{-e^t(s^L_1(x)+\ldots+s^L_{n_L-1}(x))}\right)\\
&-e^t\log\left(1+e^{-s^R_1(x)}+
\cdots+e^{-(s^R_1(x)+\ldots+s^R_{n_R-1}(x))}\right)\\
&+\log\left(1+e^{-e^ts^R_1(x)}+
\cdots+e^{-e^t(s^R_1(x)+\ldots+s^R_{n_R-1}(x))}\right),
\end{align*}
where $\mathrm{stretch}(x,\mu_+,t)$ denotes the hyperbolic structure obtained by the $e^t$-stretching along $\mu_+$ from $x$.
\end{lemma}

\begin{proof}
Let $\tau_0$ denote the twist parameter $\tau_{\gamma_S}(x)$, and let $\tau_t$ denote the twist parameter $\tau_{\gamma_S}(\mathrm{stretch}(x,\mu_+,t))$ after the $e^t$-stretching. 
We lift ideal triangles $\triangle_1^L, \dots , \triangle_{n_L}^L$ to juxtaposed ideal triangles $\tilde \triangle_1^L, \dots , \tilde \triangle_{n_L}^L$, and $\triangle^R_1, \dots , \triangle^R_{n_R}$ to $\tilde \triangle^R_1, \dots \tilde \triangle^R_{n_R}$.
We also denote the left ideal vertex on $\reals$ of $\triangle_1^L$ by $u_0^L$ and the right ideal vertex on $\reals$ of $\triangle_1^R$ by $u^R_0$.
The geodesic $\sigma$ is lifted to a geodesic $\tilde \varsigma$ connecting $u_0^L$ to $u_0^R$.
The anchor on the side of $\triangle_1^L$ with endpoints $u_0^L, \infty$ has a form $-u_0^L+iv_0^L$, and in the same way, the anchor on the side of $\triangle_1^R$ with endpoints $u_0^R, \infty$ has a form $u_0^R+iv_0^R$.
After the $e^t$-stretch the vertices $-u_0^L$ and $u_0^R$ are moved to vertices at infinity which we respectively denote by $-u_t^L$ and $u_t^R$. See \cref{fig:20_2}. Likewise, the points $iv_0^L$ and $iv_0^R$ are respectively moved to $iv_t^L$ and $iv_t^R$. Then, we have the following equality by \cref{cross ratio}:
\begin{align}
\tau_0=\log\left(\tfrac{u^R_0}{u^L_0}\right)
\quad\text{ and }\quad
\tau_t=\log\left(\tfrac{u^R_t}{u^L_t}\right).
\end{align}
Our goal is to express $\tau_t$ in terms of $\tau_0$ and other $\teich(S)$ coordinates. From \cref{lem:component}, we obtain
\begin{equation}
\begin{split}
\tau_0
=&\log\left(
\tfrac{u^R_0}{v^R_0}
\cdot\tfrac{v^L_0}{u^L_0}
\cdot\tfrac{v^R_0}{v^L_0}\right)
=
\log\left(\tfrac{u^R_0}{v^R_0}\right)
-
\log\left(\tfrac{u^L_0}{v^L_0}\right)
+\log\left(\tfrac{v^R_0}{v^L_0}\right)
\\
=&\log\left(\frac{v^R_0}{v^L_0}\right)
-\log\left(\frac{1+e^{-s^L_1(x)}+
\cdots+e^{-(s^L_1(x)+s^L_2(x)+\cdots+s^L_{n_L-1}(x))}}
{1-e^{-\ell_{\gamma_S}(x)}}\right)\\
&+\log\left(\frac{1+e^{-s^R_1(x)}+
\cdots+s^R_{n_R-1}(x))}
{1-e^{-\ell_{\gamma_S}(x)}}\right).
\label{eq:tau0}
\end{split}
\end{equation}
The fact that the last term of \cref{eq:tau0} is equal to $\log\left(\tfrac{u^R_0}{v^R_0}\right)$ follows from \cref{lem:component} by applying the reflection isometry with respect to the imaginary axis in $\mathbb{H}$.\medskip

We draw particular attention to the term $\log\left(\frac{v^R_0}{v^L_0}\right)$, which measures the length of a geodesic segment  that traverses over $\gamma_S$. After $e^t$-stretching, the length of this segment must therefore be $\log\left(\tfrac{v^R_t}{v^L_t}\right)=e^t\log\left(\tfrac{v^R_0}{v^L_0}\right)$. Substituting this into the $e^t$-stretched version of \cref{eq:tau0} yields
\begin{equation}
\begin{split}
\tau_t
=&
\log\left(
\tfrac{u^R_t}{v^R_t}
\cdot\tfrac{v^L_t}{u^L_t}
\cdot\tfrac{v^R_t}{v^L_t}\right)
=
\log\left(\tfrac{u^R_t}{v^R_t}\right)
-
\log\left(\tfrac{u^L_t}{v^L_t}\right)
+\log\left(\tfrac{v^R_t}{v^L_t}\right)
\notag\\
=&e^t\log\left(\frac{v^R_0}{v^L_0}\right)-\log\left(\frac{1+e^{-e^ts^L_1(x)}+
\cdots+e^{-e^t(s^L_1(x)+s^L_2(x)+\cdots+s^L_{n_L-1}(x))}}
{1-e^{-e^t\ell_{\gamma_S}(x)}}\right)
\\
&+\log\left(\frac{1+e^{-e^ts^R_1(x)}+
\cdots+e^{-e^t(s^R_1(x)+s^R_2(x)+\cdots+s^R_{n_R-1}(x))}}
{1-e^{-e^t\ell_{\gamma_S}(x)}}\right).\notag
\end{split}
\end{equation}
Using \cref{eq:tau0} to replace the $\log\left(\tfrac{v^R_0}{v^L_0}\right)$ term here with expressions in $\tau_0$, $\ell_{\gamma_S}(x)$ and shearing parameters $s^L_i(x),\ s^R_j(x)$,  we obtain the desired result.
\end{proof}

\begin{remark}

Similarly, we define $\mu_-$ to be a geodesic lamination with the $n_L+n_R$ bi-infinite geodesic spiralling towards $\gamma_S$ in the direction of $\gamma_S^{-1}$. By symmetry, the relevant Dehn-twist parameter $\tau_{\gamma_S}$ satisfies:
\begin{align*}
&\tau_{\gamma_S}(\mathrm{stretch}(x,\mu_-,t))\\
&=e^t\tau_{\gamma_S}(x)-e^t\log\left(1+e^{-s^L_1(x)}+
\cdots+e^{-(s^L_1(x)+\ldots+s^L_{n_L-1}(x))}\right)\\
&+\log\left(1+e^{-e^ts^L_1(x)}+
\cdots+e^{-e^t(s^L_1(x)+\cdots+s^L_{n_L-1}(x))}\right)\\
&+e^t\log\left(1+e^{-s^R_1(x)}+
\cdots+e^{-(s^R_1(x)+\cdots+s^R_{n_R-1}(x))}\right)\\
&-\log\left(1+e^{-e^ts^R_1(x)}+
\cdots+e^{-e^t(s^R_1(x)+\ldots+s^R_{n_R-1}(x))}\right).
\end{align*}
However, we should emphasise that in changing the lamination from $\mu_+$ to $\mu_-$, the shearing parameters $s^L_i$ and $s^R_j$ are different from the identically denoted variables for $\mu_+$. In particular, it is \emph{generically false} that 
\[
\tau_{\gamma_S}(\mathrm{stretch}(x,\mu_+,t))+\tau_{\gamma_S}(\mathrm{stretch}(x,\mu_-,t))=2e^t\tau_{\gamma_S}(x).
\]
\end{remark}
\medskip

\subsubsection{Opposite spiralling case}
We now consider the case when the stretching locus is a complete geodesic lamination where the leaves on the left and right components of $S\setminus \gamma_S$ spiral toward $\gamma_S$ in opposite directions.
Let $\lambda_+$ denote a complete lamination on $S$ consisting of 
the closed geodesic $\gamma_S$,
$n_L+n_R$ geodesics going from the boundary cusps of $S$ and spiralling towards $\gamma_S$ in such a way that those on the left component spiral in the direction of $\gamma_S$, and those on the right component spiral in the direction of $\gamma_S^{-1}$, and
 the boundary geodesics on $S$.
See \cref{fig:23_1}.

\begin{figure}[h!]
\begin{center}
\includegraphics[scale=0.35]{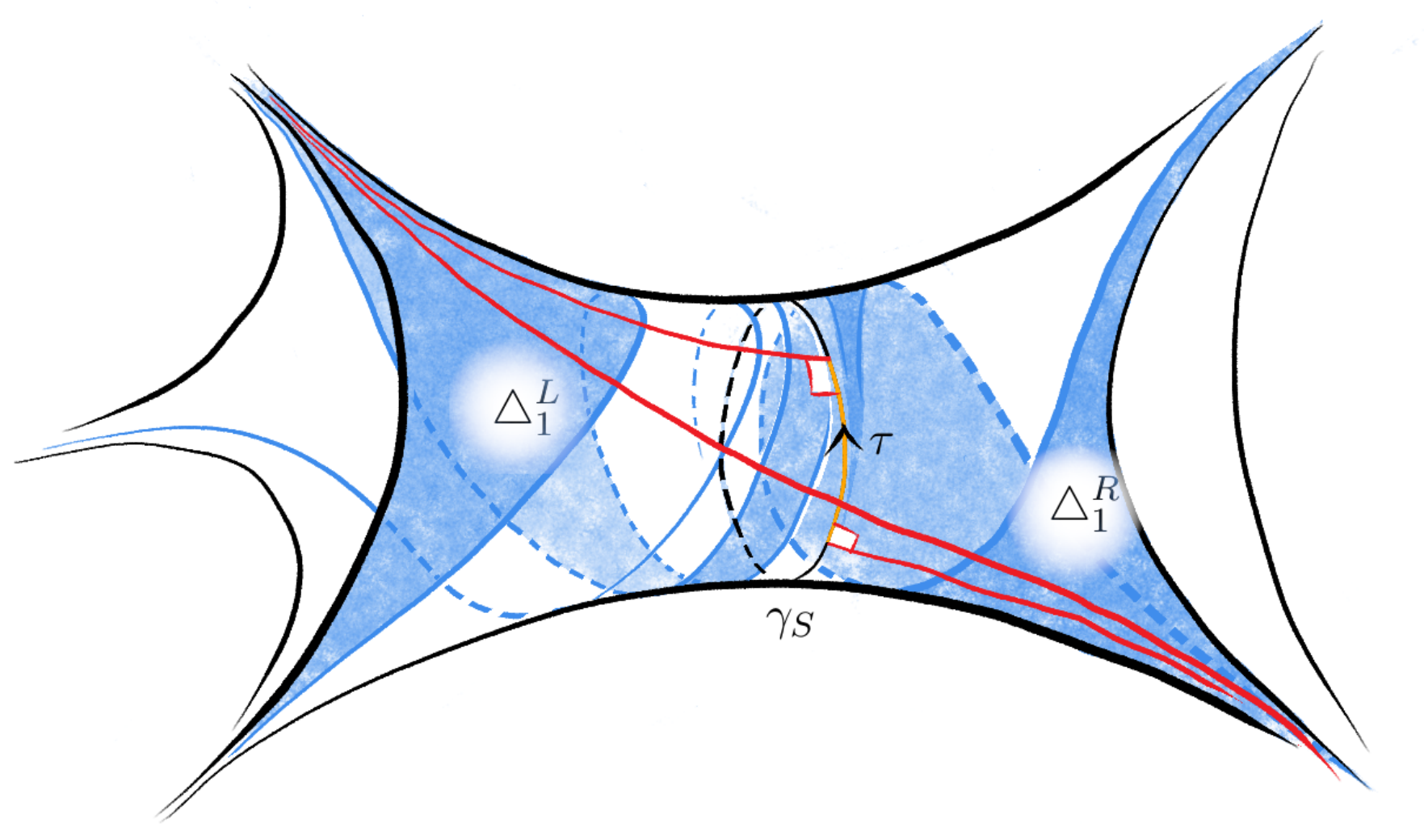}
\caption{A $(3,2)$-crowned annulus with $\tau$ coordinate given by the signed length of the orange geodesic arc. }
\label{fig:23_1}
\end{center}
\end{figure}

We label the ideal triangles on the left component by $\triangle^L_1,\ldots,\triangle^L_{n_L}$ and those on the right by $\triangle^R_1,\ldots,\triangle^R_{n_R}$ ordered as in \cref{fig:23_2}, \ie  $\triangle_{i+1}^L$ (resp. $\triangle^R_{i+1}$) lies on the right (resp. the left) of $\triangle^L_i$ (resp. $\triangle^R_i$). Let $s^L_i\in\mathbb{R}$ denote the shearing parameter for the edge shared by $\triangle^L_i$ and $\triangle^L_{i+1}$ (cyclically indexed) and likewise let $s^R_i\in\mathbb{R}$ denote the shearing parameter for the edge shared by $\triangle^R_i$ and $\triangle^R_{i+1}$. 

\begin{figure}[h!]
\begin{center}
\includegraphics[scale=0.55]{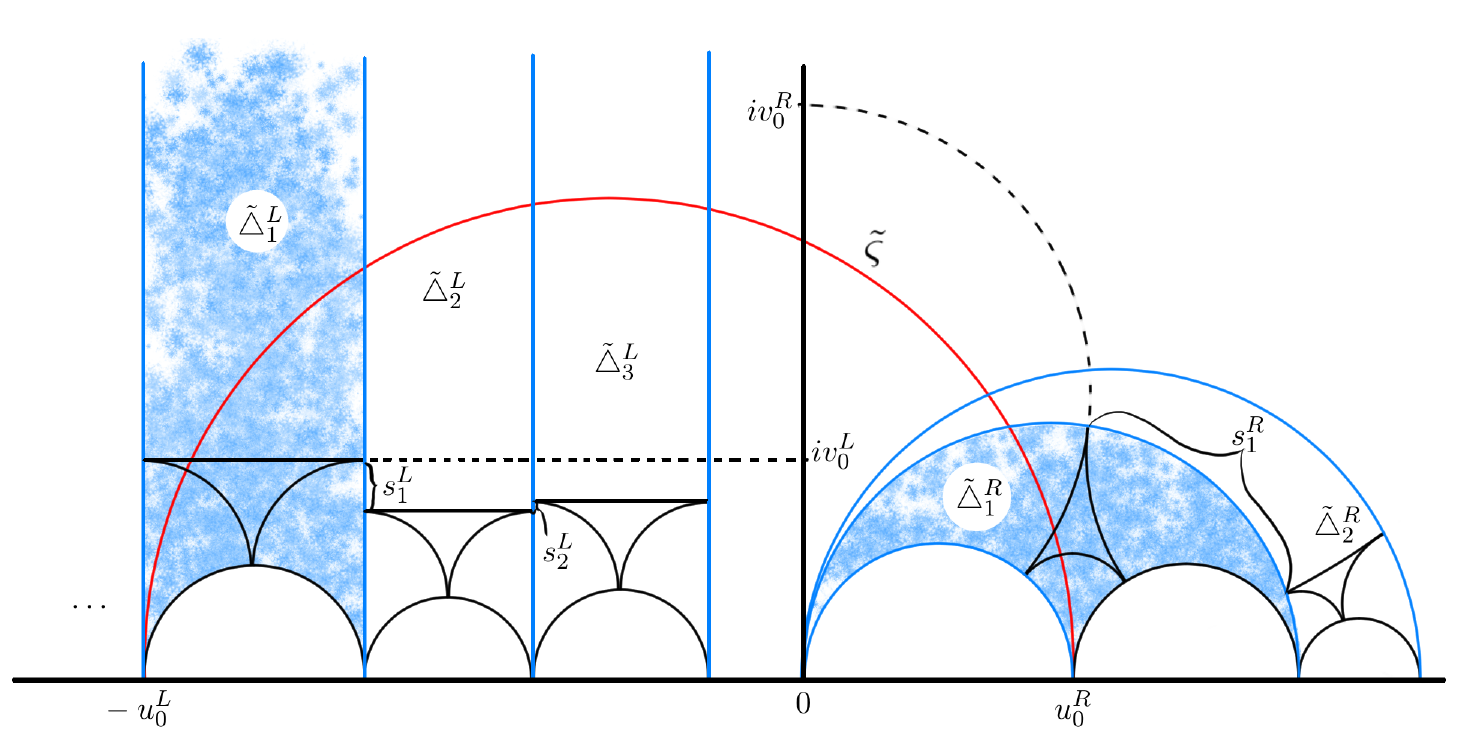}
\caption{The universal cover of the $(3,2)$-annulus in \cref{fig:23_1}.}
\label{fig:23_2}
\end{center}
\end{figure}

\begin{lemma}\label{lem:gencrown}
Let $(X,x)$ be  an $(n_L,n_R)$-crowned hyperbolic annulus with core geodesic $\gamma_S$, and set $\tau_{\gamma_S}$ to be the twist parameter for $\gamma_S$.
Then, we have
\begin{align*}
&\tau_{\gamma_S}(\mathrm{stretch}(x,\lambda_+,t))\\
&=e^t\tau_{\gamma_S}(x)+2\left(\log(1-e^{-e^t\ell_{\gamma_S}(x)})-e^t\log(1-e^{-\ell_{\gamma_S}(x)})\right)
\\&+e^t\log\left(1+e^{-s^L_1(x)}+
\cdots+e^{-(s^L_1(x)+\ldots+s^L_{n_L-1}(x))}\right)\\
&-\log\left(1+e^{-e^ts^L_1(x)}+
\cdots+e^{-e^t(s^L_1(x)+\ldots+s^L_{n_L-1}(x))}\right)\\
&+e^t\log\left(1+e^{-s^R_1(x)}+
\cdots+e^{-(s^R_1(x)+\ldots+s^R_{n_R-1}(x))}\right)\\
&-\log\left(1+e^{-e^ts^R_1(x)}+
\cdots+e^{-e^t(s^R_1(x)+\ldots+s^R_{n_R-1}(x))}\right).
\end{align*}
\end{lemma}

\begin{proof}
We employ the same strategy of proof as for \cref{lem:pregencrown}. Let $\tau_0$ and $\tau_t$ respectively denote the twist parameter $\tau_{\gamma_S}$ evaluated at $x$ and at $\mathrm{stretch}(x,\lambda_+,t)$. As before, this tells us that the correspondingly labelled $u^L_0,u^R_0, iv^L_0, iv^R_0$ (see \cref{fig:23_2}), and the points after the $e^t$-stretching, $u^L_t,u^R_t, iv^L_t, iv^R_t$ satisfy
\begin{align}
\tau_0=\log\left(\tfrac{u^R_0}{u^L_0}\right)
\quad\text{ and }\quad
\tau_t=\log\left(\tfrac{u^R_t}{u^L_t}\right).
\end{align}
From \cref{lem:component}, we obtain
\begin{align}
\tau_0
=&\log\left(
\tfrac{u^R_0}{v^R_0}
\cdot\tfrac{v^L_0}{u^L_0}
\cdot\tfrac{v^R_0}{v^L_0}\right)
=\log\left(\tfrac{v^R_0}{v^L_0}\right)
-\log\left(\tfrac{v^R_0}{u^R_0}\right)
-
\log\left(\tfrac{u^L_0}{v^L_0}\right)
\notag\\
=&\log\left(\tfrac{v^R_0}{v^L_0}\right)\notag\\
&-\log\left(\frac{1+e^{-s^L_1(x)}+
\cdots+e^{-(s^L_1(x)+s^L_2(x)+\cdots+s^L_{n_L-1}(x))}}
{1-e^{-\ell_{\gamma_S}(x)}}\right)\notag\\
&-\log\left(\frac{1+e^{-s^R_1(x)}+
\cdots+e^{-(s^R_1(x)+s^R_2(x)+\cdots+s^R_{n_R-1}(x))}}
{1-e^{-\ell_{\gamma_S}(x)}}\right).\label{eq:explain}
\end{align}
To see how the last term in \cref{eq:explain} follows from \cref{lem:component}, apply the M\"obius transformation $z\mapsto -z^{-1}$ to the universal cover of the right component of $S\setminus\gamma_S$ and observe that its image precisely satisfies the configuration of \cref{lem:component}, and that our particular definition of shearing parameter remains unchanged under such a transformation, and the only changes are for $u^R_0$ and $iv^R_0$, which respectively get sent to $-(u^R_0)^{-1}$ and $(v^R_0)^{-1} i$.
Hence

\begin{align*}
\tfrac{v^R_0}{u^R_0}
=
\tfrac{(u^R_0)^{-1}}{(v^R_0)^{-1}}
=
\frac{1+e^{-s^R_1(x)}+
\cdots+e^{-(s_1^R(x)+s_2^R(x)+\cdots+s_{n_R-1}^R(x))}}
{1-e^{-\ell_{\gamma_S}(x)}}.
\end{align*}

Proceeding as in the proof of \cref{lem:pregencrown}, we observe that the segment between $v_0^L$ and $v_0^R$ is $e^t$-stretched to a segment of length $\log\left(\tfrac{v^R_t}{v^L_t}\right)=e^t\log\left(\tfrac{v^R_0}{v^L_0}\right)$. Substituting this into the $e^t$-stretched version of \cref{eq:explain} yields
\begin{align}
\tau_t
=&
\log\left(
\tfrac{u^R_t}{v^R_t}
\cdot\tfrac{v^L_t}{u^L_t}
\cdot\tfrac{v^R_t}{v^L_t}\right)
=\log\left(\tfrac{v^R_t}{v^L_t}\right)
-\log\left(\tfrac{v^R_t}{u^R_t}\right)
-
\log\left(\tfrac{u^L_t}{v^L_t}\right)
\notag\\
=&e^t\log\left(\frac{v^R_0}{v^L_0}\right)\notag\\
&-\log\left(\frac{1+e^{-e^ts^L_1(x)}+
\cdots+e^{-e^t(s^L_1(x)+s^L_2(x)+\cdots+s^L_{n_L-1}(x))}}
{1-e^{-e^t\ell_{\gamma_S}}}\right)
\notag\\
&-\log\left(\frac{1+e^{-e^ts^R_1(x)}+
\cdots+e^{-e^t(s^R_1(x)+s^L_2(x)+\cdots+s^R_{n_R-1}(x))}}
{1-e^{-e^t\ell_{\gamma_S}(x)}}\right).\notag
\end{align}

Replacing the $\log\left(\tfrac{v^R_0}{v^L_0}\right)$ term here, using \cref{eq:explain}, with expressions in $\tau_0$, $\ell_{\gamma_S}(x)$ and shearing parameters then yields the desired result.
\end{proof}

\begin{remark}
\label{rmk:oppositespiral}
We  analogously define $\lambda_-$ with $n_L+n_R$ geodesics going from the boundary cusps of $S$ and spiralling towards $\gamma_S$ in such a way that those on the left component spiral in the direction of $\gamma_S^{-1}$, and those on the right component spiral in the direction of $\gamma_S$. By symmetry, the relevant Dehn-twist parameter $\tau_{\gamma_S}$ satisfies:
\begin{align*}
&\tau_{\gamma_S}(\mathrm{stretch}(x,\lambda_-,t))\\
&=e^t\tau_{\gamma_S}(x)-2\left(\log(1-e^{-e^t\ell_{\gamma_S}(x)})-e^t\log(1-e^{-\ell_{\gamma_S}(x)})\right)\\
&-e^t\log\left(1+e^{-s^L_1(x)}+
\cdots+e^{-(s^L_1(x)+\ldots+s^L_{n_L-1}(x))}\right)\\
&+\log\left(1+e^{-e^ts^L_1(x)}+
\cdots+e^{-e^t(s^L_1(x)+\ldots+s^L_{n_L-1}(x))}\right)\\
&-e^t\log\left(1+e^{-s^R_1(x)}+
\cdots+e^{-(s^R_1(x)+\ldots+s^R_{n_R-1}(x))}\right)\\
&+\log\left(1+e^{-e^ts^R_1(x)}+
\cdots+e^{-e^t(s^R_1(x)+\ldots+s^R_{n_R-1}(x))}\right).
\end{align*}
We again emphasise that in changing the lamination from $\lambda_+$ to $\lambda_-$, the new shearing parameters $s^L_i$ and $s^R_j$ are (a priori) different to the corresponding variables for $\lambda_+$, and it is \emph{generically false} that 
\[
\tau_{\gamma_S}(\mathrm{stretch}(x,\lambda_+,t))+\tau_{\gamma_S}(\mathrm{stretch}(x,\lambda_-,t))=2e^t\tau_{\gamma_S}(x).
\]

\end{remark}

\subsection{Stretch maps for complete finite-area hyperbolic surfaces}

We now explain how we can utilise the expressions for stretch maps on crowned annuli to describe stretch maps for finite-leaf laminations of complete finite-area hyperbolic surfaces with possibly boundary cusps, \ie complete hyperbolic surfaces with finitely many boundaries, each with at most finitely many boundary cusps.

A  finite lamination $\Lambda$ on a complete finite-area hyperbolic surface $(S,x)$ consists of simple closed geodesics and simple bi-infinite geodesics which either spiral to (and from) the aforementioned simple closed geodesics or to cusps. 
As before, we fix an orientation on $S$.
With each closed geodesic $\gamma$, we can associate its length function $\ell_\gamma\colon \teich(S)\to\mathbb{R}_{>0}$ and a  Fenchel--Nielsen twist function $\tau_\gamma\colon \teich(S)\to\mathbb{R}$. 
We observe that each bi-infinite geodesic in $\Lambda$  is the shared edge between two (possibly non-distinct) ideal triangles and hence has a well-defined shearing parameter $s_{\alpha}$. 
The combination of the Fenchel--Nielsen twist parameters and the shearing parameters globally parameterises $\teich(S)$.
(See \cite[Section~7]{CF} and \cite[Section~2.4.1]{Huang-Thesis}.)
We further note that the length parameter $\ell_{\gamma}$ is a linear combinations of the shearing parameters of bi-infinite geodesics spiralling to $\alpha$ (see \cite{Pen-D} or \cite{Huang-Thesis}). Specifically, if the geodesic \emph{rays} (\ie rays consisting of one end of a bi-infinite geodesic) spiralling to a given side of $\gamma$ lie on the bi-infinite geodesics $\alpha_1,\ldots,\alpha_k$, then there is a choice of signs $\epsilon_i\in\{+,-\}$ such that
\begin{align}
\ell_{\gamma}
=
\epsilon_1 s_{\alpha_1}+\ldots+\epsilon_k s_{\alpha_k}.\label{eq:closedleaf}
\end{align}
The upshot here is that the expressions for Thurston's stretch map in \cref{thm:key,lem:pregencrown,lem:gencrown} are given in terms of a natural coordinate system on $\teich(S)$ which is a mixture of Fenchel--Nielsen twist coordinates and the shearing parameters.

The shearing parameter on a leaf $\alpha$ is equal to  the signed length of the geodesic segment on  $\alpha$ in $\Lambda$ bounded by the anchor points of adjoining ideal triangles to $\alpha$. 
Under the action of a $K$-stretch map, these parameters are multiplied by the factor $K$. Likewise, under the action of a $K$-stretch map, the length $\ell_\gamma$ of a simple closed geodesic $\gamma$ in $\Lambda$ is also multiplied by the factor $K$ (note that this agrees with \cref{eq:closedleaf}).

The action of a $K$-stretch map on Fenchel--Nielsen twist parameters is more complex, and is precisely spelled out by \cref{thm:key,lem:pregencrown,lem:gencrown}. To be more concrete: Given a simple closed geodesic $\gamma\subset\Lambda$, consider an arbitrary lift $\tilde{\gamma}$ of $\gamma$ to the universal cover $\tilde{S}$ of $S$, and consider the subset of $\tilde{S}$ consisting of $\tilde{\gamma}$ and all the ideal triangles in the complement of the lift of $\Lambda$ in $\tilde{S}$ which spiral to $\tilde{\gamma}$. The metric completion $\tilde{A}$ of this collection of ideal triangles is invariant under translation along $\tilde{\gamma}$ by the deck transformation in $\pi_1(S)$ corresponding to $\tilde{\gamma}$. The quotient $A$ of $\tilde{A}$ by the aforementioned deck transformation is a crowned annulus around a copy of $\gamma$, and the restriction of $\Lambda$ on $A$ simultaneously takes the form of the laminations described in \cref{thm:key,lem:pregencrown,lem:gencrown} whilst having the same shearing parameters and twist coordinates as corresponding geodesic leaves on $\Lambda$ on $(S,x)$, which allows us to invoke \cref{thm:key,lem:pregencrown,lem:gencrown} to determine the deformation of the Fenchel--Nielsen twist coordinate $\tau_\gamma$ for $\gamma$.

\subsection{Anti-stretch paths}

\begin{definition}[{anti-stretch maps, \cite{The}}]
Although we have hitherto defined $e^t$-stretch maps and stretch paths exclusively for $t\geq0$, the notion of stretch maps and stretch paths are well defined for $t<0$, and are respectively referred to as the \emph{antistretch map} and \emph{antistretch path}.
\end{definition}

\begin{theorem}[back-time convergence {\cite[Theorem~4.1]{The}}]\label{thm:backtime}
Consider a finite-leaf lamination $\Lambda$, and denote its closed leaves by $\gamma_1,\ldots,\gamma_k$. As $t\to\infty$, for any $x\in\teich(S)$, the $e^{-t}$-stretch path with respect to $\Lambda$,  \ie $x_{-t}:=\mathrm{stretch}(x,\Lambda,-t)$, converges to the uniformly weighted projectivised multicurve
\[
[\gamma_1+\gamma_2+\cdots+\gamma_k]\in\pml(S)=\partial_\infty\teich(S),
\] 
where $\pml(S)=\partial_\infty \teich(S)$ denotes the boundary of the Thurston compactification of $\teich(S)$. Note that this is independent of the starting point $x$.
\end{theorem}

\begin{proof}
We only sketch the proof, since the theorem is proved in  \cite{The}. The proof we give here is in the trend of the techniques we have developed so far in this paper.
To begin with, we observe that the shearing lengths assigned to the bi-infinite leaves in $\Lambda$, and the Fenchel--Nielsen parameters attached to $\gamma_1,\ldots,\gamma_k$ suffice to parametrise the Teichm\"uller space $\teich(S)$. As $s\to\infty$, by the $e^{-s}$-stretch map (\ie the anti-stretch map) with respect to $\Lambda$,
\begin{enumerate}
\item
all of the shearing lengths  for the bi-infinite leaves in $\Lambda$ exponentially shrink to $0$,
\item
the Fenchel--Nielsen length parameter for $\gamma_j$ also exponentially shrinks to $0$, and
\item
the Fenchel--Nielsen twist parameter for $\gamma_j$ behaves as $\pm2s+o(1)$.
\end{enumerate}
Point~3 comes from analysing the behaviour of \cref{lem:gencrown} and \cref{rmk:oppositespiral} as $t=-s\to-\infty$ there: every term tends either to $0$ or a constant except for 
\[
\pm2\log\left(1-e^{-e^s\ell_j(x)}\right)=\mp2s+O(1).
\]
The first two points tell us that the geometry of $S\setminus(\gamma_1\cup\cdots\cup\gamma_k)$ converges to a cusped surface of finite area, possibly with multiple connected components.
On every $(S, x_{-t})$, we consider a collar neighbourhood for each of $\gamma_1, \dots, \gamma_k$ whose boundary components have length $2$.
Then, each side of the  collar neighbourhood around $\gamma_j$ is an annulus of width 
\[
\mathrm{arcsinh}\left(\frac{1}{\sinh(\frac{1}{2} e^{-s}\ell_{\gamma_j})}\right)=s+O(1),
\]
and geometrically tends to a cusp neighbourhood bounded by a horocycle of length $2$. 
Since the complement of the collar neighbourhoods around the $\{\gamma_j\}$ converges to a compact surface bounded by horocycles, the complement of the collar neighbourhood must have diameter bounded above by some constant independent of $t$. Therefore, the length of an arbitrary simple closed geodesic in the complement of the $\{\gamma_j\}$ necessarily stabilises and converges as $s$ tends to $\infty$. In contrast, for an arbitrary simple closed geodesic $\gamma$ that intersects  $M>1$ times the $\gamma_1 \sqcup \dots \sqcup \gamma_k$, its length grows due to the following two dominant contributing factors:
\begin{itemize}
\item
the growing width of the collar neighbourhoods around the $\gamma_j$, and
\item
the accumulating Fenchel--Nielsen twists around the $\{\gamma_j\}$ induced by the anti-stretch.
\end{itemize}
The former contributing factor tells us that the length of $\gamma$ as $s\to\infty$ is bounded below by $2Ms+o(1)$. Coupled with the second contributing factor, we easily obtain the coarse upper bound of $4Ms+o(1)$, which we now improve upon. Observe that the direction of the $\mp2s+O(1)$ Fenchel--Nielsen twist is orthogonal to the length $2s+O(1)$ geodesic cutting across the collar neighbourhood surrounding each $\gamma_j$. The hyperbolic Pythagoras' theorem (see, e.g.: \cite[Theorem~2.2.2.i]{B}) therefore tells us that the length of each subsegment of $\gamma$ that cuts across the collar neighbourhood of $\gamma_j$ grows as
\[
2\mathrm{arcosh}(\cosh(s+O(1))\cosh(t+O(1))=2s+O(1),
\]
which shows that the length of $\gamma_j$ is bounded above by $2Ms+O(1)$ --- which equals the previously established lower bound. The embedding of $x_{-s}=\mathrm{stretch}(x,\Lambda,-s)$ in $\mathbb{P}(\mathbb{R}_{\geq0}^{\mathcal{S}})$ can be normalised via division by $2s$ and represented by
\[
\phi:\mathcal{S}\to\mathbb{R}_{\geq0},\quad \gamma\mapsto i(\gamma,\gamma_1+\gamma_2+\cdots+\gamma_k),
\]
where $i(\alpha,\beta)$ denotes the geometric intersection number between $\alpha$ and $\beta$. The limit of $x_{-s}$ in $\pml(S)$ is therefore given by the projective class of $\phi\in\mathbb{R}_{\geq0}^{\mathcal{S}}$, \ie the projective multicurve $[\gamma_1+\gamma_2+\cdots+\gamma_k]\in\pml(S)$.
\end{proof}

\section{Stretch vectors} \label{s:infinitesimal}

We now return to the context where $S=S_{g,n}$ is a general orientable surface of genus $g$ with $n$ cusps. The aim of this subsection is to derive results pertaining to tangent vectors induced by stretch paths.

\begin{definition}[stretch vector]\label{def:infinitesimal}
Consider a point $x\in\teich(S)$ and a complete lamination $\lambda$ on $S$. Consider the geodesic ray $\left\{\operatorname{stretch}(x,\lambda,t)\right\}_{t\geq0}$, where $\operatorname{stretch}(x,\lambda,t)$ is obtained from $x$ by the $e^t$-stretching with respect to $\lambda$. We refer to the tangent vector
\[
v_\lambda:=
\left.\tfrac{\mathrm{d}}{\mathrm{d}t}\right|_{t=0}
\operatorname{stretch}(x,\lambda,t)
\in T_x\teich(S)
\]
as the \emph{stretch vector} at $x$ with respect to $\lambda$.
\end{definition}

We begin by giving  an infinitesimal version of the fact referenced in the second to last sentence of \cref{maxstretch} with a sketch of proof. This claim is also implicitly asserted in the proof of {\cite[Theorem~5.1]{ThS}}.

\begin{lemma}[maximally stretched lamination {\cite[Theorem~5.1]{ThS}}]
\label{thm:infcomparison}
Consider a complete geodesic lamination $\lambda$ with corresponding stretch vector $v_{\lambda}\in T_x\teich(S)$. Then, for any measured lamination $\mu\in\ml(S)$ such that the geodesic lamination $|\mu|$ supporting $\mu$ has a nonempty transverse intersection with $\lambda$, we have 
\[
v_{\lambda}(\log\ell(\mu))<1.
\]
\end{lemma}

\begin{proof}
We may assume by multiplying by a scalar that $\ell_x(\mu)=1$, and hence we compute
\begin{align}
v_{\lambda}(\log\ell(\mu))
=\frac{v_{\lambda}(\ell(\mu))}{\ell_x(\mu)}
=v_{\lambda}(\ell(\mu))
=\lim_{t\to0}\frac{\ell_{x_t}(\mu)-\ell_x(\mu)}{t},
\label{eq:infinitesimalcomparison}
\end{align}
where $x_t=\operatorname{stretch}(x,\lambda,t)$. Although the precise value of $\ell_{x_t}(\mu)$ is difficult to compute, it suffices for our purposes to find a small enough upper bound so that \cref{eq:infinitesimalcomparison} approaches a value less than $1$. Specifically, we compute the length of the lamination to which $\mu$ deforms with respect to the change of metric parameterised by $x_t$. To this end, we first show that each ideal triangle that is a complementary component of $S \setminus \lambda$ can be partially foliated by vertical geodesics asymptotic to a leaf of $\lambda$ and horizontal horocyclic leaves outside the central stable triangle. 
The adjectives ``vertial" and ``horizontal" are used to emphasize the fact that the two foliations are perpendicular, which is used in the computation that follows.
The stretching map expands the vertical direction at most by the factor $e^t$ and shrinks the horizontal direction.
Then, there exist some $\theta\in(0,\tfrac{\pi}{2})$ and a subset $\nu$ of $\mu$ of length $L\in(0,1]$ such that the geodesics constituting $\nu$ intersect the vertical geodesic partial foliation  at an angle between $\theta$ and $\tfrac{\pi}{2}$. Computing the change in the norm of each unit tangent vector along the geodesic segments of $\nu$ as one deforms the metric $x=x_0$ to $x_t$ shows that the new norm is smaller than $\cos(\theta)e^t$. This suffices to show that \cref{eq:infinitesimalcomparison} is strictly less than $L\cos(\theta)+(1-L)<1$.
\end{proof}

\begin{lemma}\label{thm:infkey}
Consider a pair of complete geodesic laminations $\Lambda_\pm$ on $(S,x)\in\teich(S)$ which agree everywhere except on a $(1,1)$-crowned annulus contained in $S$ with core geodesic $\gamma_0$ of length $\ell_0$, in which $\Lambda_\pm$ respectively coincide with $\lambda_\pm$ (as by the notation of \cref{thm:key}). Then, the difference between the stretch vectors $v_{\Lambda_\pm}$ is expressed as:
\begin{align}
v_{\Lambda_+}-v_{\Lambda_-}
=
\left(
\frac{4\ell_0 e^{-\ell_0}}{1-e^{-\ell_0}}-4\log(1-e^{-\ell_0})
\right)
E_{\gamma_0},\label{eq:infdiff}
\end{align}
where $E_{\gamma_0}$ is the unit Fenchel--Nielsen twist vector with respect to $\gamma_0$. In particular,
\begin{align}
\|v_{\Lambda_+}-v_{\Lambda_-}\|_{\mathrm{Th}}
=
(4\ell_0e^{-\ell_0}+o(\ell_0e^{-\ell_0}))
\| E_{\gamma_0}\|_{\mathrm{Th}}
\text{, as }\ell_0\to\infty.\label{eq:asympdiff}
\end{align}
\end{lemma}

\begin{proof}
Let $S_0\subset S$ denote the specified $(1,1)$-crowned annulus in $S$ containing $\gamma_0$ as its core geodesic. We first observe that since $\Lambda_+$ and $\Lambda_0$ differ only in the interior of $S_0$,
 the difference between the metrics $\operatorname{stretch}(x,\Lambda_+,t)$ and $\operatorname{stretch}(x,\Lambda_-,t)$ occurs only within $S_0$. In other words, the restriction of the metrics $\operatorname{stretch}(x,\Lambda_{\pm},t)$ to $S\setminus S_0$ are isometric. Moreover, since there is only one way to attach a pair of $(1,0)$-crowned annuli of cuff length $e^t\ell_0$ to each of the two crowned boundaries of $S\setminus S_0$, this further informs us that the geodesically bordered hyperbolic metrics induced by the two stretchings $\operatorname{stretch}(x,\Lambda_{\pm},t)$ on $S\setminus\gamma_0$ are also necessarily isometric. 
 Therefore, the vector $v_{\Lambda_+}-v_{\Lambda_-}$ is a multiple of $E_{\gamma_0}$, the unit Fenchel--Nielsen twist vector with respect to $\gamma_0$. \cref{thm:key} tells us that this factor is
\begin{align*}
\left.\frac{d}{dt}\right|_{t=0}
4\left(\log(1-e^{-e^t\ell_0})-e^t\log(1-e^{-\ell_0})\right)
=
4\left(\frac{\ell_0e^{-\ell_0}}{1-e^{-\ell_0}}-\log(1-e^{-\ell_0})\right).
\end{align*}
This gives \cref{eq:infdiff}, and \cref{eq:asympdiff} follows by L'H\^{o}pital's rule.
\end{proof}

\begin{corollary}
\label{thm:linindep}
The tangent vectors $v_{\Lambda_+}$ and $v_{\Lambda_-}$ at $T_x\teich(S)$ are distinct.
\end{corollary}

\begin{proof}
Note that 
\begin{align*}
\frac{4\ell_0e^{-\ell_0}}{1-e^{-\ell_0}}-4\log(1-e^{-\ell_0})
=\text{positive term}-(\text{negative term})
>0.
\end{align*}
\cref{eq:infdiff} then tells us that $v_{\Lambda_+}\neq v_{\Lambda_-}$.
\end{proof}

We can give a much more general version of \cref{thm:linindep}. Let $\Lambda_1$ and $\Lambda_2$ be two complete geodesic laminations on $(S,x)$ such that $\Lambda_1\cap\Lambda_2$ contains the boundary of some crowned annulus $T\subset S$, as well as the unique simple closed geodesic $\gamma_T$ contained in the interior of $T$. Further require that $\Lambda_1\cap T$ and $\Lambda_2\cap T$, which are both complete geodesic laminations on $T$, are respectively equal to $\lambda_+$ and $\lambda_-$  as defined in  \cref{lem:gencrown}. Given this setup, we have the following result:

\begin{lemma}
 Suppose that all of the shearing parameters assigned to the bi-infinite geodesic leaves of $\Lambda_1\cap T$ and $\Lambda_2\cap T$ which spiral to $\gamma_T$ are positive. Then, the tangent vectors $v_{\Lambda_1}$ and $v_{\Lambda_2}$ at $T_x\teich(S)$ are distinct.
\end{lemma}

\begin{proof}
We can see from \cref{lem:gencrown} and \cref{rmk:oppositespiral} that the directional derivatives $v_{\Lambda_1}(\tau_{\gamma_T})$ and $v_{\Lambda_2}(\tau_{\gamma_T})$ can be expressed purely in terms of the geometry of $T\subset S$. In particular, the expression of
\[
(v_{\Lambda_1}-v_{\Lambda_2})(\tau_{\gamma_T})
=
v_{\Lambda_1}(\tau_{\gamma_T})-v_{\Lambda_2}(\tau_{\gamma_T})
\]
is a sum of positive multiples of terms of the form
\begin{align*}
\frac{\ell_{\gamma_T} e^{-\ell_{\gamma_T}}}{1-e^{-\ell_{\gamma_T}}}-\log(1-e^{-\ell_{\gamma_T}}),
\end{align*}
and terms of the form
\begin{align*}
&\log\left(1+e^{-s_1}+e^{-(s_1+s_2)}+\cdots+e^{-(s_1+\cdots+ s_{n-1})}\right)\\
&+\frac{s_1e^{-s_1}+(s_1+s_2)e^{-(s_1+s_2)}+\cdots+(s_1+\cdots+ s_{n-1})e^{-(s_1+\cdots+ s_{n-1})}}{1+e^{-s_1}+e^{-(s_1+s_2)}+\cdots+e^{-(s_1+\cdots+ s_{n-1})}}.
\end{align*}
Our previous computation for the proof of \cref{thm:linindep} and  the assumption of positive shearing parameters then ensure that all these terms are positive, and hence $(v_{\Lambda_1}-v_{\Lambda_2})(\tau_{\gamma_T})\neq0$. Therefore, $v_{\Lambda_1}-v_{\Lambda_2}\neq 0$, and these two vectors are distinct.
\end{proof}

\subsection{Non-chain-recurrent laminations}
\label{non-chain}
We first review Thurston's construction (\cite[Theorem~8.5]{ThS}) showing that any two points $x, y \in \teich(S)$ can be jointed by a geodesic with respect to the Thurston metric.
Such a geodesic is obtained by concatenating finitely many stretch paths.
The initial stretch path  in the geodesics is a stretch path with respect to any complete lamination $\Lambda$ which contains the maximal ratio-maximising lamination $\mu(x,y)$  (\cref{defn:maxratmax}) between $x$ and $y$. In particular, one can always take $\Lambda$ to be the unique complete lamination extending a maximal chain-recurrent lamination containing $\mu(x,y)$. 
We continue the stretching along $\Lambda$ from $x$ to $x_t$ until such a time $t$ that the maximal ratio-maximising lamination $\mu(x_t,y)$ changes in topology. 
(Still it is shown that $\mu(x_t,y)$ contains $\mu(x,y)$.)
One may then exchange $\Lambda$ for a complete geodesic lamination containing the updated maximal ratio-maximising lamination. Again, we can choose this new complete geodesic lamination to be a complete lamination extending a maximal chain-recurrent lamination containing the updated maximal ratio-maximising lamination. Repeating this selection procedure suffices to show that one can always join two arbitrary points in Teichm\"uller space via a concatenation of stretch paths with respect to (complete laminations containing) maximal chain-recurrent laminations. When the maximal ratio-maximising lamination $\mu(x,y)$ is a maximal chain-recurrent lamination, there is a unique Thurston geodesic joining $x$ and $y$ given by the stretch path with respect to (the complete lamination containing) the maximal chain-recurrent lamination containing $\mu(x,y)$.

Based on the above observations, we heuristically expect that stretch paths with respect to maximal chain-recurrent laminations are “extremal” among Thurston metric geodesics, which in turn suggests that stretch vectors for stretch paths with respect to maximal chain-recurrent laminations should constitute extreme points in the unit Thurston norm tangent sphere $\mathbf{S}_x\subsetneq T_x\teich(S)$. 
(Here a point $p$ on the boundary $S$ of a convex  body  is called an extreme point when $p$ is not an interior point of a segment contained in $S$.) 
In the case when $S$ is a one-cusped torus or a four-cusped sphere, this conjectural picture is true \cite[Theorem 6.5]{DLRT},
 but is unknown in general (see \cref{maximal ratio-maximising}). 
We note, however, that not all stretch vectors are extreme points if we drop the condition of chain-recurrence.

\begin{lemma}
Let $\Lambda_0$ be a complete geodesic lamination which contains the non-chain-recurrent lamination $\lambda_0$ of \cref{fig:2} as a sublamination. Then, $v_{\Lambda_0}$ is not an extreme point in $\mathbf{S}_x$.
\end{lemma}

\begin{proof}
Let 
\begin{itemize}
\item
$\gamma_S$ denote the simple closed geodesic contained in $\lambda_0$, 
\item
$A$ denote the $(1,1)$-crowned annulus in $(S,x)$ containing $\lambda_0$, and
\item
$\Lambda_{\pm}$ be the complete geodesic laminations obtained by replacing $\lambda_0\subset\Lambda_0$ respectively with the $\lambda_\pm$ of \cref{fig:2}.
\end{itemize}
Since $\Lambda_0$, $\Lambda_+$ and $\Lambda_-$ all agree on $S\setminus A$, the $e^t$-stretch maps with respect to these three laminations deform $S\setminus A$ in the same way. 
We can construct two pairs of pants (possibly one or two of its boundary components are cusps) on both sides of $\gamma_S$, the union of which contain the $(1,1)$-crowned annulus $A$.
We let $\gamma$ be the union of their boundary components, including $\gamma_S$.
Then, from \cref{thm:key}, it follows that for any Fenchel--Nielsen coordinate with respect to a pants decomposition containing $\gamma$, the coordinates for the $e^t$-stretch maps with respect to $\Lambda_0$ and $\Lambda_\pm$ completely agree except in the twist parameter for $\gamma_S$. Indeed, \cref{thm:key} tells us that the twist parameters for the stretch maps with respect to $\Lambda_\pm$ average to the twist parameter for the stretch map with respect to $\Lambda_0$. Therefore,
\[
v_{\Lambda_0}=\tfrac{1}{2}\left(v_{\Lambda_+}+v_{\Lambda_-}\right),
\]
where $v_{\Lambda_+}$ and $v_{\Lambda_-}$ are the respective stretch vectors with respect to $\Lambda_+$ and $\Lambda_-$. Therefore, $v_{\Lambda_0}$ is not an extreme point in $\mathbf{S}_x$.
\end{proof}

In \cref{maximal twist}, we shall consider geodesic laminations $\lambda$ on $S=S_{g,n}$ which contain, as a sublamination, (disjoint) simple closed geodesics $\alpha,\beta_1,\beta_2,\beta_3 $ and $\beta_4$ such that the collections $\{\alpha,\beta_1,\beta_2\}$ and $\{\alpha,\beta_3,\beta_4\}$ both bound pairs of pants on $S$ (see \cref{fig:doublespiral}). Our goal is to show that stretch vectors for certain $\lambda$ of the above form maximise and minimise the Fenchel--Nielsen twist parameter for the curve $\alpha$ over the collection of all unit (with respect to the Thurston norm) tangent vectors which ``maximally stretch'' $\alpha,\beta_1,\beta_2,\beta_3,$ and $\beta_4$. 
The precise statement will be given in \cref{lem:maxtwist}.
We presently only consider the setting where the simple closed geodesics $\alpha,\beta_1,\beta_2,\beta_3$ and $\beta_4$ are all distinct. 
However, elementary topological arguments suffice to allow us to extend the main results in this subsection also to cases where some of the aforementioned geodesics are the same, namely (up to symmetry) the cases when
\begin{itemize}
\item
$\beta_1=\beta_2$;

\item
$\beta_1=\beta_2$ and $\beta_3=\beta_4$; or

\item
$\beta_1=\beta_3$ and $\alpha=\beta_2=\beta_4$.

\end{itemize}

\subsection{Construction of a horogeodesic}
\label{horogeodesic}
Let $\Sigma$ be the four-holed sphere bounded by the simple closed geodesics  $\beta_1,\beta_2,\beta_3$ and $\beta_4$.
A key part of our strategy is to study the lengths of simple closed geodesics which lie in the interior of $\Sigma$ and which intersect $\alpha$ precisely twice. 
Let $\gamma_0$ be a shortest one among such simple closed geodesics, and let $\gamma_m$, $m\in\mathbb{Z}$, be the $m$-times iterated (left) Dehn twist  around $\alpha$ of $\gamma_0$. The Hausdorff limits $\lambda_\pm$ of $\gamma_m$ as $m\to\pm\infty$ are geodesic laminations in the interior of $\Sigma$ containing $\alpha$ and two other bi-infinite geodesic leaves $\alpha^L$, $\alpha^R$ spiralling around $\alpha$ in opposite directions from opposite sides (\cref{fig:doublespiral}).

\begin{figure}[h!]
\begin{center}
\includegraphics[scale=0.35]{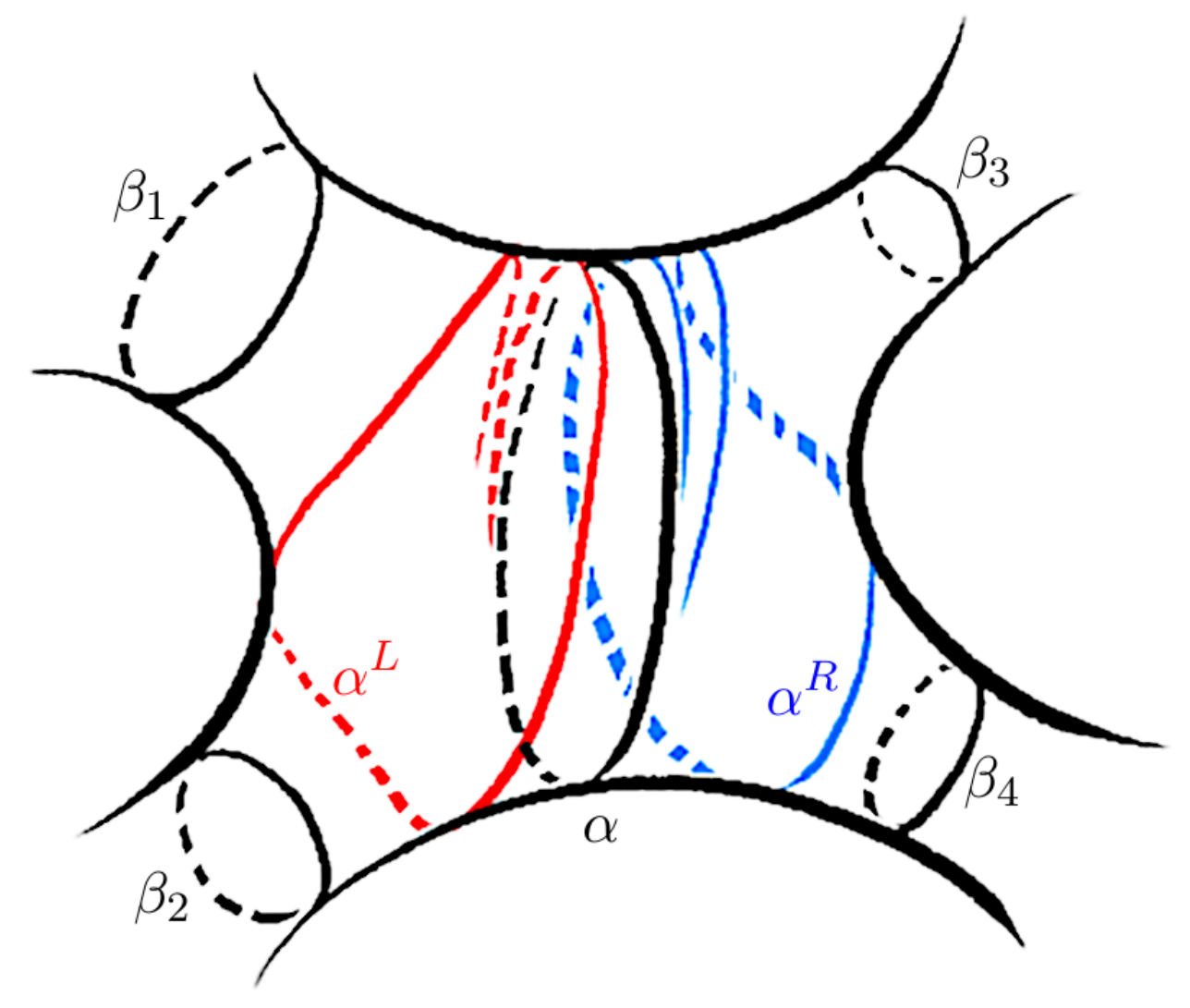}
\caption{The picture of $\Sigma$, where $\Lambda_+$ is the union of $\alpha,\alpha^L,\alpha^R,\beta_1,\beta_2,\beta_3$ and $\beta_4$.}
\label{fig:doublespiral}
\end{center}
\end{figure}

We extend the lamination $\lambda_+\cup\{\beta_1,\beta_2,\beta_3,\beta_4\}$ arbitrarily to a complete geodesic lamination $\Lambda_+$ on $S$ and likewise extend $\lambda_-\cup\{\beta_1,\beta_2,\beta_3,\beta_4\}$ to a complete lamination $\Lambda_-$. 
Note that the length of $\gamma_m$ under stretching by $\Lambda_{\pm}$ depends only on the restriction of $\Lambda_{\pm}$ to $\Sigma$. 
There are (at most) $2^4$ different choices (two choices for each $\beta_i$) for how one extends  $\lambda_{\pm}$ to a complete lamination on $\Sigma$. 
Each of these $2^4$ choices comes from adding geodesic leaves $\sigma_1,\sigma_2,\sigma_3,\sigma_4$ which spiral from $\beta_1,\beta_2,\beta_3,\beta_4$ to $\alpha$.

We first study the behaviour of the length of $\gamma_m$, for $m\gg0$, under the $K$-stretch map with respect to $\Lambda_+$. We shall first homotope $\gamma_m$ to a closed curve  obtained by concatenating segments along leaves of the stretch-invariant horocyclic foliation and those on leaves of  $\Lambda_+$ as follows. For obvious reasons, such a curve is termed \emph{horogeodesic}.
Since $\gamma_m$ intersects $\alpha$ twice, the latter separates $\gamma_m$ into two geodesic segments, which we name $\gamma^L$ and $\gamma^R$. 
When $m$ is sufficiently large, we know that $\gamma_m$ well-approximates $\lambda_+$ and hence the geodesic segment $\gamma^L$ is homotopic (fixing the endpoints) to a path which traverses along:
\begin{itemize}
\item
a segment on a leaf of the stretch-invariant horocyclic foliation, so as to go from $\alpha$ to $\alpha^L$,
\item
a segment on $\alpha^L$, and then 
\item
a segment on a leaf of the stretch-invariant horocyclic foliation, so as to go from $\alpha^L$ back to $\alpha$.
\end{itemize}
We homotope $\gamma^R$ analogously to a horocyclic segment, followed by an $\alpha^R$ segment and then another horocyclic segment. See \cref{fig:holonomy} where the horogeodesic path is depicted on a lift to the universal cover, with the six segments labeled $l^L,\ l^R,\  h_1,\  h_2,\   h_3$, and $ h_4$.

\begin{figure}[h!]
\begin{center}
\includegraphics[scale=0.4]{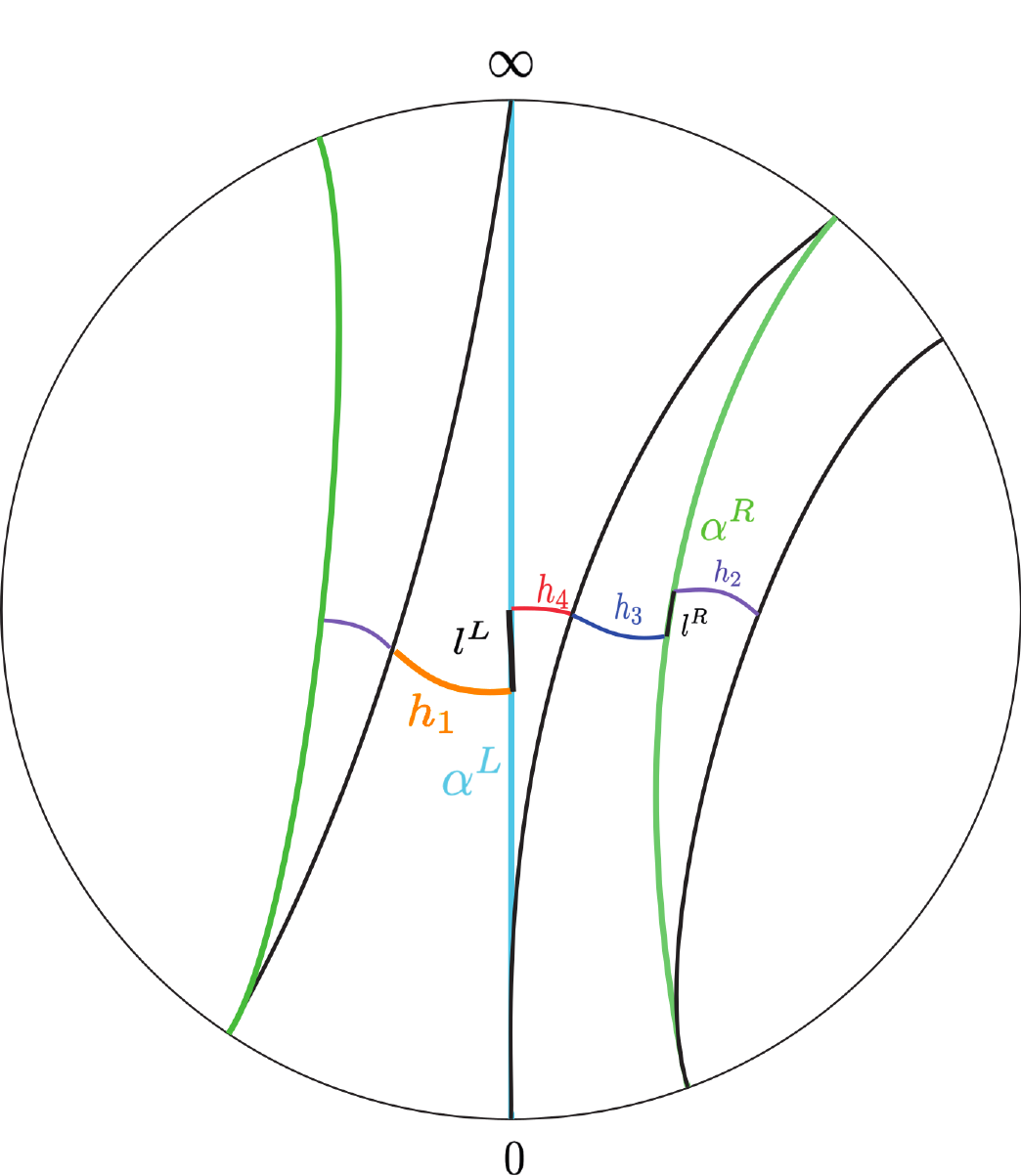}
\caption{The curve $\gamma_m$ is homotopic to a closed horogeodesic obtained by concatenating  $l^L,\ l^R,\  h_1,\  h_2,\   h_3$, and $ h_4$.}
\label{fig:holonomy}
\end{center}
\end{figure}

\subsection{Maximal twist lemma}
\label{maximal twist}
We now express the length of $\gamma_m$ in terms of the lengths of these six segments by computing a $\mathrm{PSL}_2(\mathbb{R})$-conjugacy representative of the holonomy matrix for $\gamma_m$. 

\begin{lemma}
The length $\ell_{\gamma_m}$ of $\gamma_m$  is given in terms of the lengths of $l^L,l^R,h_1,h_2,h_3$ and $h_4$ (we use the same symbols to represent their lengths) by the following formula:
\begin{align}
2\cosh(\tfrac{\ell_{\gamma_m}}{2})=&e^{\frac{1}{2}(l^L+l^R)}(1+h_1h_2)(1+h_3h_4)+e^{\frac{1}{2}(l^L-l^R)}h_1h_4\label{eq:gntrace}\\
&+e^{\frac{1}{2}(l^R-l^L)}h_2h_3+e^{-\frac{1}{2}(l^L+l^R)}.\notag
\end{align}
\end{lemma}

\begin{proof}
We  lift $\alpha^L$ to a geodesic with endpoints at $0$ and $\infty$, 
  and $h_4$ to a segment joined to $i$ on the right of the lift of $\alpha^L$ (see \cref{fig:holonomy}). We obtain the holonomy matrix for $\gamma_m$ by composing the matrices corresponding to the following sequence of transformations:
\begin{enumerate}
\item
hyperbolically translate along the lift of $\alpha^L$ by $l^L$, so that the endpoint of the lift of $h_1$ is now placed at $i$;
\item
rotate by $\pi$ around $i$, so that the lift of $h_1$ is now placed to the right of $i$ on the lift of $\alpha^L$;
\item
parabolically translate along the new lift of $h_1$, so that the shared endpoint of $h_1$ and $h_2$ is now placed at $i$;
\item
rotate by $\pi$ around $i$ so that the new lift of $h_2$ is now to the left of $i$;
\item
parabolically translate along the lift of $h_2$, so that the shared endpoint of the respective lifts of $h_2$ and $\alpha^R$ is positioned at $i$;
\item
hyperbolically translate along the lift of $\alpha^R$ by $l^R$, so that the endpoint of the lift of $h_3$ is now placed at $i$;
\item
rotate by $\pi$ around $i$, so that the lift of $h_3$ is now placed to the right of $i$ on the lift of $\alpha^R$;
\item
parabolically translate along the lift of $h_3$, so that the shared endpoint of $h_3$ and $h_4$ is now placed at $i$;
\item
rotate by $\pi$ so that the lift of $h_4$ is now to the left of $i$; and
\item
parabolically translate along the lift of $h_4$, so that the shared endpoint of the respective lifts of $h_4$ and $\alpha^R$ is positioned at $i$. This results in a $\gamma_m$-translate of our initial configuration.
\end{enumerate}

Putting together the matrices (ordered from right to left), we obtain:
\begin{align*}
&\pm
\left[\begin{smallmatrix}
1&h_4\\0&1
\end{smallmatrix}\right]
\left[\begin{smallmatrix}
0&-1\\1&0
\end{smallmatrix}\right]
\left[\begin{smallmatrix}
1&-h_3\\0&1
\end{smallmatrix}\right]\left[\begin{smallmatrix}
0&-1\\1&0
\end{smallmatrix}\right]
\left[\begin{smallmatrix}
e^{\frac{l^R}{2}}&0\\0&e^{\frac{-l^R}{2}}
\end{smallmatrix}\right]
\left[\begin{smallmatrix}
1&h_2\\0&1
\end{smallmatrix}\right]
\left[\begin{smallmatrix}
0&-1\\1&0
\end{smallmatrix}\right]
\left[\begin{smallmatrix}
1&-h_1\\0&1
\end{smallmatrix}\right]
\left[\begin{smallmatrix}
0&-1\\1&0
\end{smallmatrix}\right]
\left[\begin{smallmatrix}
e^{\frac{l^L}{2}}&0\\0&e^{\frac{-l^L}{2}}
\end{smallmatrix}\right]\\
=&\pm\left[
\begin{array}{cc}
e^{\frac{1}{2}(l^L+l^R)}(1+h_1h_2)(1+h_3h_4)+e^{\frac{1}{2}(l^L-l^R)}h_1h_4&*\\
*&e^{\frac{1}{2}(l^R-l^L)}h_2h_3+e^{-\frac{1}{2}(l^L+l^R)}
\end{array}
\right].
\end{align*}
We omit the non-diagonal entries because we only need to know the trace of this matrix, which is equal to $2\cosh(\tfrac{\ell_{\gamma}}{2})$. This yields \cref{eq:gntrace}.
\end{proof}

For $t\geq 0$, we consider the $e^t$-stretch along  $\Lambda_+$.  To simplify notation, we set  \[
l_{\gamma}(t):=\ell_{\gamma_m}(x_t),\ l^L(t):=l^L(x_t),\ l^R(t):=l^R(x_t)
\]
and for $j=1,\ldots, 4$, we let  $h_j(t)$
denote the arc-length of the $h_j$ horocyclic segment for $x_t\in\teich(S)$ under the same stretch. 
Note that the homotopy class of $\gamma_m$ is unaffected by the $e^t$-stretching with respect to $\Lambda_+$, and the structure of its homotopy representative in terms of stretch-invariant segments is preserved over the $e^t$-stretch path $x_t$ with respect to $\Lambda_+$. This tells us that \cref{eq:gntrace} is 
both well defined and true with  $l_\gamma(t), l^L(t)$ and $l^R(t)$.
Since $l^L(t)$ and $l^R(t)$ measure the lengths of segments on a leaf in $\Lambda_+$, they must be $e^t$-stretched and we have:
\[
l^L(t)=l^L(0)e^t\text{ and }
l^R(t)=l^R(0)e^t.
\]
Note in particular that 
\begin{align}
{(l^L)}'(0)=l^L(0)\text{ and }{(l^R)}'(0)=l^R(0).\label{eq:diff1}
\end{align}

\begin{lemma}  
When the surface is $e^t$-stretched along $\Lambda_+$,  we have, with the above notation:

\begin{align}
h_1(t)=\frac{e^{-e^t d_1(0)}+e^{-e^t d_2(0)}}{1-e^{-e^t l_{\gamma}(0)}},\label{eq:h_j}
\end{align}
where $d_1(0)$ and $d_2(0)$ are constants depending only on $h_1$ and $h_2$ respectively.

Furthermore, we have the following formula for the derivative of $h_1$ at $t=0$:
\begin{align}
h_1'(0)=
-\frac
{d_1(0)e^{-d_1(0)}+d_2(0)e^{-d_2(0)}}
{1-e^{-l_{\gamma}(0)}}
-\frac
{l_{\gamma}(0)e^{-l_{\gamma}(0)}(e^{-d_1(0)}+e^{-d_2(0)})}
{(1-e^{-l_{\gamma}(0)})^2}.\label{eq:diff2}
\end{align}
Similar expressions hold for $h_2(t)$, $h_3(t)$ and $h_4(t)$.
\end{lemma}

\begin{proof}

 First observe that the complement of $\Lambda_+$ on the pair of pants bounded by $\{\beta_1,\beta_2,\alpha\}$ consists of two ideal triangles $\triangle_1$ and $\triangle_2$. The set $h_1\cap\triangle_1$  consists of infinitely many connected components all of which are horocyclic segments. These horocyclic segments exponentially decrease in length, the closer a component of $h_1\cap\triangle_1$ is to $\alpha$. Let $d_1(t)$ denote the distance between the longest horocyclic subsegment of $h_1\cap\triangle_1$ and the nearest horocyclic boundary edge of the central stable triangle on $\triangle_1$ with respect to the metric $x_t\in\teich(S)$ restricted to $\triangle_1$.
 Then, the lengths of the segments comprising $h_1\cap\triangle_1$ are given by
\[
e^{-d_1(t)},\ e^{-d_1(t)-l_{\gamma}(t)},\ e^{-d_1(t)-2l_{\gamma}(t)},\ e^{-d_1(t)-3l_{\gamma}(t)},\ldots.
\]
Likewise consider the horocyclic segments which constitute the connected components of $h_1\cap\triangle_2$, and define $d_2(t)$ in the same way as $d_1(t)$.
Then the lengths of the segments  comprising $h_1\cap\triangle_2$ are then given by
\[
e^{-d_2(t)},\ e^{-d_2(t)-l_{\gamma}(t)},\ e^{-d_2(t)-2l_{\gamma}(t)},\ e^{-d_2(t)-3l_{\gamma}(t)},\ldots.
\]
When the surface is $e^t$-stretched along $\Lambda_+$, the horocyclic heights $d_i(0)+jl_{\gamma}(0)$ increase to $e^t(d_i(0)+jl_{\gamma}(0))$  and hence the length of $h_1(t)$, \ie the sum of the lengths of all these horocyclic segments, is given by
\[
h_1(t)=\frac{e^{-e^t d_1(0)}+e^{-e^t d_2(0)}}{1-e^{-e^t l_{\gamma}(0)}},
\]
which is \cref{eq:h_j}.
This is a function which monotonically decreases to $0$ as $t\to\infty$. The derivative of $h_1$ at $t=0$ is then given by
\[
h_1'(0)=
-\frac
{d_1(0)e^{-d_1(0)}+d_2(0)e^{-d_2(0)}}
{1-e^{-l_{\gamma}(0)}}
-\frac
{l_{\gamma}(0)e^{-l_{\gamma}(0)}(e^{-d_1(0)}+e^{-d_2(0)})}
{(1-e^{-l_{\gamma}(0)})^2},
\]
which is \cref{eq:diff2}.

\end{proof}

\begin{lemma}[maximal twist lemma]\label{lem:maxtwist}
Given simple closed geodesics $\beta_1,\beta_2,\beta_3$ and $\beta_4$ which bound a four-holed sphere $\Sigma$ on $(S,x)$ and a simple closed geodesic $\alpha$ in the interior of $\Sigma$, define $\lambda_+$ as above, and let $\Lambda_+$ be any complete geodesic lamination on $S$ which contains $\beta_1,\beta_2,\beta_3,\beta_4$ and $\lambda_+$. Further let $v_{\Lambda_+}\in T_x\teich(S)$ denote the stretch vector for $\Lambda_+$. Then there are no Thurston norm unit vectors $w\in\mathbf{S}_x \subsetneq T_x\teich(S)$ such that all the following conditions are satisfied:
\begin{enumerate}
\item
$w(\ell_{\alpha})=v_{\Lambda_+}(\ell_{\alpha})=\ell_{\alpha}(x)$, where $\ell_{\alpha}:\teich(S)\to\mathbb{R}$ is the hyperbolic length function for $\alpha$;

\item
for each $j=1,2,3,4$, $w(\ell_{\beta_j})=v_{\Lambda_+}(\ell_{\beta_j})=\ell_{\beta_j}(x)$; and

\item
$w(\tau_{\alpha})>v_{\Lambda_+}(\tau_{\alpha})$, where $\tau_{\alpha}:\teich(S)\to\mathbb{R}$ is the Fenchel--Nielsen twist function for $\alpha$.
\end{enumerate}
\end{lemma}

\begin{proof}
We show that any $w\in T_x\teich(S)$ satisfying the three listed properties necessarily has Thurston norm strictly larger than $1$. In particular, we show that there is some $\gamma_m$, the $m$-times iterated Dehn twist of the shortest closed geodesic $\gamma_0$ intersecting $\alpha$ at two points, for which  $\|w\|_{\mathrm{Th}}\geq w(\log\ell_{\gamma_m})>1$. 

Note that the numbers $\ell_{\beta_1},\ell_{\beta_2},\ell_{\beta_3},\ell_{\beta_4},\ell_{\alpha}$ and $\tau_{\alpha}$ are parameters of the Teichm\"{u}ller space of $\Sigma$.
This implies that $d\ell_{\gamma_m}$ is a linear combination of $d{\ell_\alpha}$, $d\ell_{\beta_1}, d\ell_{\beta_2}, d\ell_{\beta_3}, d\ell_{\beta_4}$ and $d{\tau_\alpha}$.
Since the $d\tau_\alpha$-component of $d\ell_{\gamma_m}$ is equal to $\partial_{\tau_\alpha}(\ell_{\gamma_n})$, we have
\[
w(\log\ell_{\gamma_m})
=\frac{w(\ell_{\gamma_m})}{\ell_{\gamma_m}}
=\frac{v_{\Lambda_+}(\ell_{\gamma_m})+\partial_{\tau_{\alpha}}(\ell_{\gamma_m})\epsilon}{\ell_{\gamma_m}},
\]
for $\epsilon=w(\tau_{\alpha})-v_{\Lambda_+}(\tau_{\alpha})>0$ (by condition $(3)$). This is possible because $w$ and $v_{\Lambda_+}$ agree with regard to their $d{\ell_\alpha},d{\ell_{\beta_1}}, d{\ell_{\beta_2}},d{\ell_{\beta_3}}$, and $d\ell_{\beta_4}$-components (see conditions~{(1)} and {(2)}), and only differ in their $d{\ell_{\tau_\alpha}}$-components (see condition~{(3)}).

We approximate the right-most term using Wolpert's cosine formula \cite[\S3]{Wol}: for $m\gg0$, the geodesic $\gamma_m$ intersects $\alpha$ at an angle close to $0$ (at both points), and hence $\partial_{\tau_{\alpha}}(\ell_{\gamma_m})\approx 2$. 
Thus, we may assume, by setting $m$ sufficiently high, that $\frac{\partial_{\tau_{\alpha}}(\ell_{\gamma_m})\epsilon}{\ell_{\gamma_m}}\approx\tfrac{2\epsilon}{\ell_{\gamma_m}}>\frac{\epsilon}{\ell_{\gamma_m}}$. 
Our goal therefore becomes to show the last inequality of 
\[
\|w\|_{\mathrm{Th}}\geq w(\log\ell_{\gamma_m})>\tfrac{v_{\Lambda_+}(\ell_{\gamma_m})+\epsilon}{\ell_{\gamma_m}(x)}>1.
\]
To this end, we now turn to computing $v_{\Lambda_+}(\ell_{\gamma_m})$.

Observe that $v_{\Lambda_+}(\ell_{\gamma_m})$ can be determined by taking $\ell_{\gamma_m}$ of the $e^t$-stretch path $x_t$, based at $x_0=x$, for  $\Lambda_+$ and computing the  derivative of ${l_{\gamma_m}(t):=\ell_{\gamma_m}(x_t)}$ at $t=0$. 
Therefore, we differentiate \cref{eq:gntrace} and substitute in \cref{eq:diff1} to replace $(l^L)'$ and $(l^R)'$. Due to the comparative complexity of \cref{eq:diff2}, we do not insert it directly into our present computations, but reserve this for a later step in our analysis. In any case, differentiating \cref{eq:gntrace} and rearranging slightly we obtain

\begin{align*}
v_{\Lambda_+}(\ell_{\gamma_m})
=&
\left.\frac{\tfrac{d}{dt} 2\cosh(\tfrac{1}{2}l_{\gamma_m}(t)))}{\sinh(\tfrac{1}{2}l_{\gamma_m}(t))}\right|_{t=0}\\
=&
\tfrac{(l^L(0)+l^R(0))}{2\sinh(\tfrac{1}{2}l_{\gamma_m}(0))}
\left(e^{\frac{1}{2}(l^L(0)+l^R(0))}(1+h_1(0)h_2(0))(1+h_3(0)h_4(0))\right.
\\
&
+\left.e^{\frac{1}{2}(l^L(0)-l^R(0))}h_1(0)h_4(0)
+e^{\frac{1}{2}(l^R(0)-l^L(0))}h_2(0)h_3(0)
+e^{-\frac{1}{2}(l^L(0)+l^R(0))}\right) \left(*\right)
\\
&
+\left.\tfrac{e^{\frac{1}{2}(l^L(0)+l^R(0))}}{\sinh(\tfrac{1}{2}l_{\gamma_m}(0))}
\right((h_1'(0)h_2(0)+h_1(0)h_2'(0))(1+h_3(0)h_4(0))
\\
&
+(1+h_1(0)h_2(0))(h_3'(0)h_4(0)+h_3(0)h_4'(0))
\\
&
+e^{-l^R(0)}(h_1'(0)h_4(0)+h_1(0)h_4'(0))+e^{-l^L(0)}(h_2'(0)h_3(0)+h_2(0)h_3'(0))
\\
&
-l^L(0)e^{-l^L(0)}h_2(0)h_3(0)-l^R(0)e^{-l^R(0)}h_1(0)h_4(0)
\\
&
\left.-(l^L(0)+l^R(0))e^{-(l^L(0)+l^R(0))}\right)  \left(**\right)
.
\end{align*}
The terms in the first two lines (*) in the above expression are positive, whereas the remaining lines (**) are all negative (since $h_j'(0)<0$). In fact, the total sum of the terms (*) is precisely equal to $(l^L(0)+l^R(0))\mathrm{coth}(\tfrac{1}{2}l_{\gamma_m}(0))>l^L(0)+l^R(0)$. 
On the other hand, by the triangle inequality, the difference between $l^L(0)+l^R(0)$ and $l_{\gamma_m}(0)$ is less than $h_1(0)+h_2(0)+h_3(0)+h_4(0)$. Moreover, recall from \cref{eq:h_j} that
\[
h_1(0)
=\frac{e^{-d_1(0)}+e^{-d_2(0)}}{1-e^{-l_{\gamma_m}(0)}}.
\]
As $m\to\infty$, both $d_1(0)$ and $d_2(0)$ grow with order $O(m\ell_{\alpha}(x))$, and we see therefore that $h_1(0)$ behaves as $O(e^{-m\ell_{\alpha}(x)})$ as $m\to\infty$. This order of growth with respect to $m$ holds for all four $h_j(0)$ terms. 
Its immediate consequence is that for $m$ sufficiently large, 
\begin{align}
(l^L(0)+l^R(0))\mathrm{coth}(\tfrac{1}{2}l_{\gamma_m}(0))+\tfrac{\epsilon}{2}
>l^L(0)+l^R(0)+\tfrac{\epsilon}{2}
>l_{\gamma_m}(0).
\end{align}
Let us now turn to the terms in  (**): the linear growth of $d_j$ with respect to $m$ ensures that $h_1'(0)$, as given by \cref{eq:diff2}, behaves as $O(me^{-m\ell_{\alpha}(x)})$. By symmetry, so do all four $h_j'(0)$ terms. This, coupled with the fact that 
\[
\tfrac{e^{\frac{1}{2}(l^L(0)+l^R(0))}}{\sinh(\tfrac{1}{2}l_{\gamma_m}(0))}\to1\text{ as }m\to\infty, 
\]
implies that the dominant term in the (**) summands is of order $O(me^{-m\ell_{\alpha}(x)})$, and this too shrinks to $0$ as $m\to\infty$. Therefore, for sufficiently large $m$,  $\frac{\epsilon}{2}$ suffices to cover the total negativity of the (**) summands. We have therefore shown that for sufficiently large $m$,
\[
v_{\Lambda_+}(\ell_{\gamma_m})+\epsilon=
\text{the sum of}\left(*\right)+\tfrac{\epsilon}{2}+\text{the sum of}\left(**\right)+\tfrac{\epsilon}{2}>\ell_{\gamma_m}(0):=\ell_{\gamma_m}(x),\]
as desired.
\end{proof}

\begin{remark}
\label{rmk:mintwist}
By symmetry, \cref{lem:maxtwist} tells us that there are no Thurston norm unit vectors $w\in\mathbf{S}_x \subsetneq T_x\teich(S)$ such that the following conditions are all satisfied
\begin{enumerate}
\item
$w(\ell_{\alpha})=v_{\Lambda_-}(\ell_{\alpha})=\ell_{\alpha}(x)$, where $\ell_{\alpha}:\teich(S)\to\mathbb{R}$ is the hyperbolic length function for $\alpha$;

\item
for each $j=1,2,3,4$, $w(\ell_{\beta_j})=v_{\Lambda_-}(\ell_{\beta_j})=\ell_{\beta_j}(x)$;

\item
$w(\tau_{\alpha})<v_{\Lambda_-}(\tau_{\alpha})$, where $\tau_{\alpha}:\teich(S)\to\mathbb{R}$ is a Fenchel--Nielsen twist function for $\alpha$.

\end{enumerate}

In particular, this follows directly from \cref{lem:maxtwist} by taking $\alpha^{-1}$ as input instead of $\alpha$.
\end{remark}

\begin{definition}[twist width]
We refer to $(v_{\Lambda_+}-v_{\Lambda_-})(\tau_{\alpha})$ as the \emph{twist width} of $\alpha$ with respect to $\beta_1,\beta_2,\beta_3,\beta_4$.
\end{definition}

\begin{lemma}
\label{lem:width}
The twist width of $\alpha$ with respect to $\beta_1,\beta_2,\beta_3,\beta_4$ is well defined, \ie it does not depend on the choice of the twist coordinate $\tau_\alpha$ or on that of the extension of $\lambda_\pm$ to  $\Lambda_\pm$.
\end{lemma}

\begin{proof}
Any two choices of Fenchel--Nielsen twist coordinates for $\alpha$ differ by some analytic function on Teichm\"uller space. The derivative of this function is cancelled out when taking the difference between $v_{\Lambda_+}$ and $v_{\Lambda_-}$ and it hence renders twist width unaffected by the choice of the twist coordinate.

Next, we observe that  \cref{lem:maxtwist} tells us that for any complete extension $\Lambda_+$, setting 
\begin{align*}
N:=&N_x([\alpha\sqcup\beta_1\sqcup\beta_2\sqcup\beta_3\sqcup\beta_4])\\
=&\left\{v\in\mathbf{S}_x
\mid
\iota_x([\alpha\sqcup\beta_1\sqcup\beta_2\sqcup\beta_3\sqcup\beta_4])(v)=\|v\|_{\mathrm{Th}}
\right\}, 
\end{align*}
we have
\[
v_{\Lambda_+}(\tau_{\alpha})
=
\max_{w\in N}
w(\tau_{\alpha}),
\]
and hence the value of $v_{\Lambda_+}(\tau_{\alpha})$ is independent of the choice of the extension from $\lambda_+$ to $\Lambda_+$. Likewise, \cref{rmk:mintwist} tells us that 
\[
v_{\Lambda_-}(\tau_{\alpha})
=
\min_{w\in N}
w(\tau_{\alpha}),
\]
and hence is also independent of the choice of the extension of $\lambda_\pm$ to $\Lambda_\pm$. Therefore, the twist width $(v_{\Lambda_+}-v_{\Lambda_-})(\tau_{\alpha})$ is also independent of the extension, and is well defined.
\end{proof}

\begin{definition}[$\epsilon$-slender pairs of pants]\label{defn:slender}
For a pair of pants $P$, denote the $\epsilon$-neighbourhood of $\partial P$ by $B(\partial P,\epsilon)$. For $\epsilon>0$, we say that $P$ is $\epsilon$-slender if and only if $P\setminus B(\partial P,\epsilon)$ consists of two connected components, each of which is either a hypercycle-bounded punctured monogon or triangle (see \cref{fig:34}). We further say that a sequence $\{P_i\}_{i\in\mathbb{N}}$ of pairs of pants is \emph{asymptotically slender} if and only if for every $\epsilon>0$, there is some index $I_\epsilon\in\mathbb{N}$ such that $P_i$ is $\epsilon$-slender for all $i>I_\epsilon$.
\end{definition}

\begin{figure}[h!]
\begin{center}
\includegraphics[scale=0.4]{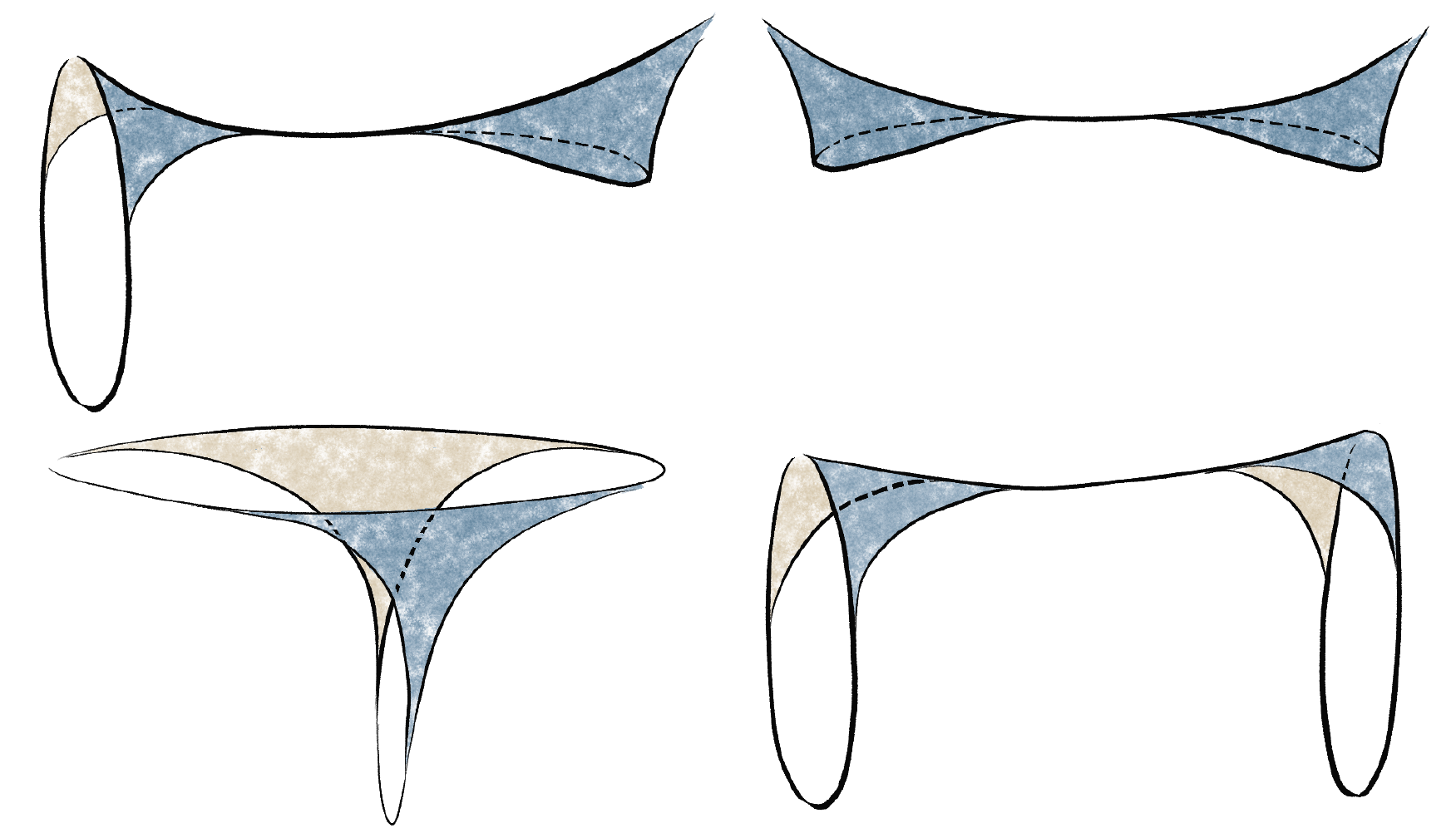}
\caption{The four types of $\epsilon$-slender pairs of pants.}
\label{fig:34}
\end{center}
\end{figure}

\begin{lemma}
\label{lem:shearblowup}
Consider an asymptotically slender sequence of pairs of pants $\{P_i\}_{i\in\mathbb{N}}$ and denote their boundaries as a union of (possibly length $0$) curves $\partial P_i=\alpha_i\sqcup\beta_i\sqcup\gamma_i$. Then, 
\[
\lim_{i\to\infty}
\min\left\{
\left|
\ell_{\alpha_i}
-
\ell_{\beta_i}
-
\ell_{\gamma_i}
\right|,
\left|
\ell_{\alpha_i}
-
\ell_{\beta_i}
+
\ell_{\gamma_i}
\right|,
\left|
\ell_{\alpha_i}
+
\ell_{\beta_i}
-
\ell_{\gamma_i}
\right|
\right\}
=\infty.
\]
\end{lemma}

\begin{proof}
For each pair of pants $P_i$, let $\sigma_{\alpha_i}$, $\sigma_{\beta_i}$ and $\sigma_{\gamma_i}$ respectively denote the  orthogeodesics on $P_i$ which respectively join $\beta_i$ to $\gamma_i$,  $\alpha_i$ to $\gamma_i$, and $\alpha_i$ to  $\beta_i$. By moving the endpoints of $\sigma_{\alpha_i}$, $\sigma_{\beta_i}$ and $\sigma_{\gamma_i}$ consistently  around $\alpha_i$, $\beta_i$ and $\gamma_i$ in one of the two directions and making them spiral infinitely, we obtain eight complete laminations of $P_i$. 
Let $\varsigma_{\alpha_i}, \varsigma_{\beta_i}$, and $\varsigma_{\gamma_i}$ be  leaves of $P_i$ corresponding to $\sigma_{\alpha_i}$, $\sigma_{\beta_i}$ and $\sigma_{\gamma_i}$ respectively, and $\lambda_i$ a complete geodesic lamination consisting of $\alpha_i, \beta_i, \gamma_i, \varsigma_{\alpha_i}, \varsigma_{\beta_i}$ and $\varsigma_{\gamma_i}$.
Denote the shearing parameters  $\varsigma_{\alpha_i}$, $\varsigma_{\beta_i}$ and $\varsigma_{\gamma_i}$ by $s_{\alpha_i}$, $s_{\beta_i}$ and $s_{\gamma_i}$ respectively. For either of the two choices of spiralling direction for $\alpha_i$, there are four possibilities for spiralling directions for the $\beta_i,\gamma_i$ pair. 
Since the length of a closed geodesic is equal to the absolute value of the sum of  shearing parameters of geodesics spiralling around it, (see e.g. \cite[\S4.2]{PT}), we see that there are four possible values for the shearing coordinate corresponding to $|s_{\alpha_i}|$ depending on spiralling, \ie
\begin{gather*}
\tfrac{1}{2}|\ell_{\alpha_i}+\ell_{\beta_i}+\ell_{\gamma_i}|,
\tfrac{1}{2}|\ell_{\alpha_i}-\ell_{\beta_i}-\ell_{\gamma_i}|,
\tfrac{1}{2}|\ell_{\alpha_i}-\ell_{\beta_i}+\ell_{\gamma_i}|,
\text {and }
\tfrac{1}{2}|\ell_{\alpha_i}+\ell_{\beta_i}-\ell_{\gamma_i}|.
\end{gather*}
The same holds also for $|s_{\beta_i}|$ and $|s_{\gamma_i}|$.
Therefore, by changing the directions of spiralling, it suffices to show that one of $|s_{\alpha_i}|, |s_{\beta_i}|, |s_{\gamma_i}|$ goes to $\infty$.

For small $\epsilon>0$,  the two ideal triangles in $P_i\setminus\lambda_i$ approximately agree with the two components of $P_i\setminus B(\partial P_i,\epsilon)$, and hence the centre of mass for the two central stable triangles lie in distinct components of $P_i\setminus B(\partial P_i,\epsilon)$. See \cref{fig:34} for a list of all possible configurations (up to permuting $\alpha_i,\beta_i$ and $\gamma_i$) for how these central stable triangles can be positioned relative to each other. We show that these central stable triangles may be set to be arbitrarily far apart as $i\to\infty$, which implies that the shearing coordinate of the leaf separating the two ideal triangles goes to $\infty$,  and hence that one of $|s_{\alpha_i}|, |s_{\beta_i}|, |s_{\gamma_i}|$ goes to $\infty$.
For simplicity, we consider the case when the leaf is $\varsigma_{\alpha_i}$.

For any $L>0$, take $\epsilon'>0$ sufficiently small so that the distance between the central stable triangle and the $\epsilon'$-thin part of each ideal triangle in $P_i\setminus\lambda_i$ is substantially larger than $L$. For any $\epsilon>0$, the condition of $\epsilon$-slenderness ensures that for $i>I_\epsilon$, the two components of $P\setminus B(\partial P,\epsilon)$ are separated by one or two orthogeodesics shorter than $2\epsilon$. Taking $\epsilon$ small enough separates the bulk of the $\epsilon'$-thick portion of each ideal triangle in $P_i\setminus\lambda_i$ from each other, thereby ensuring that the distance between their central stable triangles is at least $2L$, and hence that $|s_{\alpha_i}|>L$.
\end{proof}

\begin{theorem}[shrinking twist widths]\label{thm:shrink}
Consider a sequence of collections of curves $\{\alpha_i,\beta^L_i,\gamma^L_i,\beta^R_i,\gamma^R_i\}_{i\in\mathbb{N}}$ such that each of the two sequences $\{\alpha_i,\beta^L_i,\gamma^L_i\}_{i\in\mathbb{N}}$ and $\{\alpha_i,\beta^R_i,\gamma^R_i\}_{i\in\mathbb{N}}$ constitutes the boundary geodesics for an asymptotically slender sequence of pairs of pants $P_i^L, P_i^R$ on $(S,x)$ respectively.
Then the twist width of $\alpha_i$ with respect to $\beta^L_i,\gamma^L_i,\beta^R_i,\gamma^R_i$ tends to $0$.
\end{theorem}

\begin{proof}
Take two sequences of complete geodesic laminations $\{\Lambda_{+,i}\}_{i\in\mathbb{N}}$ and $\{\Lambda_{-,i}\}_{i\in\mathbb{N}}$ which extend $\alpha_i\sqcup \beta^L_i\sqcup \gamma^L_i\sqcup \beta^R_i\sqcup \gamma^R_i$ and  satisfy the conditions of \cref{lem:maxtwist}. Our goal is to show that
\[
\lim_{i\to\infty}
(v_{\Lambda_{+,i}}-v_{\Lambda_{-,i}})(\tau_{\alpha_i})=0.
\]

Each of the laminations $\lambda_{\pm, i}$ consists of $\alpha_i$ and two geodesics spiralling around $\alpha_i$, which we denote by $\sigma^L_{1,\pm}$ for the one on the left  and $\sigma^R_{1,\pm}$ for the one on the right.
We recall from \cref{lem:width} that we may use these $\lambda_{\pm,i}$ and can choose
any extension of $\lambda_{\pm,i}$ to $\Lambda_{\pm,i}$, provided that each of $\Lambda_{\pm,i}$ is complete and contains each of $\lambda_{\pm,i}$ respectively and $\beta^L_i,\gamma^L_i,\beta^R_i,\gamma^R_i$.
Choose the extension of $\lambda_{\pm,i}$ to $\Lambda_{\pm,i}$ as in \cref{fig:37} by adding  geodesics $\sigma^L_{\pm,2}$ spiralling around $\alpha_i, \gamma_i^L$, and $\sigma^L_{\pm,4}$ spiralling around $\alpha_i, \beta_i^L$ on the left side of each of $\lambda_{\pm,i}$, and in the same way, $\sigma^R_{\pm,2}, \sigma^R_{\pm,4}$ on the right side of each $\lambda_{\pm, i}$. 
For each of $\sigma^L_{\pm, j}$ and $\sigma^R_{\pm,j}$ $(j=1,2,4)$, we denote its shearing parameter by $s^L_{\pm,j}$ and $s^R_{\pm,j}$.
Since we need to consider two ends of $\sigma^L_{\pm, 1}$ both spiralling around $\alpha_i$, and  contributions of both, we define $s^L_{\pm, 3}, s^R_{\pm, 3}$ to be the same as $s^L_{\pm,1}, s^R_{\pm,1}$ respectively.

Now, taking a covering associated with $\alpha_i$, we get a $(4,4)$-crowned annulus.
Then, we know from \cref{lem:gencrown} that $(v_{\Lambda_{+,i}}-v_{\Lambda_{-,i}})(\tau_{\alpha_i})$ at $x\in\teich(S)$ is a linear combination of derivatives (with respect to $t$ at $t=0$) of
\begin{itemize}
\item \textbf{type one terms}
\begin{align*}
\log(1-e^{-e^t\ell_{\alpha_i}})-e^t\log(1-e^{-\ell_{\alpha_i}})\text{, and}
\end{align*}
\item \textbf{type two terms}
\begin{gather*}
\begin{align*}
e^t&\log\left(1+e^{-s_1(x)}+e^{-(s_1(x)+s_2(x))}+e^{-(s_1(x)+s_2(x)+s_{3}(x))}\right)\\
-&\log\left(1+e^{-e^t s_1(x)}+e^{-e^t(s_1(x)+s_2(x))}+e^{-e^t(s_1(x)+s_2(x)+s_{3}(x))}\right),
\end{align*}
\end{gather*}
\end{itemize}
where  $s_1(x), s_2(x), s_3(x)$ are one of  $s_{\pm,1}^L(x), s^L_{\pm 2}, s^L_{\pm,3}(x)$ and $s^R_{\pm,1}, s^R_{\pm,2} , s_{\pm,3}^R$, where the signs are the same for all three.

The condition that $\{P_i^L\}$ and $\{P_i^R\}$ are asymptotically slender tells us that either $\ell_{\alpha_i}\to\infty$ or $\ell_{\alpha_i}\to0$. Due to the discreteness of the (simple) length spectrum of $(S,x)$, this can only happen if $\ell_{\alpha_i}$ eventually equals $0$. However, this then contradicts $\alpha_i$ being an interior geodesic rather than a boundary cusp. 
Thus we have  $\ell_{\alpha_i}\to\infty$.
Now employing the newly established fact that $\ell_{\alpha_i}\to\infty$, we see that the derivatives at $t=0$ of the type one terms, which take the form 
\begin{align}
\frac{\ell_{\alpha_i} e^{-\ell_{\alpha_i}}}{1-e^{-\ell_{\alpha_i}}}-\log(1-e^{-\ell_{\alpha_i}}),\label{eq:termtype1}
\end{align}
necessarily tend to $0$ as $i\to\infty$.

We next consider the type two terms. 
%
%
%
We work with the left pair of pants bordered by $\alpha_i,\beta^L_i,\gamma^L_i$.
The right pair of pants is similarly dealt with. 
Choose the extension of $\lambda_{\pm,i}$ to $\Lambda_{\pm,i}$ as explained before fixing orientations on $\beta_i^L, \gamma_i^L$ and making $\sigma^L_{\pm,2}$ and $\sigma^L_{\pm,4}$ spiral towards the given orientations.
Then, we have
\begin{align*}
s_{+,2}^L&=\ell_{\gamma^L_i},\quad
s_{+,4}^L=\ell_{\beta^L_i},\quad
s_{+,1}^L=s_{+,3}^L=\tfrac{1}{2}(\ell_{\alpha_i}-\ell_{\beta^L_i}-\ell_{\gamma^L_i}),\text{ and}\\
s_{-,2}^L&=-\ell_{\gamma^L_i},\quad
s_{-,4}^L=-\ell_{\beta^L_i},\quad
s_{-,1}^L=s_{-,3}^L=\tfrac{1}{2}(\ell_{\alpha_i}+\ell_{\beta^L_i}+\ell_{\gamma^L_i}).
\end{align*}

We first analyse the $\Lambda_{-,i}$ case. For $\Lambda_{-,i}$ (see \cref{fig:37}), the shearing coordinates in question correspond to
\begin{align*}
s_{-,1}^L&=\tfrac{1}{2}(\ell_{\alpha_i}+\ell_{\beta^L_i}+\ell_{\gamma^L_i}),\\
s_{-,1}^L+s_{-,2}^L&=\tfrac{1}{2}(\ell_{\alpha_i}-\ell_{\beta^L_i}+\ell_{\gamma^L_i}),\\
s_{-,1}^L+s_{-,2}^L+s_{-,3}^L&=\ell_{\alpha_i}+\ell_{\gamma^L_i}.
\end{align*}

\begin{figure}[h!]
\begin{center}
\includegraphics[scale=1]{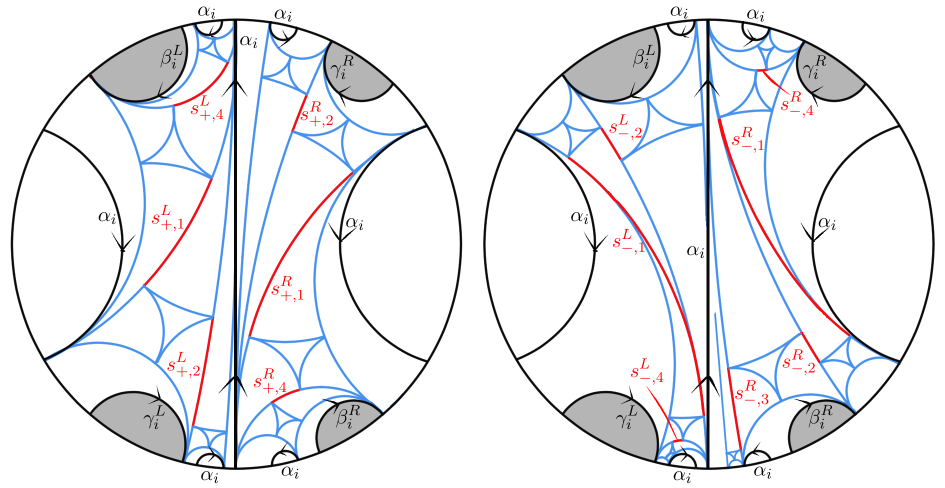}
\caption{A depiction of the shearing parameters  $s_{\pm,j}^L$ and  $s_{\pm,j}^R$ for  $\Lambda_{+,i}$ (left) and $\Lambda_{-,i}$ (right).}
\label{fig:37}
\end{center}
\end{figure}

Both $s_{-,1}^L$ and $s_{-,1}^L+s_{-,2}^L+s_{-,3}^L$ are positive and tend to $\infty$ (because $\ell_{\alpha_i}\to\infty$). By \cref{lem:shearblowup}, the absolute value of $s_{-,1}^L+s_{-,2}^L$ tends to $\infty$ as $i\to\infty$. If $s_{-,1}^L+s_{-,2}^L$ is positive, then the derivative of the type two terms for $\Lambda_{-,i}$ takes the form
\begin{equation}
\begin{split}
&\log\left(1+e^{-s_{-,1}^L}+e^{-(s_{-,1}^L+s_{-,2}^L)}+e^{-(s_{-,1}^L+s_{-,2}^L+s_{-,3}^L)}\right)\\
&+\frac{s_{-,1}^Le^{-s_{-,1}^L}+(s_{-,1}^L+s_{-,2}^L)e^{-(s_{-,1}^L+s_{-,2}^L)}+(s_{-,1}^L+s_{-,2}^L+ s_{-,3}^L)e^{-(s_{-,1}^L+s_{-,2}^L+s_{-,3}^L)}}{1+e^{-s_{-,1}^L}+e^{-(s_{-,1}^L+s_{-,2}^L)}+e^{-(s_{-,1}^L+s_{-,2}^L+s_{-,3}^L)}},
\label{eq:derivative}
\end{split}
\end{equation}
and $s_{-,1}^L, s_{-,1}^L+s_{-,2}^L$ and $s_{-,1}^L+s_{-,2}^L+s_{-,3}^L$ go to $\infty$. 

If $s_{-,1}^L+s_{-,2}^L$ is negative, then by adding  $e^t(s_{-,1}^L+s_{-,2}^L)$ to each term, which cancels out after the subtraction, the type two terms take the form
\begin{equation}
\begin{split}
e^t&\log\left(1+e^{-(s_{-,1}^L-(s_{-,1}^L+s_{-,2}^L))}+e^{(s_{-,1}^L+s_{-,2}^L)}+e^{-((s_{-,1}^L+s_{-,2}^L+s_{-,3}^L)-(s_{-,1}^L+s_{-,2}^L))}\right)\\
-&\log\left(1+e^{-e^t(s_{-,1}^L-(s_{-,1}^L+s_{-,2}^L))}+e^{e^t(s_{-,1}^L+s_{-,2}^L)}+e^{-e^t((s_{-,1}^L+s_{-,2}^L+s_{-,3}^L)-(s_{-,1}^L+s_{-,2}^L))}\right),
\end{split}
\end{equation}
and note that $(s_{-,1}^L-(s_{-,1}^L+s_{-,2}^L)), -(s_{-,1}^L+s_{-,2}^L),((s_{-,1}^L+s_{-,2}^L+s_{-,3}^L)-(s_{-,1}^L+s_{-,2}^L))$  all go to $\infty$ as $i \to\infty$. Thus, each of the derivatives at $t=0$ of the type two terms may always be written in the form 
\begin{align}
\log\left(1+e^{-u}+e^{-v}+e^{-w}\right)+\frac{ue^{-u}+ve^{-v}+we^{-w}}{1+e^{-u}+e^{-v}+e^{-w}},\label{eq:genericform}
\end{align}
for $u,v,w$ which go to $\infty$ as $i\to\infty$, and hence each derivative term also tends to $0$.

We finally consider the case of $\Lambda_{+,i}$, which is more delicate than that of $\Lambda_{-,i}$. Our present goal is to show also in this setting that the derivatives of the type two terms tends to $0$. \cref{fig:37} tells us that
\begin{align*}
s_{+,1}^L&=\tfrac{1}{2}(\ell_{\alpha_i}-\ell_{\beta^L_i}-\ell_{\gamma^L_i}),\\
s_{+,1}^L+s_{+,2}^L&=\tfrac{1}{2}(\ell_{\alpha_i}-\ell_{\beta^L_i}+\ell_{\gamma^L_i}),\\
s_{+,1}^L+s_{+,2}^L+s_{+,3}^L&=\ell_{\alpha_i}-\ell_{\beta^L_i}.
\end{align*}
It is possible for any of these terms to be positive or negative, and hence we invoke \cref{lem:shearblowup} to see that precisely one of the following options holds for sufficiently large $i$.
\begin{enumerate}[I.]
\item
$\ell_{\alpha_i}\gg\ell_{\beta^L_i}+\ell_{\gamma^L_i}$;
\item
$\ell_{\beta^L_i}\gg\ell_{\alpha_i}+\ell_{\gamma^L_i}$;
\item
$\ell_{\gamma^L_i}\gg\ell_{\alpha_i}+\ell_{\beta^L_i}$;
\item
$\ell_{\alpha_i}\ll\ell_{\beta^L_i}+\ell_{\gamma^L_i}$, $\ell_{\beta^L_i}\ll\ell_{\alpha_i}+\ell_{\gamma^L_i}$, and $\ell_{\gamma^L_i}\ll\ell_{\alpha_i}+\ell_{\beta^L_i}$,
\end{enumerate} 
and to hence deal with this problem on a case-by-case basis.\medskip

For Case~I, we leave the expression for the type two terms unchanged. All three sums $s_{+,1}^L$, $s_{+,1}^L+s_{+,2}^L$ and $s_{+,1}^L+s_{+,2}^L+s_{+,3}^L$ tend to $\infty$ as $i\to\infty$, which leads to \cref{eq:derivative}, with the $-$-subscripts changed to $+$,  tending to $0$ as $i \to \infty$.

For Case~II, by adding $e^t(\ell_{\alpha_i}-\ell_{\beta^L_i})$ to each term, which cancels out after the subtraction, the type two terms  are re-expressed as:
\begin{align*}
e^t&\log\left(1
+e^{-\tfrac{1}{2}(\ell_{\beta^L_i}-\ell_{\alpha_i}-\ell_{\gamma^L_i})}
+e^{-\tfrac{1}{2}(\ell_{\beta^L_i}-\ell_{\alpha_i}-\ell_{\gamma^L_i})-\ell_{\gamma^L_i}}
+e^{-(\ell_{\beta^L_i}-\ell_{\alpha_i}-\ell_{\gamma^L_i})-\ell_{\gamma^L_i}}\right)\\
-&\log\left(1
+e^{-\tfrac{e^t}{2}(\ell_{\beta^L_i}-\ell_{\alpha_i}-\ell_{\gamma^L_i})}
+e^{-\tfrac{e^t}{2}(\ell_{\beta^L_i}-\ell_{\alpha_i}-\ell_{\gamma^L_i})-e^t\ell_{\gamma^L_i}}
+e^{-e^t(\ell_{\beta^L_i}-\ell_{\alpha_i}-\ell_{\gamma^L_i})-e^t\ell_{\gamma^L_i}}\right).
\end{align*}
As before, its derivative at $t=0$ is expressed in the form of \cref{eq:genericform} with $u,v,w\to\infty$ as $i\to\infty$, and hence tends to $0$ as $i\to\infty$.

For Case~III, we likewise re-express the type two terms as:
\begin{align*}
e^t&\log\left(1
+e^{-\tfrac{1}{2}(\ell_{\gamma^L_i}-\ell_{\alpha_i}-\ell_{\beta^L_i})-\ell_{\beta^L_i}}
+e^{-\tfrac{1}{2}(\ell_{\gamma^L_i}-\ell_{\alpha_i}-\ell_{\beta^L_i})-(\ell_{\alpha_i}+\ell_{\beta^L_i})}
+e^{-(\ell_{\gamma^L_i}-\ell_{\alpha_i}-\ell_{\beta^L_i})-(\ell_{\alpha_i}+\ell_{\beta^L_i})}\right)\\
-&\log\left(1
+e^{-\tfrac{e^t}{2}(\ell_{\gamma^L_i}-\ell_{\alpha_i}-\ell_{\beta^L_i})-e^t\ell_{\beta^L_i}}
+e^{-\tfrac{e^t}{2}(\ell_{\gamma^L_i}-\ell_{\alpha_i}-\ell_{\beta^L_i})-e^t(\ell_{\alpha_i}+\ell_{\beta^L_i})}
+e^{-e^t(\ell_{\gamma^L_i}-\ell_{\alpha_i}-\ell_{\beta^L_i})-e^t(\ell_{\alpha_i}+\ell_{\beta^L_i})}\right).
\end{align*}
This again suffices to ensure that the derivative at $t=0$ of this term tends to $0$ as $i\to\infty$.

Finally, for Case~IV, we re-express the type two terms as:
\begin{align*}
e^t&\log\left(1
+e^{-\tfrac{1}{2}(\ell_{\beta^L_i}+\ell_{\gamma^L_i}-\ell_{\alpha_i})}
+e^{-\tfrac{1}{2}((\ell_{\alpha_i}+\ell_{\gamma^L_i}-\ell_{\beta^L_i})-(\ell_{\beta^L_i}+\ell_{\gamma^L_i}-\ell_{\alpha_i}))}
+e^{-\tfrac{1}{2}(\ell_{\alpha_i}+\ell_{\gamma^L_i}-\ell_{\beta^L_i})}\right)\\
-&\log\left(1
+e^{-\tfrac{e^t}{2}(\ell_{\beta^L_i}+\ell_{\gamma^L_i}-\ell_{\alpha_i})}
+e^{-\tfrac{e^t}{2}((\ell_{\alpha_i}+\ell_{\gamma^L_i}-\ell_{\beta^L_i})-(\ell_{\beta^L_i}+\ell_{\gamma^L_i}-\ell_{\alpha_i}))}
+e^{-\tfrac{e^t}{2}(\ell_{\alpha_i}+\ell_{\gamma^L_i}-\ell_{\beta^L_i})}\right).
\end{align*}
And this again suffices to ensure that the derivative at $t=0$ of this term tends to $0$ as $i\to\infty$.

We therefore see that all the type one and type two terms associated with the left pair of pants tend to $0$. By symmetry, this is true also for the terms associated with the right pair of pants. Therefore, the twist width of $\alpha_i$ with respect to $\beta^L_i,\gamma^L_i,\beta^R_i,\gamma^R_i$ tends to $0$.
\end{proof}

 \section{The geometry of convex bodies}
\label{sec:convexbodies}
Throughout this section, we consider a convex body, \ie a compact convex subset with non-empty interior, of an $n$-dimensional vector space $\mathbb V$.
We shall introduce  several geometric/combinatorial structures associated with  points on the boundary of the convex body $D$. These include notions such as the \emph{dimension} of a point, \emph{face-dimension}, \emph{adherence},  \emph{adherence-dimension},  \emph{codimension} and a few others. Several among these concepts are new, they apply to arbitrary compact convex bodies in finite-dimensional vector spaces, and may be of general interest. We later apply them to points in $\mathbf{S}_x^*:=\iota_x(\pml(S))\subsetneq T^*_x\teich(S)$ and its dual unit sphere in the tangent space, thereby laying the groundwork for establishing topological rigidity, \cref{first part}. At the same time, we shall recall some classical notions from convexity theory.
We start by recalling the important notion of support hyperplane for the convex body $D$.

\begin{definition}[support hyperplane]
\label{defn:hyperplane}
For a point $P$ in the boundary $\partial D$  of a convex bodyt $D\subsetneq \mathbb{V}$, an affine hyperplane $\pi\subset\mathbb{V}$ is a \emph{support hyperplane} to $\partial D$ at $P$ if the following hold:
\begin{itemize}
\item
$\pi$ contains $P$ and 
\item
one of the two closed half-spaces bounded by $\pi$ contains $D$.
\end{itemize}
\end{definition}

\subsection{Dimension and faces}

\begin{definition}[dimension of a point]
\label{defn:dimension}
For any point $P$ on $\partial D$, let  $\mathbb{A}_P$ denote a maximal-dimensional affine subspace of $\mathbb{V}$ such that $P$ is contained in the interior of the set $\mathbb{A}_P\cap \partial D$, where the latter is regarded as a subset of $\mathbb{A}_P$ (endowed with the standard topology on affine subspaces). We define the \emph{dimension of $P$ (with respect to $\partial D$)} as the dimension of $\mathbb{A}_P$:
\[
\dim_{\partial D}(P)
:=
\dim \mathbb{A}_P.
\]
\end{definition}

\begin{remark}[Example]
Consider the convex body $D=[0,2]\times[0,2]\subset\mathbb{R}^2$. We first note (somewhat trivially) that the dimension of any point on $\partial D$ must be strictly less than $2$ as they are not in the interior of $D$. For the point $P=(0,1)$, the $y$-axis $\{x=0\}$ is a $1$-dimensional (affine) subspace such that the intersection $\{x=0\}\cap \partial D$ is a closed interval containing $P$ as an interior point and hence $\dim_{\partial D}(P)=1$. For $Q=(0,0)$, the intersection $\mathbb{L}\cap \partial D$ of every line $\mathbb{L}$ through $Q$, besides the $x$ and $y$ axes, with $\partial D$ is simply equal to $\{Q\}$, which is a nowhere-dense subset of $\mathbb{L}$. The two exceptional lines which pass through $Q$ are the $x$ and $y$ axes, and their intersections with $\partial D$ is a closed interval having $Q$ as a boundary (and non-interior) point. Therefore, $Q$ has dimension strictly less than $1$, \ie $\dim_{\partial D}(Q)=0$.

\end{remark}

\begin{lemma}[alternative formulation of dimension]
\label{lem:alternative}
\cref{defn:dimension} may be alternatively, and equivalently, formulated by replacing $\mathbb{A}_P\cap \partial D$ with $\mathbb{A}_P\cap D$ in the statement of this definition.
\end{lemma}

\begin{proof}
Let $\mathbb{A}$ denote a maximal-dimensional affine subspace of $\mathbb{V}$ such that $P$ is contained in the interior of the set $\mathbb{A}\cap \partial D$, where the latter is regarded as a subset of $\mathbb{A}$. We first note that, by definition, there is an open ball $B$ in $\mathbb{A}$ so that
\[
P\in B\subset \mathbb{A}\cap \partial D\subseteq \mathbb{A}\cap D,
\]
and $P$ is an interior point of $\mathbb{A}\cap D$ (regarded as a subset $\mathbb{A}$). Hence the notion of dimension defined in \cref{defn:dimension} is smaller than or equal to the alternative proposed here.

Conversely, consider any maximal-dimensional affine subspace $\mathbb{A}'$ containing $P$ as an interior point of $\mathbb{A}'\cap D$, with $\mathbb{A}'\cap D$ regarded as a subset of $\mathbb{A}'$. 
This means that there is some open ball $B'\subset \mathbb{A}'\cap D$ in $\mathbb{A}'$ containing $P$. The existence of a supporting hyperplane to $D$ at $P\in\partial D$ asserts that there is some linear function $w^*:\mathbb{V}\to\mathbb{R}$ such that $P\in\ker w^*$ and the interior of $D$ is a subset of the preimage $(w^*)^{-1}((0,\infty))$. Considering the restriction of $w^*$ on $\mathbb A'$, this implies that either 
$
\mathbb{A}'\cap D=\mathbb{A}'\cap \partial D,
$
or $\mathbb{A}'\cap D$ contains an interior point $Q\in D\setminus \partial D$ of $D$, and hence $Q$ is not an element of the closed affine half-space $(w^*)^{-1}((-\infty,0])$. 
This  means that $P$ lies in the affine half-space $(w^*)^{-1}((-\infty,0])\cap \mathbb{A}'$ but $Q$ does not, thereby indicating that $P$ is a boundary point of $\mathbb{A}'\cap D$. This contradicts the original assumption that $P$ is an interior point of $\mathbb{A}'\cap D$, and hence excludes the latter possibility. Therefore, we have
$
P\in B'\subset \mathbb{A}'\cap D=\mathbb{A}'\cap \partial D,
$
and we see that the two notions of dimension are equal.
\end{proof}

\begin{definition}[face for a point]
\label{defn:face}
Given a point $P\in\partial D$, we refer to the intersection $\mathbb{A}_P\cap\partial D$ of any maximal dimensional affine subspace $\mathbb{A}_P$ (constructed as in \cref{defn:dimension}) with $\partial D$ as a \emph{face (of $\partial D$) for $P\in \partial D$}. We shall also use the following terminology:
\begin{itemize}
\item The dimension of the face $\mathbb A_P \cap \partial D$ is defined to be the dimension of $\mathbb A_P$.
\item
We refer to the collection of faces $\mathbb{A}_P\cap\partial D$ for $P\in\partial D$ as the collection of faces on $\partial D$.
\item
We refer to any face that  is a subset of another face as a \emph{subface}.
\item
For each face $F$ of $\partial D$, we call the subset of points in $F$ which have an open-ball neighbourhood (with respect to the topology induced on the face) contained in $F$, the \emph{interior} of the face.
\end{itemize}
\end{definition}
It will turn out in \cref{lem:uniqueface} that for any $P \in \partial D$, its face is unique, independent of the choice of $\mathbb A_P$.

We next show that the above notion of face is equivalent to the one found in the literature on convex geometry  (see, e.g. \cite[p. ~162]{Rock}):

\begin{definition}[face]
\label{defn:classicalface}
A convex subset $F\subset D$ is called a \emph{face} of $D$ if and only if for every $x\in F$ and every $y,z\in D$ such that $x$ lies on the open interval between $y$ and $z$, we have $y,z\in F$.
A face $F$ is said to be of $P$ when $P$ is an interior point of $F$, \ie for any $Q \in F$ there is a segment with an endpoint at $Q$ which contains $P$ in its interior.
\end{definition}

We shall show that the two notions are equivalent, \ie a face in the sense of \cref{defn:classicalface} is always a face of some $P \in \partial D$ in the sense of \cref{defn:face}  if we ignore the special case where $F=D$, which is permitted under \cref{defn:classicalface} but not \cref{defn:face}.

\begin{proof}[Proof that the two notions of faces agree]
Consider a face $F\subset \partial D$ in the sense of \cref{defn:face}, and let $P\in F$ be an interior point of $F$. We assume (for a proof by contradiction) that $F$ is not a face in the classical sense (\cref{defn:classicalface}) and so there is some $x\in F$ and $y,z\in D$ such that $x$ lies on the open interval between $y$ and $z$, but neither $y$ nor $z$ lies in $F$.
Note that it is impossible to have  $y\in F$ and $z\notin F$ (or $z \in F$ and $y\notin F$) because the fact that $x$ lies on the line joining $y$ and $z$ would then mean that $z$ lies on a line generated by two elements of $F$ and hence would lie on the affine space $\mathbb{A}$ generated by $F$ and hence would lie in $F=\mathbb{A}\cap D$.

Suppose that neither $y$ nor $z$ is contained in $F$.
Then the convex hull of $F\cup\{y,z\}$ is strictly larger than $F$.
Let $\mathbb{A}'$ be the affine space generated by (the convex hull of) $F\cup\{y,z\}$. 
Then, $\mathbb{A}'$ properly contains $\mathbb{A}$, and hence $\mathbb{A}'$ has dimension strictly greater than $\mathbb{A}$ by $1$ (as the line joining $y$ and $z$ transverse to $F$ is only $1$-dimensional). This in turn means that the convex hull of $F\cup\{y,z\}$ is a topological closed ball of one dimension higher than $F$, and  hence $\mathbb{A}'\cap D$, which necessary contains the convex hull of $F\cup \{y,z\}$, must have the same dimension as $\mathbb{A}'$. In particular, $P$ must be an interior point of $\mathbb{A}'\cap D$ (or it would be a boundary point of $F$, thereby contradicting \cref{defn:face}). This contradicts the maximality of $\mathbb{A}$, thereby showing that our assumption for contradiction is false, and that $F$ must be a face in the classical sense.

Conversely, given a face $F\subset D$ in the classical sense (\cref{defn:classicalface}) which is not equal to $D$ itself, let $\mathbb{A}$ denote the affine subspace generated by $F$. We first note that $F$ is necessarily equal to $\mathbb{A}\cap D$; for otherwise, we would be able to produce an interval with one endpoint based at an interior point of $F$ and the other endpoint in $\mathbb{A}\cap D\setminus F$, thereby contradicting \cref{defn:classicalface}. Since $F$ is (by definition) convex and closed (a consequence of the definition), it is a topological closed ball of the same dimension as $\mathbb{A}$. 
Let $P\in F$ be an interior point of $F$, regarded as a subset of $\mathbb{A}$. We need to show that there is no affine space $\mathbb{A}'$ strictly larger than $\mathbb{A}$ such that $\mathbb{A}'\cap D$ contains $P$ as an interior point of $\mathbb{A}'\cap D$, the latter regarded as a subset of $\mathbb{A}'$. If there is such an affine space $\mathbb A'$, then (using the fact that $P$ is such an interior point) we can find a small interval in $\mathbb{A}'\cap D$ which has non-empty transverse intersection with $F=\mathbb{A}\cap D$. This then contradicts the classical definition (\cref{defn:classicalface}) of a face, and hence $\mathbb{A}$ is indeed maximal. Therefore, $F$ is a face for $P$, in the sense of \cref{defn:face}.
\end{proof}

\begin{remark}\label{faceintersection}
It is clear from either \cref{defn:face} or \cref{defn:classicalface} that any non-trivial intersection of a collection of faces is itself a face.
\end{remark}

We now show that every point $P$ on $\partial D$ has a unique face for $P$. This will form the basis for the next notion of dimension which we wish to introduce. 

\begin{lemma}[face uniqueness]
\label{lem:uniqueface}
For every $P\in \partial D$, there is a unique face for $P$ in the sense of \cref{defn:face}. 
\end{lemma}

\begin{proof}
Let $\mathbb{A}_P$ and $\mathbb{A}_P'$ be affine subspaces of $\mathbb{V}$ which satisfy the conditions stated in the alternative formulation of \cref{defn:dimension} obtained by replacing $\partial D$ with $D$ via \cref{lem:alternative}. Let $\mathbb A_P''$ be the convex hull of $\mathbb{A}_P \cup \mathbb{A}_P'$. 
This is an affine subspace, and we now show that $\mathbb{A}_P''$ also satisfies the conditions for (the alternative formulation of) \cref{defn:dimension}. 
By assumption, there exist sets $B\subset \mathbb{A}_P\cap D$ and $B'\subset\mathbb{A}_P'\cap D$ which contain $P$ and are open subsets respectively in $\mathbb{A}_P$ and $\mathbb{A}_P'$. In particular, we may choose $B$ and $B'$ to be convex open balls in $\mathbb A_P$ and $\mathbb A_{P'}$ respectively. This means that the convex hull of $B\cup B'$, which we denote by $B''$, is a convex open ball in $\mathbb{A}_P''$ containing $P$. Furthermore, since $D$ is convex, it must contain $B''$, and hence $B''$ is a subset of $\mathbb{A}_P''\cap D$, which is open in $\mathbb{A}_P''$. 
If $\mathbb A_{P''}$ has dimension higher than that of $\mathbb A_P$ (resp. $\mathbb A_{P'}$), it contradicts the maximality of the dimension of $\mathbb A_P$ (resp. $\mathbb A_{P'}$) contained in \cref{defn:dimension}.
Therefore, we have
$
\mathbb{A}_P=\mathbb{A}_P''=\mathbb{A}_P',
$
which implies the uniqueness of $\mathbb{A}_P$, and hence of the face $\mathbb{A}_P\cap \partial D=\mathbb{A}_P\cap D$.
\end{proof}

\begin{definition}[convex stratification]\label{defn:convexstrat}
\cref{lem:uniqueface} canonically partitions $\partial D$ into the interiors of faces. We refer to this decomposition into strata (\ie the interiors of faces) as the \emph{convex stratification} of $\partial D$. Note that the strata are convex and hence are open cells. 
\end{definition}

\begin{remark}
\label{rem:uniqueaffine}
The proof of \cref{lem:uniqueface} shows that not only is the face for $P$ unique, but the affine subspace $\mathbb{A}_P$ containing the face for $P$ is also unique. In particular, the fact that $P$ is an interior point of $\mathbb{A}_P\cap \partial D$ with respect to the topology on $\mathbb{A}_P$ implies that $\mathbb{A}_P$ is the affine subspace generated by the face for $P$.
\end{remark}

\begin{lemma}
\label{thm:facesupport}
Every support hyperplane at $P$ contains the face for $P$.
\end{lemma}

\begin{proof}
 Let $\mathbb{A}_P$ be the subspace passing through $P$ used to define the face for $P$ in \cref{defn:dimension}, and $\mathbb{A}_P^0$ its translate under the translation sending $P$ to $0$, \ie $\mathbb A_P^0:=\mathbb A_P- P$. We claim that any arbitrary support hyperplane $\pi$ at $P$ necessarily contains $\mathbb{A}_P$. 
 We also consider the translate $\pi^0:=\pi-P$ of $\pi$.
 We have only to show that $\pi^0$ contains $\mathbb A_P^0$.
 If not, then there is some vector $v\in\mathbb{A}_P^0\setminus\pi^0$, having all of its non-zero multiples in $\mathbb{A}_P^0$ but not in $\pi^0$. In particular, we may replace $v$ with its sufficiently small scalar multiple so that both $\pm v$ are contained in an open ball around $\vec{0}$ in  $F-P=(\mathbb{A}_P\cap D)-P$. However, this contradicts the assumption of $\pi$ being a support hyperplane, for $P+v$ and $P-v$ must lie on different sides of $\pi$. Therefore, $\pi$ contains $\mathbb{A}_P$, hence the face $\mathbb{A}_P\cap D$ of $P$.
\end{proof}

\begin{definition}[{exposed face (see, e.g. \cite[p. ~162]{Rock})}]
\label{defn:exposedface}
A set which is the intersection of $D$ with one of its support hyperplanes (recalled in \cref{defn:hyperplane}) is called an \emph{exposed face}. 
\end{definition}

\begin{remark}
When $D$ is a polytope (that is, the intersection of a finite number of half-spaces), the notion of face coincides with that of an exposed face. For general convex bodies, this is false: exposed faces are always faces (see, e.g. \cite[p. ~162]{Rock}), but the converse is false in general.  Consider the following example: take $D$ to be the $\epsilon$-neighbourhood of any convex polygon in $\mathbb{R}^2$. Then, $\partial D$ is made up of straight edges joined to circular arcs, and a point at which one joins a straight edge and a circular arc is a face, but not an exposed face. 
\end{remark}

\begin{proposition}
\label{thm:convextoface}
For any convex set $C\subset \partial D$, there exists a unique smallest face of $\partial D$ containing $C$.
\end{proposition}

\begin{proof}
Since $C$ is a convex set, its interior, as a subset of $\partial D$, is a convex $k$-dimensional open ball $\mathrm{int}(C)$. By \cref{lem:uniqueface}, for each point $P\in\mathrm{int}(C)$, the face for  $P$ is  obtained as an intersection of the form $\mathbb{A}_P\cap\partial D$. For a point $Q$ in $\mathrm{int}(C)$, by the convexity of $\mathrm{int}(C)$, there is a segment on $\partial D$ with endpoint $Q$ containing $P$ in its interior.
Therefore by \cref{defn:classicalface}, $Q$ must be contained in the face for $P$. Therefore, there is a unique face $F_C$ which is the face for every interior point $P\in\mathrm{int}(C)$. Since faces are closed, the face $F_C$ must contain all of $C$. In fact, $F_C$ must be the smallest face of $\partial D$ that contains $C$ as a subset because $C$ contains interior points of $F_C$, whereas subfaces of $F_C$ are necessarily on the boundary (see \cref{defn:classicalface}.)
\end{proof}

\begin{remark}
It is possible to prove \cref{thm:convextoface} without invoking \cref{lem:uniqueface}. First use the hyperplane separation theorem to show that there must be some exposed face containing $C$, and hence show that the set of faces containing $C$ is non-empty. Then take the intersection of all faces containing $C$, and show (by definition) that this intersection is a face. The advantage of the method provided is that we see that $C$ must contain interior points of $F_C$.
\end{remark}

\subsection{Face-dimension}

The goal of this subsubsection is to use the unique association of faces to points established in \cref{lem:uniqueface} to introduce a new notion of dimension for points in $\partial D$ which we call face-dimension (\cref{def:face-dimension}).

\begin{definition}[adherence]
\label{defn:adherence}
We say that a  face $F$ of $\partial D$ is \emph{adherent} to a face $F'\supseteq F$ if for any face $F''$ containing $F$ there is a face which contains both $F'$ and $F''$ as subfaces.
See \cref{fig-face} for an example.
We also introduce the following related notions:
\begin{itemize}
\item Each face $F$ is adherent to a unique maximal face (possibly $F$ itself, and the uniqueness is justified below). We refer to this maximal face as the \emph{adherence closure of $F$}.
We say that $F$ is \emph{adherence-closed} if its adherence closure is the face $F$ itself.
Adherence closures are necessarily adherence-closed as shown in \cref{closure closed}.
\item For an adherence-closed face $F$, the union of its interior and the interiors of all subfaces whose adherence closures coincide with $F$ is called the \emph{adherence core} of $F$. 
\item An adherence-closed face $F$ is said to be \emph{adherence-complete} if it is the adherence closure of each of its subfaces. 
\end{itemize}
\end{definition}

\begin{proposition}[adherence closure uniqueness]
Every face $F$ has a unique adherence closure.
\end{proposition}

\begin{proof}
Let $F_1$ and $F_2$ be both adherence closures to $F$. The facts that $F$ is adherent to $F_1$ and  that $F_2$ contains $F$ imply that there is some face in $\partial D$ containing both $F_1$ and $F_2$. 
Let $\hat{F}$ denote the intersection of every face that contains both $F_1$ and  $F_2$. Note that $\hat{F}$ is a face. We show that $\hat{F}$ is the  adherent closure of $F$. Let $F'$ be an arbitrary face containing $F$, then there is a face $F''$ that contains both $F'$ and $F_1$. Also, since $F$ is adherent to $F_2$, there is a face $F'''$ that contains both $F_2$ and $F''$, hence $F_1\cup F_2\cup F'\subset F'''$. By the definition of $\hat F$, this means that $\hat{F}\subset F'''$, and $F'\subset F'''$ as shown above, and hence $F$ is adherent to $\hat F$. The assumption that $F_1$ and $F_2$ are maximal faces to which $F$ is adherent, combined with the fact that $\hat F$ contains both $F_1$ and $F_2$ by definition, implies that $F_1=\hat{F}=F_2$.
\end{proof}

\begin{proposition}
\label{closure closed}
The adherence closure $\widehat{F}$ of a face $F$ is adherence-closed.
\end{proposition}

\begin{proof}
Let $\overline{F}$ denote the adherence closure of $\widehat{F}$. Consider an arbitrary face $F'$ containing $F$, and let $F''$ denote a face containing both $F'$ and $\widehat{F}$. Since $\widehat{F}$ is adherent to $\overline{F}$,  there is a face $F'''$ containing both $F''$ and $\overline{F}$, and hence $F'''$ contains both $F'$ and $\overline{F}$. This shows that $F$ is adherent to $\overline{F}$, and by the maximality of $\widehat{F}$, we see that $\widehat{F}=\overline{F}$, as desired.

\end{proof}

\begin{definition}[face-dimension for points]\label{def:face-dimension}
For a point $P$ on $\partial D$, let $F$ denote the face for $P$. We refer to the dimension of the adherence closure of $F$ as the \emph{face-dimension of $P$}, and denote it by $\fdim(P)$.
\end{definition}

The following is immediately obtained from the definition.
\begin{proposition}[dimension vs. face-dimension]
\label{thm:dimvsfdim}
For any point $P\in\partial D$,
\[
\fdim (P)
\geq 
\dim (P),
\] 
with equality if and only if the face for $P$ is adherence-closed. 
\end{proposition}

\begin{remark}
It is possible for face-dimension to be strictly greater than dimension.
In the right figure of \cref{fig-face}, the point $x$ is a face of dimension $0$, whereas its face-dimension is $2$.
\end{remark}

\begin{figure}[h!]
\begin{center}
\includegraphics[height=4.5cm, valign=t]{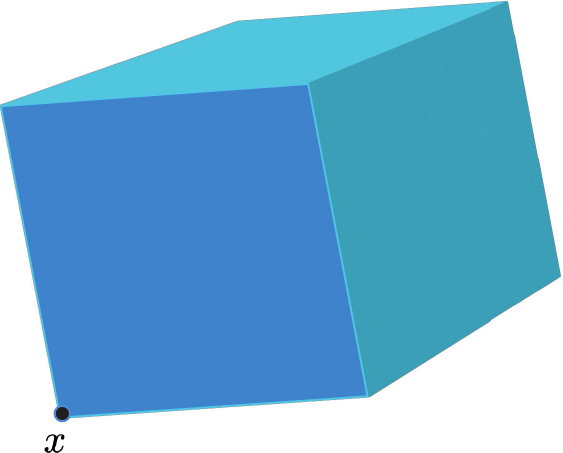}
\includegraphics[height=5cm, valign=t]{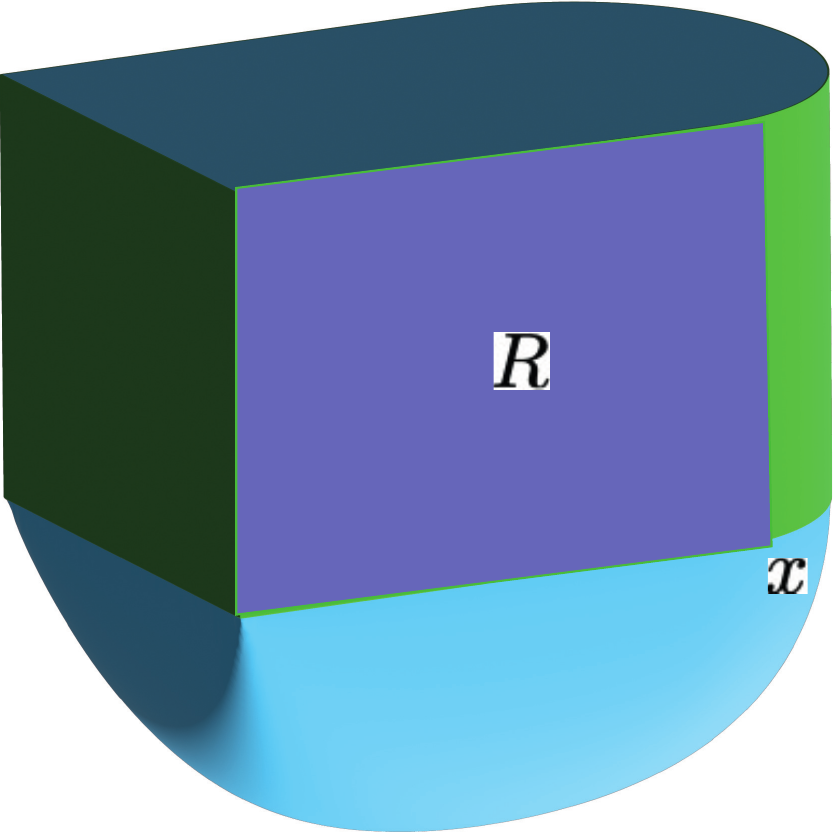}
\caption{Two different examples of two-dimensional convex spheres. 
The adherence closure of $x$ is $x$  itself in the left figure, whereas the adherence closure of $x$ is the rectangle $R$ in the right figure.}
\label{fig-face}
\end{center}
\end{figure}

 \subsection{Adherence-dimension}
 
 \begin{definition}[adherence-dimension for faces]
 \label{defn:adim}
We call a chain of faces 
\[
F_1\subsetneq F_2\subsetneq\ldots\subsetneq F_{k-1}\subsetneq F_k
\] 
an \emph{F-dim ascending chain} if the face-dimensions of the $\{F_i\}_{i=1,\ldots,k}$ are all distinct, \ie they are strictly increasing.
We also introduce the following related notions:
\begin{itemize}
\item
We call an F-dim ascending chain  $\{F_i\}_{i=1, \dots , k}$ \emph{maximal} if there is no other F-dim ascending chain containing $\{F_i\}_{i=1, \dots , k}$  under the imposed condition (such as starting with or ending with a given face).
\item
We define the \emph{adherence height} of $F$ to be the minimum of the lengths of maximal F-ascending chains \emph{ending} with $F$.
\item
We define the \emph{adherence depth} of $F$ to be the minimum of the lengths of maximal F-ascending chains \emph{starting} with $F$.
\item
We define the \emph{adherence-dimension} $\adim (F)$ of $F$ to be the sum of its adherence height and its adherence depth subtracted by $2$.
\end{itemize}
\end{definition}

\begin{remark}
The underlying concept for adherence height, depth and dimension is motivated by the notion of Krull dimension, and there is a certain level of flexibility in how they might be defined --- for example: one can choose to take the \emph{maximum} rather than the minimum length of maximal F-ascending chains starting/ending with $F$. Our choice is because
\begin{itemize}
\item
we need to take the minimum rather than maximum length for adherence height as this is needed for the proof of \cref{adherence dim},
\item
the same choice leads us to also take the minimum rather than the maximum length for defining adherence depth, to ensure that the adherence-dimension of a face $F$ is necessarily greater than or equal to the adherence-dimensions of the subfaces of $F$.
\end{itemize}
\end{remark}


\begin{definition}[Adherence-dimension for points]
We define the \emph{adherence-dimension for a point} $P\in\partial D$ as the adherence-dimension of the face $F$ for $P$. In particular, we denote this by $\adim_{\partial D}(P):=\adim(F)$.
\end{definition}

\subsection{Duality and codimension}
We have hitherto worked with an arbitrary convex body $D\subset\mathbb{V}$, but now consider the setting where $0\in\mathbb{V}$ is an interior point of $D$. 

\begin{definition}[dual convex body]
Define the {\em dual convex body} $D^*\subset \mathbb{V}^*$ in the dual space $\mathbb{V}^*$ by
\begin{align*}
D^*:=
&
\left\{
w^*\in\mathbb{V}^* \mid
\forall v\in D,\ w^*(v)\leq 1
\right\}\\
=
&
\left\{
w^*\in\mathbb{V}^* \mid
\sup_{v\in D}
\ w^*(v)\leq 1
\right\}\\
=
&
\left\{
w^*\in\mathbb{V}^* \mid
\sup_{v\in \partial D}
\ w^*(v)\leq 1
\right\}.
\end{align*}
We refer to $\partial D^*$ as the dual sphere to the sphere $\partial D$, and note that positive homothety ensures that
\begin{align*}
\partial D^*=\left\{
w^*\in\mathbb{V}^*
\mid
\sup_{v\in\partial D} w^*(v)=1
\right\}.
\end{align*}
\end{definition}

The classical notion of support hyperplane (\cref{defn:hyperplane}) establishes a geometric relation between the two dual pictures.
%
In \cref{codim of convex}, we shall define the codimension for a point $P$ as the dimension of the space of support hyperplanes at $P$. To make this precise, we parametrise support hyperplanes via normal vectors:

\begin{definition}[normal and positive normal vectors]\label{def:normal}
We say that a non-zero vector $w^* \in \mathbb{V}^*$ is \emph{normal} to an affine hyperplane $\pi\subset \mathbb{V}$ if  for any two points $P, Q\in\pi$,  
\[
w^*(P-Q)=0. 
\]
When $\pi$ is a support hyperplane of a convex subset $D\subset\mathbb{V}$, we say that a dual vector $w^* \in\mathbb{V}^*$ is a \emph{positive normal vector} to $\pi$ if  $w^*$ is normal to $\pi$ and points to the side of $\pi$ disjoint from $D$, \ie if $R$ is contained in the side disjoint from $D$, then $w^*(R-P)>0$ for any $P \in \pi$.
\end{definition}

\begin{remark}
Elementary linear algebra tells us that every support hyperplane  $\pi$ through $P\in \partial D$ takes the form
\begin{align*}
\pi=\left\{v\in \mathbb{V} \mid \ w^*(v)=c\right\},\quad c=w^*(P),
\end{align*}
for some positive normal vector $w^*\in\mathbb{V}^*$. In particular, any two positive normal vectors to the same hyperplane are positive scalar multiples of each other and may be uniquely normalised so that $c=w^*(P)=1$. 
\end{remark}

\begin{definition}[codimension]
\label{codim of convex}
For any point $P\in\partial D$, we define the following subset of $\partial D^*$:
\begin{align*}
N_P:=
\left\{
w^*\in \mathbb{V}^*\;
\begin{array}{|l}
w^*\text{ is positive normal to some}\\
\text{support  hyperplane to $\partial D$ at $P$}\\
\text{and }w^*(P)=1
\end{array}
\right\}.
\end{align*}
We refer to the dimension of $N_P$, as a subset of the $(m-1)$-sphere $\partial D^*$, as the \emph{codimension} of $P$, and denote it by $\codim_{\partial D}(P)$.
\end{definition}

\begin{remark}
In the above, we define codimension of $P$ as the dimension of the space of support hyperplanes for $P$, only phrased in terms of unit normal vectors to these hyperplanes.
\end{remark}

\begin{corollary}
\label{sum}
For any point $P$ on the boundary $\partial D$ of a convex full-dimensional subset $D\subset \mathbb{V}$, we have 
\[
\dim_{\partial D}(P)+\codim_{\partial D}(P) \leq \dim \mathbb{V}-1.
\]
\end{corollary}

\begin{proof}
We again consider to translate the picture  by $-P$.  
\cref{thm:facesupport} may be interpreted as saying that any positive normal vector $w^*$ contains $\mathbb{A}_P^0:=\mathbb A_P-P$ in its kernel. The rank-nullity theorem tells us that the space of arbitrary (\ie not just positive normal) dual vectors  which take the value $0$ on $\mathbb{A}_P^0$ is itself of dimension $\dim\mathbb{V}-\dim\mathbb{A}_P^0=\dim\mathbb V-\dim\mathbb A_P$, and imposing the normalisation condition for positive normal vectors tells us that $N_P$ has dimension at most $\dim\mathbb{V}-\dim\mathbb{A}_P-1$, and hence
\[
\codim_{\partial D}(P)
=\dim N_P \leq 
\dim \mathbb{V}-\dim\mathbb{A}_P-1
=
\dim \mathbb{V}-\dim_{\partial D}(P)-1.
\]\end{proof}

We conclude this subsection by establishing several properties of the sets $N_P$ in \cref{codim of convex}. We first show that inclusion for the $N_P$ satisfies the following  contravariance property:

\begin{lemma}
\label{lem:preface}
If $F$ is the face for $P\in\partial D$, \ie if $P$ is an interior point of $F$, then for any $Q\in F$, we have
\[
N_{P}\subseteq N_{Q}\subset\partial D^*.
\] 
In particular, if $P$ and $Q$ are interior points of the same face $F\subset \partial D$, then
\[
N_P=N_Q.
\]
\end{lemma}

\begin{proof}
If a point $P\in\partial D$ is contained in the interior of a face $F$, then any support hyperplane to $\partial D$ at $P$ is also a support hyperplane to any other point of the face (e.g. $Q\in F$). 
The claims of the lemma immediately follow from this fact.
\end{proof}

We further see that each $N_P$ is actually an exposed face in the dual sphere for $\partial D$.

\begin{theorem}
\label{thm:nface}
For every $P\in\partial D$, the set $N_P\subset \mathbb{V}^*$ is an exposed face (see \cref{defn:exposedface}) of $\partial D^*$. In particular,  this means that $N_P$ is a (non-empty) convex and compact subset of $D^*$. 
\end{theorem}

\begin{proof}
We first claim that
\begin{align*}
&N'_{P}:=
\left\{
w^*\in \mathbb{V}^* \mid
\sup_{v\in\partial D} w^*(v)=w^*(P)=1
\right\}
\end{align*}
is equal to $N_P$. To see this, we first note that any $w'^*\in N'_P$ defines an affine hyperplane of the form
\begin{align*}
\Pi:=\left\{
v\in\mathbb{V}\mid w'^*(v)=w'^*(P)=1
\right\}.
\end{align*}
The equality $\sup_{v\in\partial D} w'^*(v)=w'^*(P)$ implies that $\partial D$ (and hence $D$) lies on one side of $\Pi$ where the value of $w'^*$ is at most $1$. Therefore, if $w'^* \in N'_P$, then it is a positive normal vector to a support hyperplane of $D$ at $P$, and hence $N'_P\subseteq N_P$.

Conversely, given a positive normal vector $w^*\in N_P$, the fact that $\partial D$ lies on the side of the support hyperplane normal to $w^*$ where $w^*$ takes the value at most $1$ tells us that
\[
w^*(P)\geq\sup_{v\in\partial D}w^*(v)\geq w^*(P),\text{ and hence } \sup_{v\in\partial D}w^*(v)= w^*(P).
\]
We also know, from the definition of $N_P$, that $w^*(P)=1$, therefore $N_P\subseteq N'_P$, and hence we have $N_P=N'_P$.

We next observe that $N_P$ is the intersection of the following hyperplane (in $\mathbb{V}^*$)
\begin{align*}
\left\{
w^*\in \mathbb{V}^* \mid
w^*(P)=1
\right\}
\end{align*}
and the sphere 
\begin{align*}
\partial D^*
=
\left\{
w^*\in \mathbb{V}^* \mid
\sup_{v\in\partial D} w^*(v)=1
\right\},
\end{align*}
a  hence $N_P$ is a exposed face of $\partial D^*$ in $\mathbb{V}^*$. In particular, this is an intersection of two closed convex sets, where one is bounded, thus their intersection is convex and compact. 
\end{proof}

\begin{remark}
A consequence of the proof of \cref{thm:nface} is that each $N_P$ is expressed as
\[
N_P
=
\left\{
w^*\in \mathbb{V}^* \mid
\sup_{v\in\partial D} w^*(v)=w^*(P)=1
\right\}.
\]
This has the interpretation that $N_P$ consists of all ``unit'' dual vectors $w^*$ that attain their maximal value (on $\partial D$) at $P$.
\end{remark}

\subsection{Linear invariants}
\label{sec:linvariants}
The goal of this final subsection is to show that the concepts we have defined in this section are invariant under invertible linear maps. To begin with, we point out that the image of a convex set $D\subset \mathbb{V}$ under a linear map $f\colon\mathbb{V}\to\mathbb{W}$ is necessarily convex. Moreover, if $f$ is invertible, then the boundary of $D$ is sent to the boundary of $f(D)$:
\[
f(\partial D)=\partial f(D)\subset\mathbb{W}.
\]

\begin{theorem}[linear invariants]
\label{thm:linvariance}
Consider an invertible linear map $f\colon \mathbb{V}\to\mathbb{W}$. For every point $P\in \partial D$, 
\begin{enumerate}
\item
$f$ maps the face for $P$ on $\partial D$ to the face for $f(P)$ on $\partial f(D)$, and hence takes the convex stratification on $\partial D$ to the convex stratification on $\partial f(D)$;

\item
the dimension of $P$ is a linear invariant:
\[
\dim_{\partial D}(P)=\dim_{\partial f(D)}(f(P));
\]

\item
the relation of adherence is preserved under $f$;

\item
the face-dimension of $P$ is a linear invariant:
\[
\fdim_{\partial D}(P)=\fdim_{\partial f(D)}(f(P));
\]

\item
 adherence height and  adherence depth are linear invariants;

\item
the adherence-dimension of $P$ is a linear invariant:
\begin{align*}
\adim_{\partial D}(P)=\adim_{\partial f(D)}(f(P));
\end{align*}

\item
the codimension of $P$ is a linear invariant:
\begin{align*}
\codim_{\partial D}(P)=\codim_{\partial f(D)}(f(P)).
\end{align*}

\end{enumerate}

\end{theorem}

\begin{proof}
The classical definition of face (\cref{defn:classicalface}) combined with the fact that $f$ and its inverse take intervals to intervals imply $(1)$. More precisely, the linear image of a face $F$ (which is convex) is necessarily convex, and so we need only verify that for every $\xi\in f(F)$ and $\eta,\zeta\in f(D)$ such that $\xi$ lies on the open interval between $\eta$ and $\zeta$, both $\eta$ and $\zeta$ lie in $f(F)$. 
In this situation,  $f^{-1}(\xi)\in F$ lies on the open interval between $f^{-1}(\eta)$ and $f^{-1}(\zeta)$. Since $F$ is a face, we see that $f^{-1}(\eta),f^{-1}(\zeta)\in D$, and hence $\eta=f(f^{-1}(\eta)),\zeta=f(f^{-1}(\zeta))\in f(D)$. We see therefore that faces are preserved under invertible linear transformations, and since $f$ is a homeomorphism, any interior point $P\in F$ is sent to an interior point $f(P)\in f(F)$. This suffices to establish $(1)$. The claim $(2)$ then follows immediately from $(1)$.

We next observe that since faces are preserved, if a face $F$ is a subface of $F'$, then $f(F)$ is a subface of $f(F')$. This means that $f$ preserves the relation of adherence, \ie the claim $(3)$ is true, and hence the adherence closure of $F$ is mapped to the adherence closure of $f(F)$. This yields claim $(4)$. Furthermore, since adherence is preserved, adherence height, adherence depth and adherence-dimension must all be preserved: the claims $(5)$ and $(6)$ both hold.

Finally, we shall prove that $f$ preserves codimension. By linearity and invertibility, the map $f$ takes each support hyperplane $\pi$ at $P\in\mathbb{V}$ to a support hyperplane $f(\pi)$ at $f(P)\in\mathbb{W}$, hence $f$ induces a map from the face $N_P\subset\mathbb{V}^*$ to the face $N_{f(P)}\subset\mathbb{W}^*$. This is in fact a homeomorphism, with its inverse induced by $f^{-1}$ taking support hyperplanes at $f(P)$ to support hyperplanes at $\pi$. Therefore, the dimensions of $N_P$ and $N_{f(P)}$ agree,  hence the codimensions of $P$ and $f(P)$ must be the same, \ie  the claim $(7)$ is true.
\end{proof}

\section{The convex geometry of Thurston metric spheres}
\label{s:unitball}

We now take the various linear invariants developed in \cref{sec:linvariants} and consider them specifically in the context of convex bodies associated with the Thurston metric. In particular, the convex body $D$ in question will, unless otherwise stated, be the closed ball bounded by
\[
\mathbf{S}_x^*:=\iota_x(\pml(S))\subset T_x^*\teich(S).
\]
The theory of convex bodies (\cref{sec:convexbodies}) which we have hitherto seen tells us general properties such as the fact that the dimension of an arbitrary point $P=\iota_x([\lambda])$ is less than or equal to its face-dimension (\cref{thm:dimvsfdim}):
\[
\dim_{\mathbf{S}_x^*}(\iota_x([\lambda]))
\leq
\fdim_{\mathbf{S}_x^*}(\iota_x([\lambda])).
\]
We now consider the geometry of $\mathbf{S}_x^*$ in this specific context, and in particular establish relationships between the faces on $\mathbf{S}_x^*$ and the topological structure of the support of the measured laminations encoded by each face. We aim to access topological structure from the convex geometry of $\mathbf{S}_x^*$, and one application is in our proofs of the \emph{topological rigidity} (\cref{sec:topological}) of Thurston norm isometries.

As another example, consider the following reformulation of \cref{lem:preface} in this specific context. First recall that
\[
N_x([\lambda])
:=
N_{\iota_x([\lambda])}
=
\left\{ 
v\in \mathbf{S}_x
\mid 
\iota_x([\lambda])(v)
=\left\|v\right\|_{\mathrm{Th}}
 \right\}.
\]

\begin{lemma}
\label{faces}
 If $\iota_x([\lambda])$ is contained in the interior of a face $F$, \ie if $F$ is the face for $\iota_x([\lambda])$, then for any projective lamination $\mu$ with $\iota_x([\mu]) \in F$, we have 
 \[
 N_x([\lambda]) \subseteq N_x([\mu])\subset T_x\teich(S). 
 \]
 In particular, if  $\iota_x([\lambda])$ and  $\iota_x([\mu])$ are in the interior of the same face, then 
 \[
N_x([\lambda])=N_x([\mu])\subset T_x\teich(S).
 \]
\end{lemma}

\subsection{Minimal supporting surface and support closure}
A measured lamination or a projective lamination is decomposed into components. For each component, we can define as follows its minimal supporting surface.

We shall use  the notion of \emph{incompressible} subsurface $\Sigma$ of $S$. This is a surface with boundary embedded in $S$ such that all its boundary components are essential simple closed curves.
We also regard $S$ itself as being incompressible.

\begin{definition}[minimal supporting surface]
Let $\lambda$ be a minimal measured geodesic lamination on $S$ or the projective class of such a lamination. A \emph{minimal supporting surface} for $\lambda$ is an incompressible subsurface $\Sigma$ of $S$ containing $\lambda$ such that for any other incompressible subsurface $\Sigma'$ containing $\lambda$, we have $\Sigma \subset \Sigma'$ after moving $\Sigma$ by an isotopy. See \cref{fig:supporting} for a depiction.

\end{definition}

From the above definition, it follows that the minimal supporting surface of $\lambda$ is unique up to isotopy and depends only on $|\lambda|$. We now define the support closure of a measured/projective measured lamination. This will play an important role later.

\begin{definition}[support closure]
\label{defn:supportclosure}
For a measured lamination or a projective measured lamination $\lambda$, we consider the minimal supporting surface of every component of $\lambda$ that is not a simple closed curve.
We add to $|\lambda|$ all boundary components of the minimal supporting surfaces and straighten the curves so as to obtain a geodesic lamination. We call this geodesic lamination the \emph{support closure} of $\lambda$ and denote it by $\widehat{|\lambda|}$ (see \cref{fig:supporting}).
\end{definition}

\begin{figure}[h!]
\begin{center}
\includegraphics[scale=1.8]{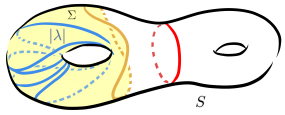}
\caption{The subsurface $\Sigma$ is the minimal supporting surface for $|\lambda|$, and $\widehat{|\lambda|}$ is obtained by adding the geodesic representative for the boundary curve of $\Sigma$ to $|\lambda|$. }
\label{fig:supporting}
\end{center}
\end{figure}

\subsection{Faces for measures of equal support}

We first establish some new notation through the following lemma.

\begin{lemma}
\label{lem:supportset}
Given a measured lamination $\lambda$ on $S$, let $|\lambda|$ denote its support geodesic lamination. The image, under $\iota_x$ for $x \in \teich(S)$, of the set of projective measured laminations whose supports are contained in $|\lambda|$ is a convex set, and there is a unique minimal face on $\mathbf S_x^*$ containing this convex set. We denote this face by $F_{|\lambda|}$.
\end{lemma}

\begin{proof}
Instead of working with projective measured laminations $[\mu]\in\pml(S)$ with support in $|\lambda|$, we use a measured lamination representative $\mu\in\ml(S)$. Since $d\log \len(\mu)=\frac{d\len(\mu)}{\len_x(\mu)}$  is homogeneous with respect to $\mu\in\ml(S)$, we assume without loss of generality that $\mu$ lies on the unit sphere of $\ml(S)$, \ie in the subset of $\ml(S)$ defined by $\len_x(\mu)=1$. We refer to such measured laminations as \emph{unit measured laminations}.

For two unit measured laminations $\lambda_1$ and $\lambda_2$ whose supports are contained in  $|\lambda|$, the linear combination 
\[
\lambda_3=t\lambda_1+(1-t)\lambda_2,\quad\text{ for }t\in [0,1],
\]
 is another unit measured lamination with support contained in $|\lambda|$, which satisfies  
\begin{align*}
\iota_x([\lambda_3])&=d\len(t\lambda_1+(1-t)\lambda_2)\\
&=td\len(\lambda_1)+(1-t)d\len(\lambda_2)=t\iota_x([\lambda_1])+(1-t)\iota_x([\lambda_2]).
\end{align*}
Thus the image, under $\iota_x$, of the set of unit measured laminations with support contained in  $|\lambda|$, is a convex set. By \cref{thm:convextoface}, there is a unique minimal face of $\mathbf S_x^*$ containing this convex set. 
\end{proof}

The following lemmas make use of the notation which we introduced in \cref{lem:supportset}, and illustrate some first applications of the notion of support closure (see \cref{defn:supportclosure}). They are key lemmas which will feature in the proofs of results to come in subsequent subsections.
Recall that a geodesic lamination is said to be recurrent if any half-leaf starting at any point $x$ returns arbitrarily near to $x$.

\begin{lemma}
\label{lem:intersection}
A recurrent geodesic lamination $\xi$ has non-empty transverse intersection with $|\lambda|$ if and only if it has non-empty transverse intersection with $\widehat{|\lambda|}$.
\end{lemma}

\begin{proof}
One direction is obvious, and we need only show that if $\xi$ has nonempty transverse intersection with $\widehat{|\lambda|}$, then it also has nonempty transverse intersection with $|\lambda|$. Assume the contrary. Then, first recall that $\widehat{|\lambda|}$ differs from $|\lambda|$ by a collection of simple closed geodesics which are homotopic to the boundaries of the minimal supporting surfaces of the components of  $\lambda$. The fact that $\xi$ has nonempty transverse intersection with $\widehat{|\lambda|}$ and not $|\lambda|$, combined with the fact that $\xi$ is recurrent, means that there must be the minimal component $\Sigma$ of a component $\lambda_0$ of $\lambda$ and  some geodesic arc $\alpha$ in $\xi\cap\Sigma$ which starts and ends on the boundary of $\Sigma$. Moreover, such a geodesic arc $\alpha$ cannot be homotopic to a boundary arc of $\Sigma$ as that would form a hyperbolic geodesic bigon. This in turn contradicts the fact that $\Sigma$ is the minimal supporting surface of $\lambda_0$, as a component of $\Sigma\setminus \alpha$ retracts onto a surface which supports $\lambda_0$, but $\Sigma$ cannot be homotoped into $\Sigma\setminus\alpha$.
\end{proof}

\begin{lemma}[non-transversality]
\label{case 1}
Suppose that  $\iota_x([\lambda])$ is contained in the face for $\iota_x([\mu])$.
Then $|\mu|$ and $\widehat{|\lambda|}$ cannot have non-empty transverse intersection.
\end{lemma}

\begin{proof}
Assume that $|\mu|$ and $\widehat{|\lambda|}$ have non-empty transverse intersection. 
Then by \cref{lem:intersection}, so do $|\mu|$ and $|\lambda|$.
We consider a complete geodesic lamination $\nu$ containing $|\mu|$.
The stretch map along $\nu$ on $(S,x)\in\teich(S)$ defines a stretch vector $v_\nu\in N_x([\mu])\subset T_x\teich(S)$. By \cref{faces}, the vector $v_\nu$ is contained in $N_x([\mu'])$ for any point $\iota_x([\mu'])$ in the face for $\iota_x([\mu])$, and hence in particular we have $v_\nu\in N_x([\lambda])$. However, this is impossible because the infinitesimal Thurston stretch map along $\nu$ maximally stretches precisely the laminations which lie in $\nu$ (this fact is asserted in \cite[Theorem~5.1]{ThS}, see \cref{thm:infcomparison} for a sketch-of-proof and \cref{maxstretch} for some discussion), whereas $\nu$ has non-empty transverse intersection with $|\lambda|$ (see \cref{lem:intersection}) and hence $v_\nu$ cannot maximally stretch $[\lambda]$.
\end{proof}

\begin{theorem}[faces for measures of equal support]
\label{thm:equisupport}
Let $[\lambda]\in\pml(S)$ be a projective measured lamination with support $|\lambda|$. Then, for $x \in \teich(S)$, the face $F_{|\lambda|}\subset T_x^*\teich(S)$ is an exposed face consisting of precisely the set of $\iota_x$ images of projective measured laminations with support contained in $|\lambda|$. In particular, the interior of $F_{|\lambda|}$ consists only of $\iota_x$-images of  projective laminations whose supports are equal to $|\lambda|$.
\end{theorem}

\begin{proof} For $x$ in $\teich(S)$, let $(S,x)$ be a hyperbolic structure on $S$.
Since $(S,x)\setminus |\lambda|$ is a bordered hyperbolic surface, where we regard $x \in \teich(S)$ as a hyperbolic structure on $S$, we may endow $(S,x)\setminus|\lambda|$ with a finite complete geodesic lamination with no closed geodesic leaves. In particular, adding this lamination to $|\lambda|$ yields a complete geodesic lamination $\xi$ on $(S,x)$ which can only support measured laminations on sublaminations of $|\lambda|\subset\xi$. The stretch vector $v_\xi\in\mathbf{S}_x\subset T_x\teich(S)$ defines a support hyperplane of the form:
\[
\pi_\xi
=
\left\{
w^*\in T_x^*\teich(S)
\mid
\sup_{v\in \mathbf{S}_x} w^*(v)
=
w^*(v_\xi)=1
\right\}.
\]
We first observe that the $\iota_x$ image of every projective measured lamination supported on $|\lambda|$ lies in this support hyperplane because $v_\xi$ stretches along $|\lambda|$. Also, the minimality of the face $F_{|\lambda|}$ ensures that it is a subset of the exposed face $\pi_\xi\cap \mathbf{S}_x^*$.

Conversely, \cref{thm:infcomparison} and the fact that $|\lambda|$ is the largest possible (transverse) measure supporting sublamination of $\xi$ implies that $\pi_\xi\cap \mathbf{S}_x^*$ can only contain $\iota_x$-images of projective measured laminations with support contained in $|\lambda|$. This also shows that $F_{|\lambda|}=\pi_\xi\cap \mathbf{S}_x^*$, and that $F_{|\lambda|}$ is an exposed face as claimed since $\pi_\xi$ is a support hyperplane.

To finish the proof, we show that the interior points of $F_{|\lambda|}$ consist only of $\iota_x$-images of projective measured laminations whose supports are precisely equal to $|\lambda|$. Consider a projective lamination $[\mu]$ such that $\iota_x([\mu])$ is an interior point for $F_{|\lambda|}=\pi_\xi\cap\mathbf{S}_x^*$. We have shown that $|\mu|\subseteq|\lambda|$, and we now suppose that $|\mu|$  is contained in  $|\lambda|$ as a proper subset. Since $\iota_x([\mu])$ is an interior point and $\iota_x([\lambda])$ lies in $F_{|\lambda|}$, \cref{faces} tells us that 
\begin{align}
N_x([\mu])= N_x([\lambda]).
\label{eq:case3}
\end{align}
Now take $\xi$ to be a complete lamination containing $|\mu|$ that intersects $|\lambda|\setminus |\mu|$ transversely. The stretch vector $v_{\xi}$, with respect to $\xi$, is contained in $N_x([\mu])$, but not in $N_x([\lambda])$ as $\xi$ intersects $|\lambda|$ transversely (\cref{thm:infcomparison}). This contradicts \cref{eq:case3}, and hence $|\mu|$ cannot be a proper subset of $|\lambda|$ and must instead be equal to $|\lambda|$.
\end{proof}

\begin{remark}
\cref{thm:equisupport} claims that any interior point of $F_{|\lambda|}$ corresponds to a projective lamination whose support is $|\lambda|$, but it does not say that every lamination whose support is precisely $|\lambda|$ is mapped into the interior of $F_{|\lambda|}$.
\end{remark}

\begin{corollary}
\label{inclusion}
For two arbitrary projective laminations $[\lambda],[\mu]\in\pml(S)$, 
\[
|\mu|\subsetneq|\lambda|\text{ if and only if }
F_{|\mu|}
\subsetneq
F_{|\lambda|}.
\]
\end{corollary}

\begin{corollary}
\label{dimension}
The dimension of $\iota_x([\lambda])$ is at most equal to the dimension of the cone of transverse measures on $|\lambda|$ subtracted by $1$. 
\end{corollary}

\begin{corollary}
\label{thm:bound}
Given an arbitrary point $\iota_x([\lambda])\in\mathbf{S}_x^*$, then the dimension $\dim_{\mathbf{S}_x^*}(\iota_x([\lambda]))$ satisfies
\[
\dim_{\mathbf{S}_x^*}(\iota_x([\lambda]))\leq
3g-4+n.
\]
\end{corollary}
\begin{proof}
By \cref{dimension}, the face of $\iota_x([\lambda])$ has the same dimension as the space of projective transverse measures on $|\lambda|$.
The dimension of the latter space is bounded by $3g-4+n$ as is shown in  \cite[Corollaire, p. 133]{Papa}.
\end{proof}

\begin{remark}\label{rmk:upperbound}
The upper bound on dimension bounds the face-dimension of every $\iota_x([\lambda])$ by $3g-4+n$ since the face dimension is the dimension of the adherence closure of the face for $\iota_x([\lambda])$, which is again a face, and hence its dimension is that of its interior point.
\end{remark}

\begin{corollary}[embedded curve complex]\label{cor:embeddedcx}
The convex stratification of $\mathbf{S}_x^*$ contains the curve complex for $S$ as a subcomplex. We denote this subcomplex as $\mathbf{C}_x^*$.
\end{corollary}

\begin{proof}
Following \cref{thm:equisupport}, we have a complete topological description of the face of $\iota_x([\Gamma])$ for an arbitrary projective multicurve $[\Gamma]$: when $|\Gamma|$ consists of $k$ simple closed geodesics, the face of $\iota_x([\Gamma])$ is a $k-1$ simplex whose boundary consists of lower-dimensional simplices corresponding to the faces of projective weighted multicurves supported on a proper subset of closed geodesics in $|\Gamma|$. 
In fact, we can easily show that for any projective positive weight given on $|\Gamma|$, its image under $\iota_x$ lies in the inteior of $F_{|\Gamma|}$.
This is precisely a geometric realisation of the curve complex, and we denote it by $\mathbf{C}_x^*$.
\end{proof}

\begin{corollary}
\label{dim of mc}
For any projectively weighted multicurve with $k$ components  $[\Gamma]\in\pml(S)$, the dimension of $\iota_x([\Gamma])$ is $k-1$.
\end{corollary}

\begin{corollary}[contravariant labelling]
\label{contravariant}
For two arbitrary projective multicurves $[\Gamma_1],[\Gamma_2]\in\pml(S)$, denote their respective supports by $|\Gamma_1|$ and $|\Gamma_2|$. Then,
\begin{align}
|\Gamma_1|\subseteq |\Gamma_2|
\textit{ if and only if }
N_x([\Gamma_2])\subseteq N_x([\Gamma_1]),
\end{align}
with equality on the left if and only if we have equality on the right.
\end{corollary}

\begin{proof}
Since $[\Gamma_1]$ and $[\Gamma_2]$ are projective multicurves, the points $\iota_x([\Gamma_1])$ and $\iota_x([\Gamma_2])$ respectivly lie in the interior of $F_{|\Gamma_1|}$ and $F_{|\Gamma_2|}$. Therefore, $|\Gamma_1|\subseteq |\Gamma_2|$ if and only if $\iota_x([\Gamma_1])$ is contained in the face $F$ of $\iota_x([\Gamma_2])$. \cref{faces} then yields the desired result.
\end{proof}

\subsection{Closure correspondence}

\begin{theorem}[closure correspondence]
\label{measure} Let $x$ be a fixed hyperbolic structure on $S$, and consider a projective lamination $[\lambda]$ with  support $|\lambda|$. Let $\widehat{F}$ be the adherence closure of the face $F$ for $\iota_x([\lambda])$.
Then the following hold:
\begin{enumerate}[(a)]
\item The face $\widehat{F}$ is equal to the (exposed) face $F_{\widehat{|\lambda|}}$ consisting of the images of all projective laminations whose supports are contained in $\widehat{|\lambda|}$.

\item The interior points of $\widehat{F}$ are $\iota_x$-images of projective laminations with support equal to $\widehat{|\lambda|}$. 

\item The $\iota_x$-image of the set of all projective laminations whose support closures are $\widehat{|\lambda|}$ coincides with the adherence core of $\widehat{F}$.
\end{enumerate}
\end{theorem}

\begin{proof} We shall show that the adherence closure $\widehat{F}$ of $F$ coincides with $F_{\widehat{|\lambda|}}$. The fact that $F_{\widehat{|\lambda|}}$ is an exposed face, \ie that it is the intersection of a support hyperplane with $\mathbf{S}_x^*$ (see \cref{thm:equisupport}), which contains $\iota_x([\lambda])$, combined with \cref{thm:facesupport} ensures that $F$ is a subface of $F_{\widehat{|\lambda|}}$. We first show that  $F$ is adherent to $F':=F_{\widehat{|\lambda|}}$. Suppose that another face $F''$ contains $F$ as a subface, and take $\iota_x([\mu])$ in the interior of $F''$. There are three mutually exclusive possibilities covering all options for the relationship between $\widehat{|\lambda|}$ and $[\mu]$:
\begin{enumerate}[(1)]
\item
$\widehat{|\lambda|}$ has nonempty transverse intersection with $|\mu|$;
 \item
$\widehat{|\lambda|}$ has empty transverse intersection with $|\mu|$, but  $\widehat{|\lambda|}\setminus |\mu| \neq \emptyset$;
 \item
$\widehat{|\lambda|}$ is properly contained in $|\mu|$.
 \end{enumerate}
 
 \medskip
\noindent
\textbf{Case (1):}
This is a situation where $F$, which is the face for $|\lambda|$, is contained in $F''$ which is the face for $|\mu|$.
Then, \cref{case 1} tells us that $|\mu|$ intersects $|\widehat{\lambda}|$ transversely, which renders Case (1) impossible.

\noindent
\textbf{Case~(2):} Next suppose that $\widehat{|\lambda|}\setminus |\mu| \neq \emptyset$. We choose a complete lamination $\xi$ containing $|\mu|$ which intersects $\widehat{|\lambda|}$ transversely, and consider the stretch vector $v_\xi$ along $\xi$. 
Then $v_\xi$ is contained in $N_x([\mu])$, but not in $N_x([\lambda])$.
On the other hand, since $F''$ which is the face for $\iota_x([\mu])$ contains $\iota_x([\lambda])$, by \cref{faces}, we have $N_x([\mu]) \subset N_x([\lambda])$.
This is a contradiction.
%

\noindent
\textbf{Case~(3):} The only remaining case is  when $\widehat{|\lambda|}$ is a proper subset of $|\mu|$. Take a unit measured lamination $\lambda'$ such that $\iota_x([\lambda'])$ is contained in the interior of $F_{\widehat{|\lambda|}}$, and note that \cref{thm:equisupport} tells us that $|\lambda'|=\widehat{|\lambda|}$. Since both $\lambda'$ and $\mu$ are measured laminations, their supports are both unions of their minimal components.
Therefore, we have a decomposition of $\mu$ (assumed to be a unit measured lamination) into disjoint sublaminations $\mu_1 \sqcup \mu_2$ such that $|\mu_1|=|\lambda'|=\widehat{|\lambda|}$. Since $|\mu|$ contains $|\lambda'|=\widehat{|\lambda|}$, we can define a path of unit measured lamination $t \lambda' +(1-t)\mu$, for $t\in[0,1]$, which is mapped to an affine segment on $\mathbf{S}_x^*$ connecting $\iota_x([\mu])$ with $\iota_x([\lambda'])$. 
Since $|t \lambda' +(1-t)\mu|\subset |\mu_1|\cup |\mu|=|\mu|$, the path is contained in $F_{|\mu|}$.
The face $F_{|\mu|}$ contains both $F'=F_{\widehat{|\lambda|}}$ and $F''$, and hence $F$ is adherent to $F'=F_{\widehat{|\lambda|}}$.

We next show that $F_{\widehat{|\lambda|}}$ is maximal among the faces to which $F$ is adherent. 
Since the adherence closure of $F$ is maximal in this sense, what we need to show is that there is no face properly containing $F_{\widehat{|\lambda|}}$ to which $F$ is adherent.
If $F$ is adherent to $F'\supseteq F_{\widehat{|\lambda|}}$, then for any interior point $\iota_x([\mu])\in F'$, our arguments for Cases~(1) and (2) show that $|\mu|$ necessarily contains $\widehat{|\lambda|}$ as a subset. 
Suppose that $\widehat{|\lambda|}\subsetneq|\mu|$. Then there is a measured lamination $\nu$ whose support contains  $\widehat{|\lambda|}$ as a sublamination and which intersects $\mu$ transversely. Let $F''=F_{|\nu|}$. Then, \cref{thm:equisupport} tells us that $F''\supset F_{\widehat{|\lambda|}}$. Since $F_{\widehat{|\lambda|}}$ contains $F$ as was explained at the beginning of the proof, $F$ is a subface of $F''$. On the other hand, since $|\mu|$ and $|\nu|$ have non-empty transverse intersection, by \cref{thm:equisupport}, they cannot lie in the same face in $\mathbf{S}_x^*$ --- any point in the interior of such a face would be the $\iota_x$ image of a projective measured lamination with both $|\mu|$ and $|\nu|$ in its support, which is impossible. This implies that there is no face (in $\mathbf{S}_x^*$) containing both $F'$ and $F''=F_{|\nu|}$, and hence contradicts the assumption that $F$ is adherent to $F'$. Therefore the possibility that $\widehat{|\lambda|}\subsetneq|\mu|$  is excluded, which implies that $\widehat{|\lambda|}=|\mu|$ and thus $F'=F_{\widehat{|\lambda|}}$, as desired. 
Thus we have shown that $F_{\widehat{|\lambda|}}$ is the adherence closure of the face for $\iota_x([\lambda])$, and we have completed the proof of (a). \cref{thm:equisupport} then gives (b).

We finally turn to proving (c). Consider a projective lamination $[\lambda']\in\pml(S)$ whose support closure is $\widehat{|\lambda|}$. Then part (a) asserts that the adherence closure of the face for $\iota_x([\lambda'])$ is precisely $F_{\widehat{|\lambda|}}$. Therefore, $\iota_x([\lambda'])$ is contained in the adherence core of $\widehat{F}=F_{\widehat{|\lambda|}}$. Conversely, suppose that $\iota_x([\mu])$ is in the adherence core of $\widehat{F}$, \ie it is contained in the interior of a face $\tilde{F}$ whose adherence closure is $F_{\widehat{|\lambda|}}$. Then part (a) shows that the adherence closure of $\tilde{F}$ is $F_{\widehat{|\mu|}}$, hence that $F_{\widehat{|\mu|}}=F_{\widehat{|\lambda|}}$. Part (b) then asserts that $\widehat{|\mu|}=\widehat{|\lambda|}$, therefore the support closure of $[\mu]$ is indeed equal to $\widehat{|\lambda|}$, as desired.
\end{proof}

\cref{measure} implies the following corollaries.

\begin{corollary}
\label{adinclusion}
For any two projective laminations $[\lambda],[\mu]\in\pml(S)$, 
\[
\text{if }|\widehat{\mu}|\subsetneq|\widehat{\lambda}|\text{, then }
F_{\widehat{|\mu|}}
\subsetneq
F_{\widehat{|\lambda|}}.
\]
\end{corollary}

\begin{corollary}
The face-dimension of $\iota_x([\lambda])$ is equal to the dimension of the cone of transverse measures on $\widehat{|\lambda|}$ subtracted by $1$. 
\end{corollary}

\begin{corollary}
\label{fdim of mc}
If $[\Gamma]\in\pml(S)$ is a projectively weighted multicurve with $k$ components, then the face-dimension of $\iota_x([\Gamma])$ is $k-1$.
\end{corollary}

\subsection{Distinguishing laminations}

One strength of \cref{measure} is that the property of being adherence-closed is linearly invariant. Results in this subsection show us how we might distinguish certain types of projective measured laminations based on linearly invariant properties (such as adherence-closedness) of the face for the $\iota_x$-images of the aforementioned laminations.

We say that a (projective class of a) measured geodesic lamination is \emph{maximal} if its support is not contained strictly in the support of any other (projective class of a) measured geodesic lamination.

\begin{theorem}
\label{characterising mc}
Let $[\lambda]$ be a projective class of a measured lamination. Then the following two properties are equivalent:
\begin{enumerate}
\item
 $[\lambda]$ is either the projective class of a weighted multicurve or it is maximal and uniquely ergodic;
 \item 
 every subface of the face for $\iota_x([\lambda])$ is adherence-closed.
 \end{enumerate}
\end{theorem}

\begin{proof}
Let $F$ be the face for $\iota_x([\lambda])$.
If $[\lambda]$ is a projectively weighted multicurve, then by \cref{measure},  every subface $\tilde{F}$ of $F$ is the face of a projective multicurve supported on a subset of $|\lambda|$. In particular, since multicurves are their own support closures, $\tilde{F}$ is adherence-closed. On the other hand, if $[\lambda]$ is maximal and uniquely ergodic, then $F$ consists of just one point $\iota_x([\lambda])$ and $F=\left\{\iota_x([\lambda])\right\}$ is the only subface of itself. Moreover, since $[\lambda]$ is maximal, then, by \cref{measure}, there is no face properly containing $F$, and hence $F$ is again adherence-closed.

Suppose next that $[\lambda]$ is neither a projectively weighted multicurve nor a maximal and uniquely ergodic lamination. We show that there is a subface of $F$ whose adherence closure is not itself. First note that if the adherence closure of $F$ is not $F$, then we are done, and thus we assume that $F$ is adherence-closed. In particular, since $\iota_x([\lambda])$ is an interior point of $F$, we know by \cref{measure} that $F=F_{\widehat{|\lambda|}}$. We cover the remainder of the proof with two cases:
\begin{enumerate}[(1)]
\item
$[\lambda]$ is either non-maximal or is disconnected,
\item
$[\lambda]$ is both maximal and connected.
\end{enumerate}
\medskip

\noindent
Case~(1): If $[\lambda]$ is either non-maximal or is disconnected, then $\lambda$ has a component $\lambda_0$ which is not a weighted simple closed curve and whose supporting surface is not all of $S$. This means that the support closure $\widehat{|\lambda_0|}$ properly contains $|\lambda_0|$, hence \cref{measure} tells us that the face $\tilde{F}$ for  $\iota_x([\lambda_0])$ is not adherence-closed. Indeed, \cref{measure} tells us that the adherence closure of $\tilde{F}$ is $F=F_{\widehat{|\lambda|}}$, hence it gives a subface of the latter which is not adherence-closed. This covers the case when $F$ is either non-maximal or is disconnected.\medskip

\noindent
Case~(2): If $[\lambda]$ is both maximal and connected, then, by the assumption that $[\lambda]$ is not uniquely ergodic, the space of transverse measures on $|\lambda|$ constitutes a convex set whose extreme points correspond to the ergodic measures.
 (The existence of such extreme points follows from a classical theorem on convex sets which can be found in  \cite{KM}.
 The proof of the fact that the ergodic measures constitute the extreme points can be found in \cite{Kat, Lev}.)
 This implies that the image of the entire set of projective transverse measures on $|\lambda|$ cannot be the interior of $F_{|\lambda|}$,  therefore  there is a projective lamination $[\mu]$ supported on $|\lambda|$ such that the face $\tilde{F}$ for $\iota_x([\mu])$ is not adherence-closed (by \cref{measure}). Since $F=F_{\widehat{|\lambda|}}$ is the adherence closure of $\tilde{F}$, the latter is a subface of $F$ which fails to be adherence-closed, as desired. This completes the proof.
\end{proof}

\begin{proposition}
\label{thm:maximality}
If a projective measured lamination $[\lambda]$ is maximal, then the adherence closure of the face for $\iota_x([\lambda])$ is maximal, \ie it is not a subface of any other face.
\end{proposition}

\begin{proof}
Since $[\lambda]$ is maximal, its support closure is $\widehat{|\lambda|}=|\lambda|$. 
Consider a face $F$ containing $F_{|\lambda|}$, and let $\iota_x([\lambda'])$ be an interior point of $F$. Then \cref{measure} tells us that 
\[
F_{|\lambda|}
\subseteq
F
\subseteq F_{\widehat{|\lambda'|}}.
\]
By \cref{inclusion} and the maximality of $[\lambda]$, we see therefore that $F_{|\lambda|}=F_{\widehat{|\lambda'|}}$, and hence $F_{|\lambda|}=F$. Since every face containing $F_{|\lambda|}$ is equal to $F_{|\lambda|}$ itself, the face $F_{\widehat{|\lambda|}}=F_{|\lambda|}$, which is the adherence closure of the face for $\iota_x([\lambda])$, must be maximal.
\end{proof}

\begin{remark}
The converse to \cref{thm:maximality} is false: simply take $[\lambda]$ to be a non-maximal lamination such that $\widehat{|\lambda|}$ is maximal.
\end{remark}

\begin{theorem}
\label{connectedness}
A projective lamination $[\lambda]$ is maximal and connected if and only if the adherence closure of the face for $\iota_x([\lambda])$ is maximal (that is, it is not a subface of any other face), and adherence-complete (\cref{defn:adherence}).
\end{theorem}

\begin{proof}
We first show the \lq\lq only if'' part.
We have already seen from \cref{thm:maximality} that if $[\lambda]$ is maximal, then the adherence closure $F'=F_{\widehat{|\lambda|}}$ of the face for $\iota_x([\lambda])$ is maximal. We need to prove adherence-completeness from connectedness. Take a point $\iota_x([\lambda'])$ lying in the interior of $F'$. By \cref{measure}, given an interior point $\iota_x([\mu])$ of an arbitrary face in $F'$, the lamination $[\mu]$ has support $|\mu|$ contained in $|\lambda'|=\widehat{|\lambda|}=|\lambda|$ (the last equality follows from the maximality of $|\lambda|$). Since $\lambda$ is connected, this implies that $|\mu|=|\lambda|$ (see e.g. \cite[Cor~1.7.3]{HP}), and by \cref{measure}, the adherence closure of the face for $\iota_x([\mu])$ is also $F'$. Therefore, $F'$ is adherence-complete.

Now, to prove the \lq\lq if'' part, suppose that $[\lambda]$ is either non-maximal or disconnected.
Then there is a component of $\lambda$ whose minimal supporting surface is not the entire $S$.
 Let $c$ denote a boundary curve of such a minimal supporting surface. Let $F'$ be the adherence closure of the face for $\iota_x([\lambda])$.
 Then, by \cref{measure}, $\iota_x([c])$ is contained in $F'$. We immediately exclude the case when $|c|=|\lambda|$; for, then $F'$ consists of one point, and cannot be a maximal face (recall that $S\neq S_{1,1}$, $S_{0,4}$). Thus, $|c|\neq |\lambda|$, and the adherence closure of the face for $\iota_x([c])$ is itself (\cref{measure}), hence $F'$ is not adherence-complete.
\end{proof}

\subsection{Tangential adherence}

In part~(c) of \cref{measure}, we described the image of the set of projective measured laminations whose support closures are $\widehat{|\lambda|}$ as a subset on $F_{\widehat{|\lambda|}}\subset\mathbf{S}_x^*$. We now present an analogous result for the set of projective measures supported precisely on $|\lambda|$.

\begin{definition}[tangential adherence]
\label{tangentially adherent}
A subface $F$ of a face $F'\subset \mathbf{S}_x^*$ is said to be \emph{tangentially adherent} to $F'$ if the following condition is satisfied:

Given $[\lambda],[\mu]\in\pml(S)$ such that $F$ is the face for $\iota_x([\lambda])$  and $F'$ is the face for $\iota_x([\mu])$, every stretch vector $v$, along a complete geodesic lamination which  is contained in $N_x([\lambda])$ is also contained in $N_x([\mu])$.
\end{definition}

\begin{proposition}
\label{set of measures}
Let $[\lambda]$ be a projective measured lamination on $S$, and consider the face $F_{|\lambda|}$ consisting of all $\iota_x$ images of projective measured laminations supported in $|\lambda|$ (see \cref{lem:intersection} and \cref{thm:equisupport}). The union of the interior of $F_{|\lambda|}$ and the interiors of all subfaces that are tangentially adherent to $F$ coincides with the $\iota_x$-image of the set of projective transverse measures supported precisely on $|\lambda|$. 
\end{proposition}
\begin{proof}
Assume without loss of generality that $[\lambda]$ is contained in the interior of $F_{|\lambda|}$. Let $[\lambda']$ be an arbitrary projective lamination with support $|\lambda'|=|\lambda|$. By \cref{thm:equisupport}, $\iota_x([\lambda'])$ is contained in $F_{|\lambda|}$, hence the face $F$ for $\iota_x([\lambda'])$ is a subface of the exposed face $F_{|\lambda|}$ (\cref{thm:facesupport}). If $F=F_{|\lambda|}$, then $\iota_x([\lambda'])$ (trivially) lies in the interior of a face tangentially adherent to $F_{|\lambda|}$. Now suppose otherwise, that the face for $\iota_x([\lambda'])$ is a proper subface $F$ of $F_{|\lambda|}$. Given an arbitrary stretch vector $v_\nu\in N_x([\lambda'])$, the vector $v_\nu$  maximally stretches  a complete geodesic lamination $\nu$ which contains $|\lambda'|=|\lambda|$. Thus, $v_\nu$ also maximally stretches $\lambda$ and so $v_\nu$ is also contained in $N_x([\lambda])$. This shows that the face $F$ is tangentially adherent to $F_{|\lambda|}$, hence an arbitrary point $\iota_x([\lambda'])$ for $\lambda'$ with support equal to $|\lambda|$ always lies in the interior of some tangentially adherent subface of $F_{|\lambda|}$.

Conversely, suppose that the face $F$ for some point $\iota_x([\mu])$ is tangentially adherent to $F_{|\lambda|}$. By \cref{thm:equisupport}, $|\mu|$ is a subset of $|\lambda|$. If $|\mu|$ is properly contained in $|\lambda|$, then we can find a complete geodesic lamination $\nu$ containing $|\mu|$ but intersecting $|\lambda|$ transversely. Then the stretch vector $v_\nu$ along $\nu$ is contained in $N_x([\mu])$ but not in $N_x([\lambda])$, which contradicts the assumption that $F$ is tangentially adherent to $F_{|\lambda|}$. We conclude that every projective lamination $[\mu]$ whose image under $\iota_x$ is contained in the interior of $F$ has support precisely equal to $|\lambda|$.
\end{proof}

A neat statement is the following:

\begin{corollary}\label{naturaltransformation}
For any projective measured lamination $[\lambda]$,
\[
\widehat{F_{|\lambda|}}
=
F_{\widehat{|\lambda|}}.
\]
\end{corollary}

\begin{proof}
By \cref{set of measures}, there exists $[\lambda']\in\pml(S)$ on the interior of $F_{|\lambda|}$ such that $|\lambda'|=|\lambda|$. Then, $F_{|\lambda|}$ is the face for $\iota_x([\lambda'])$ and by the closure correspondence theorem (\cref{measure}), we obtain:
\[
\widehat{F_{|\lambda|}}
=
\widehat{F_{|\lambda'|}}
=
F_{\widehat{|\lambda'|}}
=
F_{\widehat{|\lambda|}}.
\]
\end{proof}

As previously highlighted, one major advantage of \cref{measure} is the linear invariance of the property of adherence (and adherence closures, by extension). In contrast, we are unable to show directly that tangential adherence is linearly invariant. We note that it does follow as a consequence of the infinitesimal rigidity theorem \cref{main}. In any case, we establish the following relationship between tangential adherence and adherence:

\begin{lemma}
A subface tangentially adherent to $F'$ is adherent to $F'$.
\end{lemma}

\begin{proof}
Suppose that $F$ is tangentially adherent to $F'$, and that $F$ is the face for $\iota_x([\lambda])$ whereas $F'$ is the face for $\iota_x([\mu])$ as in the statement of \cref{tangentially adherent}. Let $F''$ be a face containing $F$ as a proper subface, and let $\iota_x([\nu])$ be an interior point for $F''$. The exposed face $F_{|\nu|}$ contains $F''$ (by \cref{thm:equisupport}), and hence also $\iota_x([\lambda])$.
It follows that $|\nu|$ contains $|\lambda|$ (\cref{measure}). Suppose that $|\nu|$ does not contain in $|\mu|$. Then we can choose a complete geodesic lamination $\xi$ containing $|\nu|$ but intersecting $|\mu|$ transversely. The stretch vector $v_\xi$ lies in $N_x([\lambda])$ because $\xi$ contains $|\mu|$. However, $v_\xi$ is not an element of $N_x([\mu])$ since $\xi$ intersects $|\mu|$ transversely, and this contradicts the condition of \cref{tangentially adherent}.

We may therefore conclude that $|\nu|$ contains $|\mu|$, and hence the face $F_{|\nu|}$ contains both $F''$  and $F'$ (by \cref{inclusion}). This affirms the property of adherence.\end{proof}

\subsection{Codimension of multicurves}

We have already determined the dimension and face-dimension of $\iota_x$-images of multicurves (\cref{dim of mc} and \cref{fdim of mc}). We now determine their codimension.

\begin{theorem}[codimension of multicurves]
\label{codim of mc}
For any projectively weighted multicurve $[\Gamma] \in \pml(S)$ having $k$ components and for any $x\in\mathcal{T}(S)$, the codimension of $\iota_x([\Gamma])$ is equal to $6g-6+2n-k$. Note in particular that this codimension is independent of $x$.
\end{theorem}

We first establish upper bound for $\codim(\iota_x([\Gamma]))$:

\begin{lemma}[codimension upper bound]
\label{inequality}
For any projectively weighted multicurve $[\Gamma]$ having $k$ components, we have
\[
\codim_{\mathbf{S}_x^*}(\iota_x([\Gamma]))\leq 6g-6+2n-k.
\]

\end{lemma}
\begin{proof}
We know from \cref{dim of mc} that $\dim_{\mathbf{S}_x^*}(\iota_x([\Gamma]))$ is equal to $k-1$. The lemma then follows immediately from \cref{sum}.
\end{proof}

\begin{remark}
In the specific case when $[\Gamma]$ is a closed geodesic, \cref{codim of mc} is a corollary of Thurston's work \cite[Theorem 10.1]{ThS}. 
\end{remark}

We next show that the upper bound on codimension is tight by constructing sufficiently many linearly independent vectors in the length increasing cone. We begin with the special case when the support of $\lambda$ is a pants decomposition of $S$.

\begin{theorem}[codimension for pants decompositions]
\label{pants case}
Every projectively weighted pants decomposition $[\Gamma]\in\pml(S)$ of $S=S_{g,n}$ has codimension $3g-3+n$.
\end{theorem}

\begin{proof}
After \cref{inequality}, we need only show that there are $3g-3+n$ linearly independent vectors contained in $N_x([\Gamma])$. In particular, since $N_x([\Gamma])$ is a convex set, this is equivalent to showing that the affine subspace generated by $N_x([\Gamma])$ is $3g-3+n$-dimensional.\medskip

Let $\gamma_1, \dots , \gamma_{3g-3+n}$ denote the simple closed geodesics constituting the support of $[\Gamma]$. For each of the $\gamma_i$, let $\Lambda_{+,i}$ and $\Lambda_{-,i}$ denote  complete geodesic laminations on $S$ such that
\begin{itemize}
\item 
$\Lambda_{\pm,i}$ contain $\gamma_1,\ldots, \gamma_{3g-3+n}$,
\item
$\Lambda_{\pm,i}$ contain geodesic leaves which bound the same $(1,1)$-cusped annulus $A_i\subset S$, which in turn contains $\gamma_i$,
\item
$\Lambda_{-,i}$ and $\Lambda_{+,i}$ have the same geodesic leaves, except in the interior of $A_i$, where the two bi-infinite geodesic leaves on $A_i$ in $\Lambda_{-,i}$ and $\Lambda_{+,i}$ which spiral toward $\gamma_i$ do so in opposite directions (see the upper left and lower left pictures in \cref{fig:2}).
\end{itemize}

It is a simple exercise to see that it is always possible to construct such $\Lambda_{\pm,i}$.

We denote the stretch vectors for these laminations by $v_{\pm,i}\in N_x([\Gamma])$. By \cref{thm:linindep}, we see that $v_{+,i}-v_{-,i}$ is precisely the unit Fenchel--Nielsen twist vector for $\gamma_i$. Hence the vectors
\[
\left\{v_{+,i}-v_{-,i}\right\}_{i=1,\ldots,3g-3+n}
\]
correspond to (non-zero multiples of) Fenchel--Nielsen twists for distinct simple closed geodesics, and hence they are linearly independent. 
Moreover, $v_{+,1}$ is linearly independent from the $v_{+,i}-v_{-,i}$ since stretching along $\Lambda_{+,1}$ increases the lengths of $\gamma_1,\ldots, \gamma_{3g-3+n}$. Hence, the following collection of $3g-2+n$ vectors 
\[
v_{+,1}
,\ (v_{+,1}-v_{-,1})
,\ (v_{+,2}-v_{-,2}),\
\ldots,
\ (v_{+,3g-3+n}-v_{-,3g-3+n})
\]
are linearly independent and generate a ($3g-2+n$)-dimensional vector space $\mathbb{A}$ in $T_x\teich(S)$. This is necessarily a subspace of the vector space generated by $N_x([\Gamma])$, and hence $\mathbb{A}\cap N_x([\Gamma])$ has the same dimension as $\mathbb{A}\cap\left\{ v\in T_x\teich(S)\mid \iota_x([\Gamma])(v)=1\right\}$,
which in turn has dimension $3g-2+n-1=3g-3+n$. Therefore, $N_x([\Gamma])$ has dimensionat least $3g-3+n$, and this gives us the desired lower bound.
\end{proof}

We now adapt this argument to deal with general multicurves:

\begin{proof}[Proof of \cref{codim of mc}]

Given an arbitrary projective class of a weighted multicurve $[\Gamma]\in\pml(S)$, let $\gamma_1,\ldots,\gamma_k$ denote the simple closed geodesics which make the support  $|\Gamma|$. We first complete $\{\gamma_i\}_{i=1,\ldots, k}$ to a pants decomposition $\{\gamma_i\}_{i=1,\ldots, 3g-3+n}$, whose union is represented by a weighted multicurve $\Gamma'$. 
In the proof of \cref{pants case}, we showed that there is a vector space of dimension $3g-2+n$ generated by stretch vectors contained in $N_x([\Gamma'])$.
Therefore, we can choose linearly independent stretch vectors $\{v_j\}_{j=1,\ldots,3g-2+n} \in N_x([\Gamma'])$.
Since $|\Gamma|$ is contained in $|\Gamma'|$ these vectors also maximally stretch $\Gamma$, and hence they are contained in $N_x([\Gamma])$.

We next replace $\gamma_{3g-3+n}$ by a different simple closed geodesic $\delta_1$ such that $\{\gamma_1,\ldots,\gamma_{3g-4+n},\delta_1\}$ is a new pants decomposition, and we let $w_1$ denote the stretch vector for some complete geodesic lamination containing $\{\gamma_1,\ldots,\gamma_{3g-4+n},\delta_1\}$. 

We show that $v_1,\ldots ,v_{3g-2+n}$, and $w_1$ are linearly independent as follows. 
If not, then there exist coefficients $a_1,\ldots ,a_{3g-2+n},b_1\in\mathbb{R}$ such that
\[
a_1v_1+a_2v_2+\ldots +a_{3g-2+n}v_{3g-2+n}+b_1w_1=\vec{0}\in T_x\teich(S).
\]
Since all of $v_1, \dots , v_{3g-2+n}$ and $w_1$ are stretch vectors for $\gamma_1,\ldots,\gamma_{3g-4+n}$, considering $d\ell_{\gamma_j}(a_1v_1+a_2v_2+\ldots +a_{3g-2+n}v_{3g-2+n}+b_1w_1)$ for $j=1, \dots, 3g-4+n$, we have $a_1+ a_2 +\dots + a_{3g-2+n}+b_1=0$.
This implies  that $d\ell_{\delta_1}(b_1w_1)=d\ell_{\delta_1}(-a_1v_1-a_2v_2-\dots -a_{3g-2+n}v_{3g-2+n})$.
However, since the vector $w_1$ is the only one among this collection which maximally stretches $\delta_1$ the above equality holds only when $b_1=0$.
 Then the linear independence of $v_1, \dots v_{3g-2+n}$  implies that $a_1=\dots =a_{3g-2+n}=0$. Hence $v_1,\ldots, v_{3g-3+n},w_1$ are linearly independent as claimed.

We continue this procedure by replacing $\gamma_{3g-4+n}$ with a simple closed curve $\delta_2$ such that $\{\gamma_1,\ldots,\gamma_{3g-5+n},\delta_1,\delta_2\}$ is a new pants decomposition and similarly produce a new tangent vector $w_2$. The same coefficient argument tells us that $\{v_1,\ldots,v_{3g-2+n},w_1,w_2\}$ is a collection of linearly independent vectors. To clarify, 
if \[
a_1v_1+a_2v_2+\ldots+a_{3g-2+n}v_{3g-2+n}+b_1w_1+b_2w_2=\vec{0},
\]
then we use the fact that all the vectors maximally stretch $\gamma_j$ for $j=1, \dots, 3g-5+n$ implies that $a_1+a_2+ \dots a_{3g-2+n}+b_1+b_2=0$ by considering its image under $d\ell_{\gamma_j}$.
Among these vectors, only $w_2$ maximally stretches $\delta_2$, which shows that $b_2=0$, by considering the value under $d{\ell_{\delta_2}}$.
The linear independence of $v_1, \dots , v_{3g-2+n}, w_1$ then implies that $a_1=\dots=a_{3g-2+n}=b_1=0$.

We iterate the above argument until we construct linearly independent tangent vectors $v_1,\ldots,v_{3g-2+n},w_1,…,w_{3g-3+n-k}$. Since these vectors maximally stretch $\gamma_1,\ldots,\gamma_k$, they must lie in $N_x([\Gamma])$, hence the codimension of $\iota_x([\Gamma])$ is at least $3g-2+n+3g-3+n-k-1=6g-6+2n-k$. This gives the desired lower bound.
\end{proof}

\begin{corollary}[embedded dual curve complex]\label{prop:embeddedccx}
For every $x\in\teich(S)$, the convex stratification of $\mathbf{S}_x$ contains a subset $\mathbf{C}_x$  which consists of a subcollection of strata which is dual to the curve complex in the sense that
\begin{itemize}
\item
the support $|\Gamma|$ of any projective multicurve $[\Gamma]\in\pml(S)$ is assigned to a face $N_x([\Gamma])\subsetneq \mathbf{S}_x$;
\item
the subset-relation for simplices in the curve complex  corresponds, by the above assignment, to the superset-relation for faces in $\mathbf{C}_x$;
\item
the dimension of each simplex in the curve complex is equal to the codimension of its corresponding face in $\mathbf{C}_x$ as a subset of $\mathbf{S}_x$.
\end{itemize}
\end{corollary}

\begin{proof}
\cref{contravariant} asserts that the assignment of the face $N_x([\Gamma])$ to the multicurve $[\Gamma]$ depends only on the support $|\Gamma|$ and  hence is well defined. Moreover, \cref{contravariant} ensures that this assignment is injective and \cref{faces} ensures that the superset-relation in $\mathbf{C}_x$ is inverted. Let us denote the collection of faces of the form $N_x([\Gamma])$, for multicurves $[\Gamma]$, by $\mathbf{C}_x$.

It only remains to demonstrate the codimension condition on $\mathbf{C}_x$. \cref{codim of mc} asserts that for any projective multicurve $[\Gamma]$ consisting of $k$ components, the dimension of $N_x([\Gamma])$ is $6g-6+2n-k$. Therefore, its codimension as a subset of $\mathbf{S}_x$ is $6g-7+2n-(6g-6+2n-k)=k-1$ --- this is precisely the dimension of the simplex in the curve complex corresponding to $|\Gamma|$. 
\end{proof}

\section{Infinitesimal rigidity of the Thurston metric}
\label{s:rigidity}
The goal of this section is to prove the infinitesimal rigidity theorem, \ie \cref{main}. Before proceeding, we clarify an important technical point. We actually have two ways in which mapping classes induce maps between tangent (and cotangent) spaces on the Teichm\"uller space as we now explain.
A diffeomorphism $h\colon S\to S$ induces a mapping class $[h]$ which acts diffeomorphically on $\teich(S)$ by the pushing-forward of metrics, as well as the natural homeomorphism $h\colon \pml(S) \to \pml(S)$ given by taking $[\lambda]$ to $[h(\lambda)]$. Then, for any $x\in\teich(S)$,
\begin{enumerate}
\item
we obtain a map (using the fact that $[h]$ is an analytic isometry and hence differentiable) dual to the derivative map $d[h]$
\[
d^*[h]\colon 
T^*_{[h](x)}\teich(S)\to T^*_x \teich(S),
\]

\item
as well as a map 
\[
\iota_x\circ h^{-1}\circ \iota_{[h](x)}^{-1}\colon \mathbf{S}^*_{[h](x)}\to\mathbf{S}^*_x,
\]
which we extend by non-negative homothety to a map 
\[
{\iota_x\circ h^{-1}\circ \iota_{[h](x)}^{-1}\colon T^*_{[h](x)}\teich(S)\to T^*_x \teich(S)}.
\]
\end{enumerate}

\begin{lemma}
The two maps $\iota_x\circ h^{-1}\circ \iota_{[h](x)}^{-1}$ and $d^*[h]$ agree.
\end{lemma}

\begin{proof}
Any tangent vector $v\in T_x\teich(S)$ at $x$ can be represented as the derivative $\frac{d}{dt} c(0)$ at $0$ of a differentiable curve $c(t) \colon (-\epsilon, \epsilon) \to \teich(S)$. Then, 
\[
d[h](v)=\tfrac{d}{dt}([h]\circ c)(0).
\]
For any measured lamination $\mu$ and any point $y\in\teich(S)$, we have $\ell_{[h](y)}(h(\mu))=\ell_y (\mu)$, hence 
\[
\ell_{[h]\circ c(t)}(h(\mu))=\ell_{c(t)}(\mu).
\]
This in turn implies that $d\log \ell(h(\mu))(\frac{d}{dt}([h]\circ c)(0))=d\log \ell (\mu)(\frac{d}{dt}(c(0))$, hence 
\begin{equation}
\label{covariant}
(\iota_{[h](x)}h([\mu]))(d[h](v))=(\iota_x([\mu]))(v).
\end{equation}
Note that for an arbitrary covector $w^*\in \mathbf{S}^*_{[h](x)}\subset T^*_{[h](x)}\teich(S)$, the values of $d^*[h](w^*)(v)=w^*(d [h])(v)$ for the vectors $v\in T_x\teich(S)$ characterise $d^*[h](w^*)$. Setting $w^*$ to be $\iota_{[h](x)}([h(\mu)])$ for $[\mu] \in \pml(S)$, we see that $d^*[h](\iota_{[h](x)}([h(\mu)])=\iota_x (\mu)$ by \cref{covariant}.
Therefore, we see that $d^*[h] \equiv \iota_x h^{-1} \iota_{[h](x)}^{-1}$ on $\mathbf{S}^*_{[h](x)}$
 and hence on $T^*_{[h](x)}\teich(S)$.
\end{proof}

\begin{corollary}
\label{mcg action}
A diffeomorphism $h\colon S\to S$ representing a mapping class $[h]\colon\teich(S)\to\teich(S)$ preserves linear invariants such as the dimension, the face-dimension, adherence-dimension and codimension of points. In other words,  for any $[\lambda] \in \pml(S)$, we have 
\begin{align*}
\dim_{\mathbf{S}^*_x} \left(\iota_x([\lambda])\right)&
=\dim_{\mathbf{S}^*_{[h](x)}} \left(\iota_{[h](x)}(h([\lambda]))\right),\\
\fdim_{\mathbf{S}^*_x} \left(\iota_x([\lambda])\right)&
=\fdim_{\mathbf{S}^*_{[h](x)}} \left(\iota_{[h](x)}(h([\lambda]))\right),\\
\adim_{\mathbf{S}^*_x} \left(\iota_x([\lambda])\right)&
=\adim_{\mathbf{S}^*_{[h](x)}} \left(\iota_{[h](x)}(h([\lambda]))\right),\\
\text{ and }
\codim_{\mathbf{S}^*_x} \left(\iota_x([\lambda])\right)&
=\codim_{\mathbf{S}^*_{[h](x)}} \left(\iota_{[h](x)}(h([\lambda]))\right).
\end{align*}
\end{corollary}

\begin{proof}
This follows from \cref{thm:linvariance} applied to the observation that $d^*[h]$ is an invertible linear map. 
\end{proof}

\subsection{Topological rigidity}
\label{sec:topological}

Recall from \S\ref{sec:strategy} that every Thurston-norm-preserving isometry $f: T_x \teich(S) \to T_y \teich(S)$ defines a self-homeomorphism of projective lamination space
\[
f_{y,x}:=\iota_x^{-1}
\circ f^*
\circ \iota_y
\colon
\pml(S)\to\pml(S),
\]
where $f^*\colon T_y^*\teich(S)\to T_x^*\teich(S)$ is the dual isometry. We now show that any such $f_{y,x}$ comes from a mapping class:

\begin{theorem}[Topological rigidity, \cref{first part}]
The map $f_{y,x}:\pml(S)\to\pml(S)$ associated with  a Thurston-norm isometry $f: T_x \teich(S) \to T_y \teich(S)$ is necessarily induced by a diffeomorphism $h: S \to S$.
\end{theorem}

As previously mentioned in \S\ref{sec:strategy}, we shall give two different proofs for \cref{first part}. The first proof utilises the notion of adherence-closedness previously introduced in \cref{defn:adherence} and developed in \cref{measure}; and the second proof uses adherence-dimension to quantify the properties of \emph{`chains of faces'}.

\subsubsection{Method~1: adherence-closedness of subfaces}
\label{sec:method1}

The first proof of topological rigidity relies on \cref{characterising mc}.
Suppose that $f:T_x\teich(S)\to T_y\teich(S)$ is a Thurston norm isometry, then its dual Thurston co-norm isometry $f^*:T_y^*\teich(S)\to T_x^*\teich(S)$ takes each face of $\mathbf{S}_y^*:=\iota_y(\pml(S))$ to that of $\mathbf{S}_x^*:=\iota_y(\pml(S))$. By \cref{thm:linvariance}, $f^*$ preserves the relation of $\dim$, adherence, and adherence-closedness.

Let $c$ be a weighted multicurve, and  $F$  the face for $\iota_y([c])\in \mathbf{S}_y^*$. From \cref{measure} and \cref{characterising mc}, we see that $\iota_y([c])$ necessarily lies on the interior of $F_{\widehat{|c|}}=F_{|c|}$. Therefore, $F_{|c|}$ is the face for $\iota_y([c])$ hence $F=F_{|c|}$, which is to say that $F$ consists of the image of the  set of projective measures supported by $|c|$. By \cref{characterising mc},  its image $f(F_{|c|})$ is a face for either a projective weighted multicurve or a maximal and uniquely ergodic projective lamination. In the latter case, the face $f(F_{|c|})$ is $0$-dimensional and maximal. The maximality condition is preserved under (invertible) linear transformations, hence $F_{|c|}$ is also maximal. However, $[c]$ would then need to be a projectively weighted pants decomposition, and would have strictly positive dimension, as follows from \cref{dim of mc}, thereby contradicting the $0$-dimensionality. Therefore $f(F_{|c|})$ must be a face for some projectively weighted multicurve $\iota_x([d])\in \mathbf{S}_x^*$, and hence coincides with $F_{|d|}$.

\cref{inclusion} says that if $|c_1|\subsetneq |c_2|$ for some two projectively weighted multicurves $[c_1]$ and $[c_2]$, then we have $F_{|c_1|} \subsetneq F_{|c_2|}$. Coupling this with the fact that inclusion between faces is preserved under $f^*$, we see that $f^*$ induces a simplicial isomorphism on the curve complex for $S$. By the work of Ivanov, Luo and Korkmaz \cite{Iv,L,K}, the map on the curve complex given by $f$ is induced by an extended mapping class $[h]$. Since the set of multicurves is dense in $\pml(S)$, we see that $f_{y,x}$ coincides with the homeomorphism induced by this mapping class $[h]$.

\subsubsection{Method~2: Characterising multicurves via adherence-dimension}
\label{sec:method3}

We now give the second proof of \cref{first part}, which uses adherence-dimension (\cref{defn:adim}).

\begin{theorem}
\label{adherence dim}
Every face $F$ on $\mathbf{S}_x^*$ has adherence-dimension $\leq3g-3+n$, where equality holds if and only if $F$ is the face for the $\iota_x$-image of a projective weighted multicurve $[c]\in\pml(S)$.
\end{theorem}

\begin{proof}
By \cref{rmk:upperbound}, the longest possible sequence of faces with distinct face-dimensions has corresponding face-dimensions valued from $0$ to $3g-3+n$, and hence adherence-dimension is bounded above by $3g-3+n$. In particular, it is easy to verify that equality holds for any face corresponding to  a projectively weighted multicurve $[c]$ by considering a projective measured lamination supported on a pants decomposition containing $|c|$.
This shows the \lq\lq if '' part of the theorem.

Now, we turn to the \lq\lq only if'' part.
Let $F$ be a face for $\iota_x([\lambda])$ such that $\adim_{\mathbf{S}_x^*}(\iota_x([\lambda]))=3g-3+n$. Then there must be two ascending sequence of faces
\[
F_1\subsetneq \ldots \subsetneq F_k=F\text{ and }F=F_k\subsetneq\ldots\subsetneq F_{3g-3+n}\subsetneq F_{3g-2+n},
\]
which respectively represent the ascending sequences used to define the adherence height and the adherence depth of $F$. In particular, these $\{F_i\}_{i=0,\ldots, 3g-2+n}$ all have distinct face-dimensions, and since dimension itself can only take value between $0$ and $3g-3+n$ (\cref{thm:bound}), we see by \cref{thm:dimvsfdim} that $\dim(F_i)=\fdim(F_i)=i-1$ for every $i=1,\ldots,3g-2+n$. This in turn means that each $F_i$ is adherence-closed, and by \cref{measure}, each face takes the form $F_i=F_{|\lambda_i|=\widehat{|\lambda_i|}}$ and
\[
\widehat{|\lambda_1|} \subsetneq 
\widehat{|\lambda_2|} \subsetneq 
\dots 
\subsetneq \widehat{|\lambda_{3g-4+n}|}
\subsetneq \widehat{|\lambda_{3g-3+n}|}.
\] 

This is possible only if every component of $\lambda_{3g-3+n}$ is either a simple closed curve or a measured lamination whose minimal supporting surface is either a torus with one hole or a sphere with four holes. Moreover, to get the length of $3g-3+n$, whenever a non-multicurve component $\nu$ appears in $\widehat{|\lambda_j|}$ for the first time, all the boundary components of the minimal supporting surface for $\nu$ must be already contained in $\widehat{|\lambda_{j-1}|}$.

Now suppose that $|\lambda_k|=|\lambda|$ contains a non-multicurve component $\nu$. Then the maximal F-dim ascending chain whose first term is the face $F_{|\nu|}$ is necessarily shorter than $F_1 \subsetneq \dots \subsetneq F_k$ since its face-dimension is equal to $\dim F_{\widehat{|\nu|}}\geq 2$. This contradicts our assumption that $F_1 \subsetneq \dots \subsetneq F_k$ attains the adherence height of $F_k$. Therefore $[\lambda]$ must be a projectively weighted multicurve.\end{proof}

\cref{thm:linvariance} says that $f^*\colon T_y^*\teich(S)\to T_x^*\teich(S)$ preserves the relation of adherence, and hence adherence-dimension. \cref{adherence dim} shows that if $F\subsetneq\mathbf{S}_y^*$ is a face for a projectively weighted multicurve, then so is $f^*(F)\subsetneq\mathbf{S}_y^*$. The rest of the proof is the same as in Method 1.

\subsection{Geometric rigidity}
\label{sec:geometric}

\cref{first part}  illustrates a form of ``mapping class realisability'' for Thurston-norm isometries  on $\pml(S)$. We first consider the consequences of such a property on length-increasing cones.

\begin{proposition}
\label{thm:conemcg}
Consider the following setup:
\begin{itemize}
\item
$x,y\in\teich(S)$ and $[\lambda]\in\pml(S)$,

\item
a Thurston-norm isometry
\[
f\colon T_x\teich(S)\to T_y\teich(S),
\]
\item
and its corresponding extended mapping class $[h]\in\mathrm{Mod}^*(S)$ (see \cref{first part}).
\end{itemize}
Then, we have: 
\begin{align*}
f(N_x([\lambda]))=N_y(h([\lambda])).
\end{align*}

\end{proposition}

\begin{proof}
\cref{first part} precisely tells us that $f^*(\iota_y([\lambda]))=\iota_x(h([\lambda]))$, and hence
\begin{align*}
&f(N_x([\lambda]))\\
=&f\left(\left\{
v\in T_x\teich(S)\ 
\begin{array}{|l}
 \ \|v\|_{\mathrm{Th}}=1\text{ and $v$ is positive normal} \\
 \text{to a support hyperplane at }\iota_x([\lambda])
\end{array}
\right\}\right)\\
=&\left\{
f(v)\in T_y\teich(S)\ 
\begin{array}{|l}
 \ \|f(v)\|_{\mathrm{Th}}=1\text{ and $f(v)$ is positive normal} \\
 \text{to a support hyperplane at }\iota_y(h([\lambda]))
\end{array}
\right\}\\
=&N_y(h([\lambda])),
\end{align*}
where the first equality holds by definition (\cref{codim of convex}), and the second equality comes from the fact that $f$ is a linear isometry and hence preserves norms and takes support hyperplanes to support hyperplanes.
\end{proof}

\begin{theorem}
\label{uniquelimit}
Consider a sequence of projectively weighted pants decompositions $\{[\Gamma_i]\}_{i\in\mathbb{N}}$ whose sequence of supports $\{|\Gamma_i|\}_{i\in\mathbb{N}}$ converges to a maximal chain-recurrent geodesic lamination $\Lambda$ in the Hausdorff topology. Then, the sequence of faces $\left\{N_x([\Gamma_i])\right\}_{i\in\mathbb{N}}$ converges to the stretch vector $v_\Lambda$. 
\end{theorem}

\begin{proof}
We show that the diameters of the faces $N_x([\Gamma_i])$, with respect to the Thurston norm, tend to $0$ as $i\to\infty$. \cref{pants case} asserts that each of these faces is $(3g-3+n)$-dimensional. In particular, with respect to the Fenchel--Nielsen coordinates for $|\Gamma_i|$, the length component of the vectors $v\in N_x([\Gamma_i])$ are all identical by the definition of $N_x([\Gamma_i])$, and the vectors there differ from each other only in the $(3g-3+n)$-dimensional subspace of twists. It therefore suffices to show that the diameter of the set of possible twist values $\{v(\tau_{\gamma})\mid v\in N_x([\Gamma_i])\}$ tends to $0$ as $i \to \infty$ for each $\gamma$ in $|\Gamma_i|$; for this would ensure that $N_x([\Gamma_i])$ is trapped in a $(3g-3+n)$-dimensional product of intervals with the width of each product interval shrinking to $0$ as $i\to\infty$, and thereby giving us the claim that
\[
\lim_{i\to\infty}\mathrm{diam}(N_x([\Gamma_i]))=0.
\]

We shall show this by employing \cref{thm:shrink}. To this end, we first show that any sequence of pairs of pants $\{P_i\}_{i\in\mathbb{N}}$, where $P_i$ is a component of $S\setminus|\Gamma_i|$, is asymptotically slender (\cref{defn:slender}). Suppose not.
Then by passing to a subsequence, we may assume that  none of the $P_i$ are $\epsilon$-slender, and hence the length of the shorter orthogeodesic in $\{P_i\}_{i\in\mathbb{N}}$ remains bounded from below by some $(2+\epsilon)\epsilon>0$. 

Now cover each $P_i$ with two disjoint (except along their boundaries) ideal triangles $\triangle^1_i,\triangle^2_i$. We invoke the fact  that the collection of embedded ideal triangles on $S$ is compact \cite[Prop. ~3.8]{HS}, and by changing indices we produce $\{\triangle^1_i\}_{i\in\mathbb{N}}$ which converges to an embedded ideal triangle $\triangle^1$. We can assume without loss of generality that a lift $\tilde{\triangle}^1$ of $\triangle^1$ has ideal vertices placed at $\{0,1,\infty\}$. Then there is a sequence of lifts $\{\tilde{\triangle}^1_i\}_{i\in\mathbb{N}}$ for $\{\triangle^1_i\}_{i\in\mathbb{N}}$ with vertices close to and limiting to $\{0,1,\infty\}$. Taking a further subsequence and possibly applying a permutation of $\{0,1,\infty\}$ by a M\"obius transformation, we can find a sequence of lifts $\{\tilde{\triangle}^2_i\}_{i\in\mathbb{N}}$ for $\{\triangle^2_i\}_{i\in\mathbb{N}}$ such that the edge shared between $\tilde{\triangle}^1_i$ and $\tilde{\triangle}^2_i$ tends to the geodesic joining $0$ and $\infty$. We denote by $\xi_i\in(0,\infty)$ the ideal vertex of $\tilde{\triangle}^2_i$ which is not  shared with $\tilde{\triangle}^1_i$. The fact that $P_i$ is not $\epsilon$-slender means that $\xi_i$ cannot converge to either $0$ or $\infty$. This in turn means that the geodesic joining $0$ and $\infty$ is a shared edge between $\tilde{\triangle}^1$ and $\tilde{\triangle}^2$ which is disjoint from $\Lambda$. This contradicts the fact that $\Lambda$ is a maximal chain-recurrent geodesic lamination because then the only possibility is that the geodesic joining $0$ and $\infty$ is the unique isolated simple bi-infinite geodesic lying on some punctured monogon (\cref{rmk:uniqueextension}); but this is impossible since such a geodesic bounds the same ideal triangle on both sides. The same observation that governs the remark at the end of  \cref{non-chain} tells us that \cref{thm:shrink} applies also to cases where one (or more) geodesic $\gamma$ in $|\Gamma_i|$ lies in a component of $S\setminus(|\Gamma_i|\setminus \gamma)$ which is a one-holed torus. We have therefore shown that the diameter of the faces $N_x([\Gamma_i])$ tend to $0$.

The compactness of the (Thurston-norm) unit sphere $\mathbf{S}_x$ ensures that the sequence of faces $\left\{N_x([\Gamma_i])\right\}_{i\in\mathbb{N}}$ has subsequences which converge to a vector $v\in\mathbf{S}_x$. We aim to show that this limit $v$ must be equal to $v_\Lambda$. To this end, let $\{\Lambda_i\}$ be a sequence of maximal chain-recurrent geodesic laminations respectively containing $|\Gamma_i|$, and consider the metrics $\mathrm{stretch}(x,\Lambda_i,t)\in\teich(S)$ (see \cref{rmk:uniqueextension}). By construction, the maximal ratio-maximising lamination $\mu(x,\mathrm{stretch}(x,\Lambda_i,t))$  is equal to $\Lambda_i$. \cref{laminationlimit} shows that the sequence $\{\Lambda_i\}$ converges, and its Hausdorff limit is $\Lambda$. Therefore, \cite[Theorem~8.4]{ThS} tells us that every convergent subsequence (which necessarily exist due to the fact that the Thurston metric sphere of radius $t$ around $x$ is compact) of $\mathrm{stretch}(x,\Lambda_i,t)$ necessarily converges to a metric $y\in\teich(S)$ such that $\mu(x,y)$ contains $\Lambda$. Since $\Lambda$ is a maximal chain-recurrent lamination, the only possibility is that $\mu(x,y)=\Lambda$, and the metric which satisfies this condition is $y=\mathrm{stretch}(x,\Lambda,t)$, and hence we have
\begin{align}
\forall t\geq0,\ 
\lim_{i\to\infty}
\mathrm{stretch}(x,\Lambda_i,t)
=\mathrm{stretch}(x,\Lambda,t).\label{eq:uniquelimit}
\end{align}

To finish the proof, we observe that since the diameter of $N_x([\Gamma_i])$ tends to $0$, we need only show that there is a convergent sequence of stretch vectors of the form $\left\{v_i\in N_x([\Gamma_i])\right\}_{i\in\mathbb{N}}$ limiting to $v_\Lambda$. \cref{eq:uniquelimit} implies that for each $k\in\mathbb{N}$, we may choose $i_k\in\mathbb{N}$ large enough so that the distance between $\mathrm{stretch}(x,\Lambda_{i_k},2^{-k})$ and $\mathrm{stretch}(x,\Lambda,2^{-k})$, with respect to the standard Euclidean metric for any a priori fixed analytically compatible global coordinate chart on $\teich(S)$, is less than $2^{-2k}$. Then, thanks to the analyticity of the stretch map with respect to the stretching time \cite[Corollary~4.2]{ThS}, we have:
\begin{align*}
x+2^{-k}v_{i_k}+O(2^{-2k})
&=\mathrm{stretch}(x,\Lambda_{i_k},2^{-k})\\
&=\mathrm{stretch}(x,\Lambda,2^{-k})+O(2^{-2k})\\
&=x+2^{-k}v_\Lambda+O(2^{-2k}),
\end{align*}
and hence 
\[
\lim_{i\to\infty} v_i=\lim_{k\to\infty} v_{i_k}=v_\Lambda.
\]
\end{proof}

\begin{corollary}[equivariant stretch vectors]
\label{equivariant stretch}
Given an arbitrary Thurston-norm isometry $f\colon T_x\teich(S)\to T_y\teich(S)$, let $[h]\in\mathrm{Mod}^*(S)$ (see \cref{first part}) denote its inducing extended mapping class. Let $\Lambda$ denote an arbitrary maximal chain-recurrent lamination of $S$. Then the stretch vector $v_{\Lambda}\in T_x\teich(S)$ (see \cref{rmk:uniqueextension}) satisfies the following equivariance property:
\begin{align*}
f(v_\Lambda)=v_{h(\Lambda)}\in T_y\teich(S).
\end{align*}
\end{corollary}

\begin{proof}
The chain-recurrence of $\Lambda$ means that it is arbitrarily closely approximated, in the Hausdorff topology, by a sequence of simple closed multicurves $\{\gamma_i\}_{i\in\mathbb{N}}$ (see \cref{lem:curveapprox}). Extend each multicurve $\gamma_i$ to a pants decomposition $\Gamma_i$. \cref{laminationlimit} shows that the sequence $\{\Gamma_i\}_{i\in\mathbb{N}}$ has a Hausdorff limit which is equal to $\Lambda$. Having satisfied the conditions for \cref{uniquelimit}, we know that any sequence of unit vectors
$
\left\{v_i\in N_x([\Gamma_i])\right\}_{i\in\mathbb{N}}
$
necessarily converges to $v_\Lambda$. Since $f$ is continuous, the sequence $\{f(v_i)\}$ converges to $f(v_\Lambda)$. However, \cref{thm:conemcg} tells us that there is a mapping class $h$ such that
\[
f(v_i)\in 
f(N_x([\Gamma_i]))
=N_y(h([\Gamma_i])).
\]
\cref{uniquelimit} ensures that the sequence $\{f(v_i)\}$ necessarily converges to $v_{h(\Lambda)}$. Therefore, $f(v_\Lambda)=v_{h(\Lambda)}$, as desired.
\end{proof}

Given an arbitrary simple closed geodesic $\gamma$ on $S$, take $\alpha_0$ to be any simple closed geodesic which has non-empty transverse intersection with $\gamma$, and define a sequence of simple closed curves $\{\alpha_m\}_{m\in\mathbb{Z}}$, where $\alpha_m$ is obtained from $\alpha_0$ by the $m$-times iterated Dehn twist around $\gamma$. For each $\alpha_m$, let $\Lambda_{+,m}$ and $\Lambda_{-,m}$ denote chain-recurrent complete geodesic laminations such that
\begin{itemize}
\item
$\Lambda_{\pm,m}$ contain geodesic leaves which bound the same $(1,1)$-cusped annulus $A_m\subset S$, which in turn contains $\alpha_m$;
\item
$\Lambda_{+,m}$ and $\Lambda_{-,m}$ have the same geodesic leaves, except on the interior of $A_m$, where the two bi-infinite geodesic leaves  in $\Lambda_{+,m}$ and $\Lambda_{-,m}$ which spiral toward $\alpha_m$ do so in opposite directions (see left-most and right-most pictures in \cref{fig:2}). 
\end{itemize}

\begin{proposition}[length extraction]
\label{lem:length}
Let $v_{\pm,m}$ denote the respective stretch vectors for the complete laminations $\Lambda_{+,m}$ and $\Lambda_{-,m}$. Then,
\begin{align*}
\ell_x(\gamma)
=
\lim_{m\to\pm\infty}
\frac{-\log\|v_{+,m}-v_{-,m}\|_{\mathrm{Th}}}{i(\gamma,\alpha_0)\cdot |m|},
\end{align*}
where $i(\gamma,\alpha_0)$ denotes the geometric intersection number between $\gamma$ and $\alpha_0$.
\end{proposition}

\begin{proof}
From \cite[Theorem~5.2]{ThS}, we know that the unit earthquake vector map $E_{(\cdot)}:\pml(S)\to T_x\teich(S)$ given by
\[
[\lambda]\mapsto \tfrac{1}{\ell_x(\lambda)}{\left.E_{\lambda}\right|_x},
\]
where $\lambda$ is an arbitrary measured lamination representing $[\lambda]$ and $\left.E_\lambda\right|_x$ is the earthquake vector at $x$ corresponding to $\lambda$, defines an embedding. Since $\pml(S)$ is compact and $\|E_{(\cdot)}\|_{\mathrm{Th}}$ is a continuous map, this means that there are constants $c_1,c_2>0$ such that 
\begin{align}
c_1\;\ell_x(\lambda)\leq\|\left.E_{\lambda}\right|_x\|_{\mathrm{Th}}\leq c_2\;\ell_x(\lambda).\label{eq:bounds}
\end{align}
Then, \cref{eq:asympdiff}, combined with \cref{eq:bounds} tells us that
\begin{align*}
c_1'\;\ell_x(\alpha_m)^2e^{-\ell_x(\alpha_m)}
\leq
\|v_{+,m}-v_{-,m}\|_{\mathrm{Th}}
\leq c_2'\;\ell_x(\alpha_m)^2e^{-\ell_x(\alpha_m)},
\end{align*}
for some positive constants $c_1', c_2'$.
Applying the logarithm to the above formula and using the fact that $\ell_x(\alpha_m)$ asymptotically grows as $|m|i(\gamma,\alpha_0)\ell_x(\gamma)$ as $m\to\pm\infty$ yields the desired result.
\end{proof}

\begin{theorem}[Geometric rigidity, \cref{second part}]
Let $f \colon T_x \teich(S) \to T_y \teich(S)$ be a linear isometry as given in \cref{main}, and let $h$ be a diffeomorphism representing the mapping class as in \cref{first part}. Then for every simple closed curve $\gamma\in\mathcal{S}$ on $S$, we have 
\[
\ell_x(\gamma)=\ell_y(h(\gamma)).
\]
\end{theorem}

\begin{proof}
Define $\alpha_m$, $\Lambda_{\pm,m}$ and $v_{\pm,m}$ as in the proof of \cref{lem:length}. 
Since $f$ is an isometry, we see that
\begin{align*}
\|v_{+,m}-v_{-,m}\|_{\mathrm{Th}}
=
\|f(v_{+,m})-f(v_{-,m})\|_{\mathrm{Th}}
=
\|w_{+,m}-w_{-,m}\|_{\mathrm{Th}},
\end{align*}
where $w_{\pm,m}\in T_y\teich(S)$ are stretch vectors with respect to the laminations $h(\Lambda_{\pm,m})$ on $y$ (\cref{equivariant stretch}). Then,
\begin{align*}
\ell_x(\gamma)
&=\lim_{m\to\infty}
\frac{-\log\|v_{+,m}-v_{-,m}\|_{\mathrm{Th}}}{i(\gamma,\alpha_0)\cdot |m|}\\
&=\lim_{m\to\infty}
\frac{-\log\|w_{+,m}-w_{-,m}\|_{\mathrm{Th}}}{i(h(\gamma),h(\alpha_0))\cdot |m|}
=\ell_y(h(\gamma)),
\end{align*}
as desired.
\end{proof}

\subsection{Infinitesimal rigidity}

We now prove \cref{main}: the infinitesimal rigidity of the Thurston metric.

\begin{proof}[Proof of \cref{main}]
By \cref{first part}, for an isometry $f$ in \cref{main}, there is a diffeomorphism $h$ representing some mapping class group which induces the same homeomorphism on $\pml(S)$ as $f_{y,x}:=\iota_x^{-1}\circ f^*\circ \iota_y$. By \cref{second part}, for any simple closed curve $\gamma$ on $S$, the length of $h(\gamma)$ with respect to the hyperbolic metric at $y$ coincides with that of $\gamma$ at $x$. Since the length functions of the simple closed curves  define an embedding of $\teich(S)$ to $\reals^\mathcal{S}$, we see that $h$ takes $(S,x)$ to $(S,y)$. Since $dh$ takes the cotangent vector corresponding to $d_y\log \len(\gamma)$ to $d_x \log \len(h(\gamma))$, we see that 
\[
dh\equiv\iota_y \circ h \circ \iota_x^{-1}\equiv f.
\]
\end{proof}

There are various corollaries to \cref{main}, including the local rigidity (\cref{thm:local} and global rigidity (\cref{thm:walsh}) results stated in the introduction. For example, we can strengthen \cref{equivariant stretch} as follows:

\begin{corollary}\label{fullstretchequivariance}
Given a Thurston-norm isometry $f\colon T_x\teich(S)\to T_y\teich(S)$, consider its corresponding extended mapping class $[h]\in\mathrm{Mod}^*(S)$ (see \cref{first part}) and let $\Lambda$ denote an arbitrary complete lamination of $S$. The stretch vector $v_{\Lambda}\in T_x\teich(S)$ satisfies the following equivariance property
\begin{align*}
f(v_\Lambda)=v_{h(\Lambda)}\in T_y\teich(S).
\end{align*}
\end{corollary}
To clarify, we have removed every assumption on $\Lambda$ apart from completeness.\medskip
 
 \cref{fullstretchequivariance} implies that the notion of tangential adherence introduced in \cref{tangentially adherent} is a linear invariant. This then strengthens \cref{set of measures} as a linearly invariant characterisation of faces of the form $F_{|\lambda|}\subset\mathbf{S}^*$.

 \subsection{Further questions}
Despite our increased understanding of the infinitesimal and global structure of Thurston's Finsler metric on Teichm\"uller space, there is still much that remains mysterious. We list some of the questions we have under consideration:
  
\begin{question}
\label{maximal ratio-maximising}
The notion of \emph{maximal ratio-maximising chain-recurrent lamination} \cite[Theorem~8.2]{ThS} and its role in Thurston's theory of stretch maps as being (intuitively speaking) ``extremal'' stretch maps \cite[Theorem~8.4]{ThS} suggests that stretch vectors with respect to maximal  chain-recurrent laminations should constitute extreme points of the faces of $\mathbf{S}_x$. This is unproven except when $S$ is a one-cusped torus or a four-cusped sphere, and/or possibly as a statement which is ``generically'' true (and even this is not without subtleties). Is this, in fact, true? 
\end{question}

\begin{question}

What are the dimensions, codimensions, face-dimensions and adherence-dimensions of arbitrary faces in $\mathbf{S}_x$ and $\mathbf{S}_x^*$?
\end{question}

\begin{question}
\label{q:covariant}
\cref{inclusion} can be paraphrased in a category-theoretic language as saying that $F_{(\cdot)}$ induces a covariant functor between two full subcategories $\left|\mathcal{ML}\right|$ and $\mathcal{F}$ of the category of sets where
\begin{itemize}
\item
the objects of $\left|\mathcal{ML}\right|$ are given by the supporting geodesic laminations $|\lambda|$ of measured laminations $\lambda\in\mathcal{ML}(S)$, and
\item
the objects of $\mathcal{F}$ are given by the faces of $\mathbf{S}^*_x$.
\end{itemize}
Moreover, \cref{adinclusion} tells us that there is a natural transformation $\eta$ from the functor $F_{(\cdot)}$ to itself that takes each object $|\lambda|\in\mathrm{ob}(\left|\mathcal{ML}\right|)$ to the morphism $\eta_{|\lambda|}$ of inclusion, \ie $F_{|\lambda|}\subset F_{\widehat{|\lambda|}}$. Is this structure part of some general phenomenon? 
\end{question}

\begin{question}
On a related note to \cref{q:covariant}, it seems highly plausible that the map $N_x(\cdot)$ taking a measured lamination $\lambda\in\mathcal{ML}(S)$ to the corresponding face should not require the information of the measure on $\lambda$ and only uses $x$ as an auxiliary metric. Philosophically speaking, then, this map should induce a contravariant functor $N_{(\cdot)}$ from $\left|\mathcal{ML}\right|$ to the full subcategory $\mathcal{N}$ of the category of sets whose objects are given by the faces of $\mathbf{S}_x$. The role of such a purported contravariant functor $N_{(\cdot)}$ is reminiscent of the correspondence between radical  ideals in coefficient rings and algebraic sets appearing in algebraic geometry. Assuming that such a conjectural theory for $N_{(\cdot)}$ holds, is there perhaps a comparable theory of convex ``varieties'' for convex geometry that applies beyond the confines of our work? And is it dual in a structured sense to what we see in \cref{q:covariant}?

\end{question}

\end{document}